\documentclass[10pt, a4paper, reqno, oneside]{amsart}
\usepackage[T1]{fontenc}
\usepackage[utf8]{inputenc}
\usepackage[italian, english]{babel}
\usepackage{geometry}
\usepackage{color}
\usepackage[hidelinks]{hyperref}
\usepackage[timezonesep={\ UTC}, showseconds=false]{datetime2}

\usepackage[autostyle]{csquotes}
\usepackage[backend=biber, sorting=nyt, doi=false, url=false, isbn=false,
style=alphabetic-verb]{biblatex}
\DefineBibliographyExtras{english}{}

\DeclareFieldFormat{postnote}{#1}
\DeclareFieldFormat{multipostnote}{#1}

\usepackage{xurl}

\usepackage{hyphenat}
\usepackage{csquotes}
\usepackage{comment}
\usepackage{enumitem}
\makeatletter
\newcommand{\mylabel}[2]{#2\def\@currentlabel{#2}\label{#1}}
\makeatother

\usepackage{upgreek}

\usepackage{amsfonts}
\usepackage{amsmath}
\usepackage{amssymb}
\usepackage{amsthm}
\usepackage{mathrsfs}
\usepackage{stmaryrd}
\usepackage{bm}
\usepackage{mathtools}
\usepackage{braket}
\usepackage{dsfont}

\usepackage{graphicx}
\usepackage{caption}
\usepackage{float}
\usepackage{multicol}
\usepackage{booktabs}
\usepackage{array}

\usepackage{tikz-cd}
\usepackage[all,cmtip]{xy}

\newtheorem{theorem}{Theorem}[section]

\newtheorem{lemma}[theorem]{Lemma}
\newtheorem{proposition}[theorem]{Proposition}
\newtheorem{corollary}[theorem]{Corollary}

\newtheorem*{theorem*}{Theorem}
\newtheorem*{statement*}{Statement}
\newtheorem*{lemma*}{Lemma}
\newtheorem*{proposition*}{Proposition}
\newtheorem*{corollary*}{Corollary}

\theoremstyle{definition}

\newtheorem{definition}[theorem]{Definition}

\theoremstyle{remark}
\newtheorem{remark}[theorem]{Remark}
\newtheorem{example}[theorem]{Example}

\newcommand{\numberset}{\mathbb}
\newcommand{\R}{\numberset{R}}
\newcommand{\C}{\numberset{C}}
\newcommand{\Q}{\numberset{Q}}
\newcommand{\Z}{\numberset{Z}}
\newcommand{\N}{\numberset{N}}
\newcommand{\F}{\numberset{F}}
\newcommand{\Zp}{\Z_p}

\newcommand{\Qp}{\Q_p}

\newcommand{\Qbar}{\bar{\Q}}

\newcommand{\Fbar}{\bar{F}}

\DeclareFontFamily{U}{wncy}{}
\DeclareFontShape{U}{wncy}{m}{n}{<->wncyr10}{}
\DeclareSymbolFont{mcy}{U}{wncy}{m}{n}
\DeclareMathSymbol{\sha}{\mathord}{mcy}{"58}
\newcommand{\Sha}{\sha}
\newcommand{\inj}{\hookrightarrow}
\newcommand{\surj}{\twoheadrightarrow}

\newcommand{\iso}{\xrightarrow{\sim}}
\newcommand{\longinj}{\lhook\joinrel\longrightarrow}

\renewcommand{\injlim}{\varinjlim}
\renewcommand{\projlim}{\varprojlim}

\DeclareMathOperator{\Jac}{Jac}

\DeclareMathOperator{\Hom}{Hom}

\DeclareMathOperator{\End}{End}
\DeclareMathOperator{\Aut}{Aut}
\DeclareMathOperator{\M}{M}
\DeclareMathOperator{\GL}{GL}
\DeclareMathOperator{\SL}{SL}

\newcommand{\matid}{\mathds{1}}

\DeclareMathOperator{\Tr}{Tr}
\DeclareMathOperator{\rk}{rk}
\newcommand{\Frob}{\mathrm{Frob}}

\DeclareMathOperator{\Ad}{Ad}
\newcommand{\adj}{\mathrm{adj}}
\newcommand{\tors}{\mathrm{tors}}
\newcommand{\ab}{\mathrm{ab}}

\newcommand{\mfrak}{\mathfrak{m}}
\newcommand{\p}{\mathfrak{p}}
\renewcommand{\P}{\mathfrak{P}}
\newcommand{\q}{\mathfrak{q}}

\DeclareMathOperator{\Gal}{Gal}

\newcommand{\cOP}{\mathcal{O}_{\P}}

\DeclareMathOperator{\nr}{nr}
\DeclareMathOperator{\Div}{Div}
\DeclareMathOperator{\CM}{CM}
\DeclareMathOperator{\lcm}{lcm}
\newcommand{\cfrak}{\mathfrak{c}}
\newcommand{\cond}{\cfrak}
\newcommand{\ellfr}{\mathfrak{l}}

\newcommand{\ev}{\mathrm{ev}}
\newcommand{\B}{B}
\newcommand{\Bhat}{\widehat{\B}}
\newcommand{\charpoly}{\mathrm{char}}
\DeclareMathOperator{\Ram}{Ram}

\newcommand{\rec}{\mathrm{rec}}

\newcommand{\cohomology}{H}

\newcommand{\hone}{\cohomology^1}

\newcommand{\honeur}{\hone_\ur}
\DeclareMathOperator{\Sel}{Sel}

\DeclareMathOperator{\Cor}{cor}
\DeclareMathOperator{\Res}{res}
\newcommand{\res}{\Res}
\newcommand{\cores}{\Cor}

\newcommand{\ur}{\mathrm{ur}}
\newcommand{\ord}{\mathrm{ord}}

\newcommand{\cont}{\mathrm{cont}}

\newcommand{\different}{\mathscr{D}}

\newcommand{\admissible}{\mathscr{S}}
\newcommand{\allowed}{\mathscr{P}}
\newcommand{\allowedprimes}{\mathscr{P}^{\mathrm{prime}}}%{\widetilde{\mathscr{P}}}

\newcommand{\nfrak}{\mathfrak{n}}

\newcommand{\cbf}{\mathbf{c}}
\newcommand{\dbf}{\mathbf{d}}

\DeclareMathOperator{\Selm}{Sel}
\DeclareMathOperator{\img}{img}

\newcommand{\cA}{\mathcal{A}}

\newcommand{\cN}{\mathcal{N}}
\newcommand{\cG}{\mathcal{G}}
\newcommand{\cO}{\mathcal{O}}

\newcommand{\fp}{\mathfrak{P}}
\newcommand{\fq}{\mathfrak{q}}

\newcommand{\fP}{\mathfrak{P}}

\makeatletter
\newcommand{\leqnomode}{\tagsleft@true\let\veqno\@@leqno}
\newcommand{\reqnomode}{\tagsleft@false\let\veqno\@@eqno}
\makeatother

\newcommand{\cOE}{\cO_E}
\newcommand{\cOEP}{\cO_{E,\P}}
\newcommand{\cOF}{\cO_F}
\newcommand{\cOK}{\cO_K}
\newcommand{\cOhat}{\widehat{\cO}}
\newcommand{\cOhatF}{\cOhat_F}

\newcommand{\Khat}{\widehat{K}}
\newcommand{\Fhat}{\widehat{F}}
\author{Luca Mastella}
\author{Ahmed Matar}
\author{Francesco Zerman}

\address{UniDistance Suisse\\ Schinerstrasse 18, 3900, Brig, Switzerland}
\email{luca.mastella@unidistance.ch}
\email{francesco.zerman@unidistance.ch}

\address{Department of Mathematics, University of Bahrain, P.O. Box 32038, Sukhair, Bahrain}
\email{amatar@uob.edu.bh}

\title[$\sha(A/K){[\P^\infty]}=0$ and anticycl. Iw. theory of $\GL_2$-abelian varieties]{Vanishing of $\sha(A/K){[\P^\infty]}$ and its consequences for the anticyclotomic Iwasawa theory of $\GL_2$-abelian varieties}


\begin{document}

\maketitle

\begin{center}
    {\itshape
    Dedicated to the memory of {Jan Nekov{\'a}{\v{r}}}}
\end{center}

\begin{abstract}
In this article we generalise a classical Kolyvagin's result on the vanishing of the $p$-part of the Shafarevich--Tate group of an elliptic curve (defined over $\Q$) over an imaginary quadratic field $K$ to modular abelian varieties of $\GL_2$-type (defined over a totally real number field $F$) and a CM field $K/F$. Combining this result with the abstract Iwasawa-theoretical argument of \cite{matar-nekovar:kolyvagin}, we show that a similar vanishing holds over the layers of a suitably defined anticyclotomic multi-$\Z_p$-extension of $K$.
\end{abstract}

\section{Introduction}

This article is a follow up of the paper \cite{matar-nekovar:kolyvagin} of the second named author together with {Jan Nekov{\'a}{\v{r}}}, generalising their results to modular $\GL_2$-abelian varieties defined over a totally real field $F$. These are $F$-simple quotients $A$ of the Jacobian of certain Shimura curves $N_H^\ast$. It follows that $\cOE := \End_F(A)$ is the ring of integers of a (totally real) number field $E$ with $[E:\Q] =\dim(A) =:g$. For a precise description of these objects see Section \ref{sec:modular-av-and-heegner}.

In particular, \cite{matar-nekovar:kolyvagin} discusses an improvement of the classical Kolyvagin's result about the vanishing of the $p$-primary part of the Shafarevich--Tate group of an elliptic curve $A$ defined over $\Q$, or more precisely a special case of this proven by Gross in \cite{gross:kolyvagin}. 
Let us recall their result. 

Let $K/\Q$ be an imaginary quadratic field of discriminant $D_K \ne -3, -4$ satisfying the \emph{Heegner hypothesis} (i.e., such that all the primes dividing the conductor $N$ of $A$ split in $K$) and let $p$ be a prime such that $p \nmid 2D_K$. By the Heegner hypothesis,  we can choose an ideal $\cN$ of $\cO_K$ such that $\cO_K/\cN \cong \Z/N\Z$. Therefore, the natural projection of complex tori
\[\C/\cO_K \to \C/\cN^{-1}\]
is a cyclic $N$-isogeny, which corresponds to a point $x_1$ on the modular curve $X_0(N)$. The theory of complex multiplication shows that $x_1$ is rational over $K[1]$, the Hilbert class field of $K$. Fix a modular parametrization $\pi \colon X_0(N) \to A$ which maps the cusp $\infty$ of $X_0(N)$ to the origin of $A$ and let $y_1=\pi(x_1) \in A(K[1])$. Finally, define the point $y_K:=\Tr_{K[1]/K}(y_1) \in A(K)$. 
\begin{theorem}[{\cite[Theorem 6.7, case $K \ne \Q(i), \Q(\sqrt{-3})$]{matar-nekovar:kolyvagin}}]\label{theorem:gross_generalisation}
  Suppose that $A[p]$ is a simple $\F_p[G_\Q]$-module. If $y_K \notin p A(K)$, then $A(K) \otimes \Zp = \Zp y_K \cong \Zp$ and $\Sha(A/K)[p^{\infty}]=\{0\}$
\end{theorem}
This result is obtained via a careful study of the assumptions used in the proof of \cite[Proposition 2.1]{gross:kolyvagin} (denoted ($C1$)-($C6$) in {\cite[Section 6]{matar-nekovar:kolyvagin}}): they prove that if $D_K \ne -3, -4$ these are all satisfied as long as $A[p]$ is an absolutely simple $\F_p[G_K]$-module and that this latter condition is (under their running hypothesis) equivalent to the (standard) simplicity of $A[p]$ as $\F_p[G_\Q]$-module.

Moreover, using the norm relations among the Heegner points (see \cite[Section 3.1, Proposition 1]{perrin-riou:heegner-points}), they extend the vanishing of the Shafarevich--Tate group along the layers $K_n$ of the anticyclotomic extension $K_\infty$ of $K$, combining Theorem \ref{theorem:gross_generalisation} with an abstract Iwasawa theoretical result (i.e., not involving any Euler system argument).
Indeed, if $a_p = p + 1 - \# \tilde{A}_p(\F_p)$, $\tilde{A}_p$ is the reduction of $A$ at $p$, $c_{\mathrm{Tam}}(A/\Q)$ is the Tamagawa number of $A$ and  
\[
    \varepsilon(p) = \begin{cases}
    1 \qquad &\text{if $p$ is split in $K/\Q$;}\\
    0  &\text{if $p$ is ramified in $K/\Q$;}\\
    -1 &\text{if $p$ is inert in $K/\Q$.}
    \end{cases}\]
they obtain the following result.

\begin{theorem}[{\cite[Theorem 6.9]{matar-nekovar:kolyvagin}}]\label{theorem:iwasawa-concrete-mn19}
Suppose that $A[p]$ is a simple $\F_p[G_\Q]$-module and that $p \nmid N\cdot  a_p \cdot  (a_p -1) \cdot (a_p - \varepsilon(p)) \cdot c_{\mathrm{Tam}}(A/\Q)$. 
If $y_K \notin p A(K)$, then, for any intermediate field $K \subseteq L \subseteq K_\infty$, $\Sha(A/L)[p^\infty] = 0$ and the Pontryagin dual of $A(L) \otimes \Q_p/\Z_p$ is a free module of rank one over $\Z_p \llbracket \Gal(L/K)\rrbracket$. Moreover for any integer $n \ge 0$, $\rk_{\Z} E(K_n) = p^n$, $\Sha(A/K_n)[p^\infty] = 0$ and $A(K_n) \otimes \Z_p$ is generated over $\Z_p[\Gal(K_n/K)]$ by the traces of Heegner points of $p$-power conductor.
\end{theorem}

The first main result of the present article is a generalisation of Theorem \ref{theorem:gross_generalisation} to the case of a modular $\GL_2$-abelian variety $A$ as in the first paragraph of this introduction. Therefore, let $K$ be a CM extension of $F$, satisfying Assumption \eqref{ass:heegner-hypothesis} of Section \ref{sec:heegner-points}. In particular, fixed an auxiliary CM point $x$ of $N_H^\ast$ of conductor $\cond(x)$, we define (mildly generalising \cite[Definition 4.4]{nekovar:euler-system-method}) for any ideal $\nfrak$ of $F$ which is the product of certain primes $\ellfr$ of $F$ which are unramified in $K/F$, a class of CM points $x(\nfrak) \in N_H^\ast(K^\ab)$ (see Definition \ref{def:allowed-primes}). Denote by $K(x)$ the field of definition of $x$ and by $K(x(\nfrak))$ the field of definition of $x(\nfrak)$; it turns out that  $K(x) \subseteq K[\cond(x)]$ and $K(x(\nfrak)) \subseteq K[\cond(x)\nfrak]$, denoting by $K[\cond]$ the ring class field of conductor $\cond$ for any ideal $\cond\subseteq F$. 

The images of the $x(\nfrak)$'s via the modular parametrisation 
\[
N_H^\ast \inj \Jac(N_H^\ast) \surj A,
\]
allow us to define a class of Heegner points $y(\nfrak) \in A(K^\ab)$ (see Definition \ref{def:heegner-points-av}), with whom we construct an Euler system. Define $y_K = \Tr_{K(x)/K}(y_1)$. In Section \ref{vanishingresult_section}, with an argument inspired by \cite{BD}, we exploit this Euler system in order to prove a vanishing theorem for $\sha(A/K)[\P^\infty]$. 

Fix a  polarization $\gamma$ of $A$, a rational prime $p \nmid 2[\cO_K^\times: \cO_F^\times]\deg(\gamma)$, a prime $\P$ of $E$ above $p$ and write $\cOP := \cOEP$.  Let $t$ be the order of $\P$ in the class group of $E$. 
The following theorem follows from Theorem \ref{main_theorem} for $M=t$, making a mild assumption \eqref{ass:image-residual} on the image of the residual representation $A[\P]$, for which we refer to Section \ref{vanishingresult_section}, if $t>1$.
\begin{theorem}\label{th:intro-main-1}
Under the above assumptions, suppose moreover that $y_K \notin \P A(K)$ and that $A[\P]$ is an absolutely simple $(\cOE/\P)[G_K]$-module. Then,  $A(K) \otimes_{\cOE} \cOP = \cOP y_K \cong \cOP$ and $\Sha(A/K)[\fP^{\infty}]=\{0\}$.
\end{theorem}

This result might be seen as a refinement of \cite[Theorem 3.2]{nekovar:euler-system-method} or \cite[Theorem A]{howard:iwasawa-gl2-abelian-varieties}. In particular, the latter exploits a Kolyvagin system argument in order to prove a more general theorem about the structure of the Shafarevich--Tate group even when $y_K \in \P A(K)$, but under some stricter assumptions (he requires that $T_\P A$ satisfies a \emph{big image} assumption, which in the case of $\GL_2$-abelian varieties is known to hold for infinitely many primes $p$ only assuming that all the $\Qbar$-endomorphisms of $A$ are defined over $F$, see \cite{ribet:galois-action-av} and \cite[Theorem 1.4]{lombardo:image}). Our approach, in the spirit of Theorem \ref{theorem:gross_generalisation}, allows us to work under a minimal set of assumptions, which are stated explicitly throughout the article. Moreover, our definition of Heegner points, that follows \cite{nekovar:euler-system-method}, has the advantage of allowing parametrisations by a notably wider class of Shimura curves than those considered in \cite{howard:iwasawa-gl2-abelian-varieties}.

Moreover, despite being inspired by \cite{BD} and \cite{gross:kolyvagin}, the proof of Theorem \ref{th:intro-main-1} faces some technical difficulties. The first comes from the fact that a \emph{geometric} formula for the action of the complex conjugation on Heegner points is not available in this context. However, since $p$ is odd, we may write $y_K = y_K^+ + y_K^-$ with $y_K^{\pm}$ in the $\pm$-eigenspace with respect to the action of the complex conjugation: since $y_K \notin \P A(K)$, for a choice of $\epsilon \in \{\pm\}$ we have that also $y_K^\epsilon \notin \P A(K)$; thus we replace $y_K$ with $y_K^\epsilon$ in the argument of \cite{BD}. Another more serious problem is that the pairing considered in \cite[Proposition 9.3]{gross:kolyvagin} is not anymore perfect in our setting: the maps $\phi_S$ and $\psi_S$, that we define in Section \ref{sec:Kolyvaigin-pairing} for any $\cOP$-submodule $S \subseteq \hone(F, A[\P^M])$, are in general just injections (see \cite[Remark 6.5]{besser:finiteness-sha} for an explicit example of this phenomenon in a similar setting).  In order to overcome this problem we need to consider the $\cOP$-span $X_S$ of $\img \phi_S$, but this leads to a series of variations in the method. For instance it more complicated to study of how $X_S$ behave when we vary $S$ (see Section \ref{compatibility_sebsection}), which is a crucial step in the proof of Proposition \ref{generating_Sel_proposition}. Another technical difference between our framework and the one of \cite{BD} (or \cite{gross:kolyvagin}), which led to the computations of Section \ref{directproduct_subsection}, is that, if $S = S_1 \oplus S_2$ it's not anymore clear whether $L_{S_1}$ and $L_{S_2}$ are linearly disjoint over $L=K(A[\P^M])$, denoting by $L_{S_i}$ the field attached to $S_i$, for $i=1,2$. However, the computations of Section \ref{directproduct_subsection} show that this is actually not needed in order to apply our version of the Kolyvagin's method.

Then, in Section \ref{representation_section}, via a detailed study of the image of the residual representation $A[\P]$, we show that in some cases the assumption that $A[\P]$ is absolutely simple as $(\cOE/\P)[G_K]$-module is equivalent to being simple as $(\cOE/\P)[G_F]$-module (Proposition  \ref{theorem:assumption-reduction}). Combining this fact with Theorem \ref{th:intro-main-1}, we get the following corollary. 
\begin{corollary}
    In the setting of Theorem \ref{th:intro-main-1}, assume moreover that $A$ has good $\P$-ordinary reduction at a prime $\p$ of $F$ above $p$, that $F_\p(\mu_p)/F_\p$ is totally ramified and $p >3$. Furthermore, suppose that $y_K \notin \P A(K)$ and that $A[\P]$ is an  simple $(\cOE/\P)[G_F]$-module. Therefore, it follows that $A(K) \otimes_{\cOE} \cOP = \cOP y_K \cong \cOP$ and $\Sha(A/K)[\fP^{\infty}]=\{0\}$.
\end{corollary}
In particular, the definition of good $\P$-ordinary reduction is a weaker version of the usual good ordinary reduction for $A$: see Section \ref{sec:ordinary} for a discussion of its consequences.

Finally, in Section \ref{Iwasawatheory_section} we discuss the consequences of Theorem \ref{th:intro-main-1} for the anticyclotomic Iwasawa theory of $A$, applying the abstract Iwasawa theoretical results of \cite[Section 3]{matar-nekovar:kolyvagin}. We remark that here the norm relations among the Heegner points $y(\nfrak)$ (see Proposition \ref{prop:norm-relations-heegner}) play a crucial role in order to construct a universal norm whose basic term is a $\cOP$-multiple of $y_K$. 

Therefore, fixed a prime $\p$ (of the kind that can be used as conductor of an Heegner point) of $F$ above $p$ and assuming $p \nmid [K(x):K]$, we construct a multi-$\Z_p$-subextension $K_\infty$ of $K(x(\p^\infty)) = \bigcup_{n\ge 0} K(x(\p^n))$, and therefore pro-dihedral over $F$, which we call the \emph{anticyclotomic $\Z_p^{f_\p}$-extension relative to $x$}, $f_\p$ denoting the inertia degree of $\p$ in $F/\Q$. Moreover,  for any integer $n\ge0$ let $K_n$ be the unique subfield of $K_\infty$ such that $\Gal(K_n/K)=(\Z/p^n\Z)^{f_\p}$.

We have the following result, which combine Theorem \ref{theorem:iwasawa-th-main}, Remark \ref{rk:iwasawa-layers} and Proposition \ref{theorem:assumption-reduction}. Here, $a_\p = \Tr(\Frob_\p | T_\lambda A)$ for any $\lambda \nmid p$ prime of $E$. Moreover, for any $\ellfr$ prime of $F$ we define
\begin{equation*}
    \varepsilon(\ellfr)=\begin{cases}
        1 &\text{if $\ellfr$ splits in $K$;}\\
        -1 &\text{if $\ellfr$ is inert in $K$;}\\
        0 &\text{if $\ellfr$ is ramified in $K$.}
    \end{cases}
\end{equation*}
\begin{theorem}\label{theorem:iwasawa-intro}
    In the setting of Theorem \ref{th:intro-main-1}, suppose moreover that $p \nmid 3 [K(x):K]$, that $A$ has good ordinary reduction at any $\q \mid p$ of $F$, that $a_\p \not\equiv 1, \epsilon(\q) \bmod{\P'}$, for any prime $\q \mid p$ of $F$ and $\P' \mid p$ of $E$ and that that $F_\p(\mu_p)/F_\p$ is totally ramified. Thus if  $y_K \notin \P A(K)$ and $A[\P]$ is a simple $(\cOE/\P)[G_F]$-module, for any intermediate field $K \subseteq L \subseteq K_\infty$ we have that
    \[
    \sha(A/L)[\P^\infty] = 0, \quad A(L)\otimes_{\cOE} E_\P/\cOP = \Selm_{\P^\infty}(A/L),
    \]
    and both $S_\p(A/L)$ and the Pontryagin dual of $\Selm_{\P^\infty}(A/L)$ are free modules of rank one over $\Lambda_L = \cOP \llbracket \Gal(L/K)\rrbracket$. Moreover for any integer $n\ge 0$, $\rk_{\cOE} A(K_n) = p^{nf_\p}$ and 
    \[A(K_n) \otimes_{\cOE}\cOP = S_\P(A/L)\]
    is generated over $\cOP[\Gal(K_n/K)]$ by the traces of the Heegner points $\{y(\p^m)\}_{m\ge0}$.
\end{theorem}

\subsection{Outline of the paper}
In Section \ref{sec:av}, we recall some background material about (simple) $\GL_2$-abelian varieties defined over $F$.
In Section \ref{sec:modular-av-and-heegner} we define the Shimura curves $N_H^\ast$ parametrising the abelian varieties $A$ that we consider in the paper. Moreover, we define a class of CM points on $N_H^\ast$ which give rise to Heegner points on $A$ and we study their properties. In particular, we prove for these points some norm relations (Proposition \ref{prop:norm-relations-heegner}), of a similar shape of those of \cite{perrin-riou:heegner-points} for Heegner points on elliptic curves.
Section \ref{vanishingresult_section} is the core of the paper: we get here the vanishing result for the Shafarevich--Tate group.
Section \ref{representation_section} is of independent interest: we study there the image of the residual representation $A[\P]$ under the $\P$-ordinary assumption at a prime $\p$ of $F$ above $p$. In particular, we find some conditions under which the assumptions \eqref{ass:abs-G-K-irr} and \eqref{ass:G_F-irr} are equivalent.
In Section \ref{Iwasawatheory_section} we obtain some consequences of the results of the previous sections for the anticyclotomic Iwasawa theory of $A$.
 
\subsection{Notation and assumptions}

For any perfect field $L$, we use th notation $G_L:=\Gal(\bar{L}/L)$ for its absolute Galois group. If, moreover, $L$ is a number field and $v$ is a prime of $L$, we denote by $\cO_L$ the ring of integers of $L$, by $L_v$ the completion of $L$ at $v$ and by $\cO_{L,v}$ the valuation ring of $L_v$. Let also $L_v^\ur$ be the maximal unramified extension of $L_v$. We denote by $Nv$ the norm of $v$ in the extension $L/\Q$.

We write $\Frob_v$ for the set of elements of $G_L$ that lie in some decomposition group at $v$ and that reduce to the Frobenius automorphism modulo the corresponding inertia at $v$. Note that, if $T$ is a $G_L$-module unramified at $v$, the characteristic polynomial for the action of $\Frob_v$ on $T$ is independent of the choice of a representative.

If $A$ is an abelian variety, we denote by $A^\vee$ the dual abelian variety and if $\alpha \colon A \to A'$ is an isogeny of abelian varieties, we let $\alpha^\vee \colon (A')^\vee \to A^\vee$ be the dual isogeny.

\subsection{Acknowledgments}

The first and third author, while writing this article, were supported by the \emph{ERC Consolidator Grant ShimBSD: Shimura varieties and the BSD conjecture}.

\section{Abelian varieties and arithmetic groups}\label{sec:av}

\subsection{Abelian varieties of $\GL_2$-type}

Let $A$ be an abelian variety of dimension $g\ge 1$ defined over a totally real number field $F$ of degree $d$. Consider the following assumption, which we assume for the rest of this section:
\leqnomode
\begin{equation}
    \hspace{31pt}\text{$\End_F(A) = \cOE$ is the ring of integers of a number field $E$, with $[E:\Q]=g$}; \tag{$\GL_2$}\label{ass:GL-2}
\end{equation}
\reqnomode
In particular, \eqref{ass:GL-2} implies that $A$ is a simple. This family of abelian varieties, first introduced in \cite{ribet:galois-action-av} after a suggestion of John Tate and called \emph{of $\GL_2$-type}, will be the main object of our study. The name descends from the fact that the Tate module of $A$ can be decomposed into rank-2 submodules, as we now recall.

Fix a rational prime $p>2$. Denote by $T_p A = \projlim_{M} A[p^M]$ the $p$-adic Tate module of $A$ and by $A[p^\infty] = \injlim_{M} A[p^M]$ its $p$-divisible group, both endowed with a $\Z_p$-linear action of $G_F$. 

\begin{definition}
    For any ideal $\mathfrak{a}$ of $\cOE$, define $A[\mathfrak{a}] := \set{P \in A(\Qbar) : \text{$\alpha P = 0$ for any $\alpha \in \mathfrak{a}$}}$.
\end{definition}
For any $M\in\Z_{\ge 1}$, the decomposition
\[
\cOE/p^M\cOE\cong\prod_{\P\mid p}\cOE/\P^{Mr_\P },
\]
where $\P$ runs through the primes of $E$ above $p$ and $r_\P$ denotes the ramification index of $\P$ over $p$, induces the decomposition
\[
A[p^M] = \bigoplus_{\P \mid p} A[\P^{Mr_\P}].
\]
Taking direct and inverse limits over $M$, we obtain the decompositions
\begin{equation*}
	A[p^\infty]= \bigoplus_{\P\mid p} A[\P^{\infty}]\quad\text{and}\quad T_p A = \bigoplus_{\P \mid p} T_\P A.
\end{equation*}
For any $\P \mid p$, the $\P$-adic Tate module $T_\P A$ is a free $\cOEP$-module of rank $2$ (see \cite[Proposition 2.2.1]{ribet:galois-action-av}) and $A[\P^M]$ is free of rank $2$ over $\cOE/\P^M$ for any $M \in \Z_{\ge 1}$. 
Moreover, since $\cO_E=\End_F(A)$ is made by endomorphisms defined over $F$, all these modules are endowed with an $\cO_E$-linear $G_F$-action. We denote the Galois representations attached to $T_\P A$ and $A[\P^M]$ respectively as
\[
\rho_\P \colon G_F \to \Aut_{\cO_{E,\P}}(T_\P A) \cong \GL_2(\cOEP),  \quad \rho_{\P, M} \colon G_F \to \Aut_{\cO_E/\P^M}(A[\P^M]) \cong \GL_2(\cOE/\P^M).
\]
By \cite[Proposition 2.1.2]{ribet:galois-action-av}, for any finite prime $\ellfr \nmid p$ of $F$ of good reduction (so that $T_\P A$ is unramified as $G_F$-representation, by the Néron--Ogg--Shafarevich criterion for abelian varieties \cite[Theorem 1]{serre-tate:good-red-av}), the characteristic polynomial
\[
\charpoly(\Frob_\ellfr | T_\P A) = \det(X\matid - \rho_\P(\Frob_\ellfr))
\]
of $\Frob_\ellfr$ on $T_\P A$ is independent of both $p$ and $\P$ (as long as $\ellfr \nmid p$)  and has coefficeints in $\cO_E$.

\subsection{Selmer and Shafarevich--Tate groups}\label{sec:selmer-sha}
We introduce here some relevant arithmetic groups attached to an abelian variety $A$ as above, whose study is the aim of this paper.
\begin{definition}
The \textit{Shafarevich--Tate group} of $A$ over a field extension $L$ of $F$ is 
\[
\sha(A/L) = \ker\biggl(\hone(L, A) \to \prod_v \hone(L_v, A)\biggr), 
\]
the product being taken over all places $v$ of $L$.
\end{definition} 
From now on, until the end of Section \ref{sec:av}, we fix a prime $\P$ of $E$ over $p$ and write $\cOP:= \cOEP$ for brevity. We fix also an integer $M \ge 1$ and consider the following assumption, that we assume to be in force in this paragraph:
\leqnomode
\begin{equation}
    \text{$M$ is divisible by the order of $\P$ in the class group of $E$}.\tag{$\P^M$-princ}\label{ass:P-M-princ}
\end{equation}
\reqnomode
It follows that $\P^M$ is principal, say generated by $\pi_M$. 

\begin{definition}
   The $\P^M$-Selmer group of $A$ over a finite extension $L/F$ is  
   \[
   \Sel_{\P^M}(A/L) = \ker\biggl(\hone(L, A[\P^M]) \to \prod_v \hone(L_v, A)\biggr),
   \]
where the product runs over all the places $v$ of $L$ and the maps are those arising in global and local cohomology from the short exact sequence of $G_{F}$-modules
$\begin{tikzcd}[cramped, column sep=small]
     0 \ar[r] & A[\P^M] \ar[r] & A \ar[r, "\pi_M"] & A \ar[r] & 0.
 \end{tikzcd}$

\end{definition}
By definition, $\Sel_{\P^M}(A/L)$ fits into the short exact sequence
\begin{equation}\label{eq:ses-kummer-P-M}
    \begin{tikzcd}
    0 \ar[r] & A(L)/\P^M A(L) \ar[r, "\delta_L"] & \Sel_{\P^M}(A/L) \ar[r] & \sha(A/L)[\P^M] \ar[r] & 0,
\end{tikzcd}
\end{equation}
where $\delta_L$ is the Kummer map with respect to the multiplication by $\pi_M$, i.e., the connecting morphism in Galois cohomology. 
Define 
\begin{align*}
    \Sel_{\P^\infty}(A/L) &:= \varinjlim_M \Sel_{\P^{M}}(A/L) \subseteq \hone(L, A[\P^\infty]);\\
    S_{\P^\infty}(A/L) &:= \varprojlim_M \Sel_{\P^{M}}(A/L) \subseteq \hone(L, T_\P A),
\end{align*}
where $M$ varies among the integers that are divisible by the order of $\P$ in the class group of $\cOE$. Taking the direct and inverse limits of 
\eqref{eq:ses-kummer-P-M}, we obtain the short exact sequences 
\begin{equation}\label{eq:ses-selmer-sha}
    \begin{tikzcd}
     0 \ar[r] & A(L) \otimes_{\cOE} E_\P/\cOP \ar[r] & \Sel_{\P^\infty}(A/L) \ar[r] & \sha(A/L)[\P^\infty] \ar[r] & 0;
\end{tikzcd}
\end{equation}
\begin{equation}\label{eq:ses-selmer-cpt-tsha}
\begin{tikzcd}
    0 \ar[r] & A(L) \otimes_{\cOE} \cOP \ar[r] & S_{\P^\infty}(A/L) \ar[r] & T_\P \sha(A/L) \ar[r] & 0,
\end{tikzcd}
\end{equation}
where $T_\P \sha(A/L) = \varprojlim_M \sha(A/L)[\P^M]$.

Let now $L/F$ be an algebraic extension of fields. If $M=\infty$ or as before, we denote by
\[
\Sel_{\P^M}(A/L) = \varinjlim_{F', \res} \Sel_{\P^M}(A/F'), \qquad S_{\P^M}(A/L) = \varprojlim_{F', \cores} S_{\P^M}(A/F'),
\]
where $F'$ varies through all intermediate field $F \subseteq F' \subseteq L$ of finite degree over $F$ and the limits are taken with respect to restriction and corestriction maps in cohomology, respectively. These still fit into the exact sequences \eqref{eq:ses-kummer-P-M}, \eqref{eq:ses-selmer-sha} and \eqref{eq:ses-selmer-cpt-tsha}.

\begin{remark}
    The definition of the Kummer map in \cite{matar-nekovar:kolyvagin} 
    is slightly different from ours, which, instead, agrees with the one in \cite{nekovar:euler-system-method}. 
    However, both maps fit in the short exact sequence \eqref{eq:ses-kummer-P-M}, therefore they have the same image and they differ by an isomorphism. The same applies if we choose a different generator of $\P^M$.
\end{remark}

We now characterize the image of the local Kummer map at primes of good reduction in terms of unramified Galois cohomology. Let $L/F$ be an algebraic field extension.
\begin{definition}
    Let $v$ be a finite place of $L$. For any integer $M\ge1$, define
    \[
    \hone_\ur(L_v, A[\P^{M}]) = \ker\Bigl(\hone(L_v, A[\P^{M}]) \to \hone\bigl(L_v^\ur, A[\P^{M}]\bigr)\Bigr).
    \]
\end{definition}

\begin{proposition}\label{prop:comparison-image-kummer-unramified-local-conditions}
    Let $v \nmid p$ be a finite place of $L$, suppose that $A$ has good reduction at $v$ and let $M\ge1$ be an integer satisfying \eqref{ass:P-M-princ}. 
    Then the local Kummer map $\delta_v$ induces an isomorphism
    \[
    \delta_v \colon A(L_v)/\P^{M}A(L_v) \iso \hone_\ur(L_v, A[\P^{M}]).
    \]
\end{proposition}

\begin{proof}
    The same argument of \cite[Lemma VIII.2.1]{silverman:arithmetic-elliptic-curves} shows that the image of $\delta_v$ is contained in $\hone_\ur(L_v, A[\P^{M}])$. By the inflation-restriction exact sequence, the latter is isomorphic to $\hone\bigl(L_v^\ur/L_v, A[\P^{M}]\bigr)$, therefore the local Kummer sequence induces the short exact sequence
    \[
    \begin{tikzcd}[column sep=small]
        0 \ar[r] & A(L_v)/\P^{M} A(L_v) \ar[r, "\delta_v"] & \hone\bigl(L_v^\ur/L_v, A[\P^{M}]\bigr) \ar[r]& \hone\bigl(L_v^\ur/L_v, A(L_v^\ur)\bigr)[\P^{M}].
    \end{tikzcd}
    \]
    Since $A$ has good reduction at $v$, the last term vanishes by \cite[Proposition 3.8]{milne:arithmetic-duality}.
\end{proof}

\begin{remark}\label{remark:hones}
    Note that the previous proposition implies that $\hone(L_v, A)[\P^{M}] \cong \frac{\hone(L_v, A[\P^{M}])}{\hone_\ur(L_v, A[\P^{M}])}$ for any finite place $v \nmid p$ of good reduction.
\end{remark}

\subsection{Duality}\label{sec:duality}
 
For any discrete or compact topological $\Z_p$-module $X$, denote by  
\[
D(X) = \Hom_{\Z_p}^\cont(X, \Q_p/\Z_p)
\]
the Pontryagin dual of $X$, endowed with the compact-open topology.

\begin{lemma}\label{lemma:pontryagin-iso}
If $X$ is a discrete or profinite $\cOP$-module, 
\[
D(X) \cong \Hom_{\cOP}^\cont(X, E_\P/\cOP) =:D_{\cOP}(X),
\]
the isomorphism depending on the choice of a generator of the inverse different of $\cOP/\Z_p$.
\end{lemma}

\begin{proof}
    Since $\cOP$ is a finite free $\Z_p$-module, we have that
    \[
    \Hom_{\Z_p}^\cont(\cOP, \Q_p/\Z_p) = \Hom_{\Z_p}(\cOP, \Q_p/\Z_p) \cong \Hom_{\Z_p}(\cOP, \Z_p) \otimes_{\Z_p} \Q_p/\Z_p
    \]
    (see \cite[Proposition 1.10]{mastella:vanishing-sha-mod-forms} and the discussion above it). 
    Chosen a generator $c \in \mathscr{D}_{\cOP/\Z_p}^{-1}$ of the inverse different, we get an isomorphism 
\begin{equation}\label{eq:generator-inverse-different}
\cOP \iso \Hom_{\Z_p}(\cOP, \Z_p), \quad a \mapsto [b \mapsto\Tr_{E_\P/\Q_p}(cab)]
\end{equation}
(see \cite[Section III.2]{neukirch:algebraic-number-theory}).
Putting these facts together and using the $(\otimes, \Hom)$-adjunction (\cite[5.5.4]{ribes-zalesskii:profinite} in the profinite case), if $X$ is a profinite or discrete topological $\cOP$-module, we obtain an isomorphism 
    \[
    \Hom_{\Z_p}^\cont(X, \Q_p/\Z_p) \cong \Hom_{\cOP}^\cont(X, \cOP \otimes_{\Z_p} \Q_p/\Z_p).
    \]
We conclude noting that $\cOP \otimes_{\Z_p} \Q_p/\Z_p \cong E_\P/\cOP$. 
\end{proof}

\subsection{The Weil pairing}\label{sec:weil-pairing}

Fix now a polarization $\gamma \colon A \to A^\vee$ and consider the following hypotheses, which will be in force until the end of Section \ref{sec:av}:
\leqnomode
\begin{gather}
    \text{$\gamma$ is $\cOE$-linear (i.e., $\alpha^\vee \circ\gamma = \gamma \circ \alpha$ for any $\alpha \in \cOE=\End_F(A)$)};\label{ass:O-E-lin}\tag{$\cOE$-lin}\\
    \text{$p \nmid \deg(\gamma)$}.\tag{$\deg$-$\gamma$}\label{ass:deg-lambda}
\end{gather}
\reqnomode
The condition \eqref{ass:O-E-lin} is equivalent to the fact that the Rosati involution attached to $\gamma$ acts trivially on $\cO_E$. The Weil pairing
\[
(\,\, , \, )_p \colon T_p A \times T_p A^\vee \to \Z_p(1)
\]
is $\Z_p$-bilinear, $G_F$-equivariant, perfect and satisfies $(\alpha_\ast x, y) = (x, \alpha^\vee_\ast y)$ for any $\alpha \in \End_F(A)$. Since the degree of the polarization $\gamma$ is prime to $p$, it follows that the map
\[
\gamma_\ast \colon T_p A \longrightarrow T_p A^\vee
\]
induced by $\gamma$ is an isomorphism (see, e.g., \cite[(10.6)]{edixhoven-moonen-van-der-geer:av}). Therefore, the pairing 
\[
(\,\, , \, )_p^{\gamma} \colon T_p A \times T_p A \to \Z_p(1)
\]
defined by $(x , y )_p^{\gamma} = (x, \gamma_\ast y)_p$ is skew-symmetric, $\Z_p$-bilinear, $G_F$-equivariant and perfect. Moreover, assumption \eqref{ass:O-E-lin} implies that this pairing is also $\cOE$-bilinear, in the sense that, for any $x,y \in T_p A$ and $\alpha \in \cOE = \End_F(A)$, we have that $(x , \alpha_\ast y )_p^{\gamma} = (\alpha_\ast x , y )_p^{\gamma}$ . 
\begin{lemma}
The Weil pairing $(\, , \,)_p^\gamma$ restricts to a skew-symmetric $G_F$-equivariant $\Z_p$-bilinear perfect pairing $T_\P A \times T_\P A \to \Z_p(1)$.
\end{lemma}

\begin{proof}
    Let $p = \prod_{i=1}^t \P_i^{r_i}$ be the prime factorization of $p$ in $\cOE$, with $\P_1 = \P$. Thanks to the decomposition
    $
    T_p A = \bigoplus_{i=1}^t T_{\P_i} A,
    $
    we only have to show that $\bigoplus_{i=2}^t T_{\P_i} A$ is the exact complement of $T_\P A$ with respect to the pairing $(\, , \,)_p^\gamma$.
     
    Let $a \in \cOE \otimes \Z_p$ corresponding, in the decomposition $\cOE \otimes \Z_p = \prod_{i=1}^t \cO_{E,\P_i}$, to $0 \in \cOEP$ and $1 \in \cO_{E, \P_i}$ for any $i > 1$. 
    Thanks to the $\cOE$ and $\Z_p$-bilinearities, we have that 
    \[
    (x, y)_p^\gamma = (x, ay)_p^\gamma = (ax, y)_p^\gamma = (0, y)_p^\gamma = 0,
    \]
    for any $x \in T_\P A$ and $y \in \bigoplus_{i=2}^t T_{\P_i} A$.     
    To conclude the proof, it is enough to observe that the existence of $0 \ne y \in T_\P A$ orthogonal to $T_\P A$ would contradict the non-degeneracy of $(\, , \,)_p$.
\end{proof}
Choose from now on a generator $c \in \different_{\cOP/\Z_p}^{-1}$ of the inverse different. The corresponding isomorphism \eqref{eq:generator-inverse-different} gives an isomorphism of $\cOP$-modules 
\begin{align*}
    \Hom_{\cOP}\bigl(T_\P A, \cOP(1)\bigr)&\cong \Hom_{\cOP}\bigl(T_\P A, \Hom_{\Z_p}(\cOP, \Z_p(1))\bigr) \cong\\
    & \cong \Hom_{\Z_p}\bigl(T_\P A \otimes_{\cOP} \cOP, \Z_p(1)\bigr) =\Hom_{\Z_p}\bigl(T_\P A, \Z_p(1)\bigr),
\end{align*}
so that the adjunction morphism of the Weil pairing induces
\[
T_\P A \iso \Hom_{\Z_p}(T_\P A, \Z_p(1)) \iso  \Hom_{\cOP}\bigl(T_\P A, \cOP(1)\bigr).
\]
Therefore, we get an $\cOP$-bilinear version of the Weil pairing
\begin{align*}
    (\, ,\,)_{\P}^\gamma \colon \colon T_\P A \times T_\P A \to \cOP(1), 
\end{align*}
characterized by 
\[
\Tr_{\cOP/\Z_p}\bigl( c a \cdot (x, y)_{\P}^\gamma\bigr)= (x, ay)_p^\gamma= (ax, y)_p^\gamma,
\]
for any $x, y \in T_\P A$, $a \in \cOP$.

Reducing modulo $\P^M$, we obtain a skew-symmetric $\cOE/\P^M$-bilinear $G_F$-equivariant perfect pairing
\[
(\, , \, )_{\P, M}^\gamma \colon A[\P^M] \times A[\P^M] \to (\cOP/\P^M)(1).
\]
The properties of this Weil pairing imply this well-known consequence for the action of $G_F$ on $T_\P A$. 
\begin{lemma}\label{lemma:det}
    $\det (\rho_\P) = \chi_p$, where $\chi_p \colon G_F \to \Z_p^\times \subseteq \cOP^\times$ is the $p$-cyclotomic character.
\end{lemma}

\begin{proof}
It follows from an explict computation, using the $G_F$-equivariance, the alternating and the nondegeneracy properties of $(\, ,\,)_{\P}^\gamma$.
\end{proof}
It follows that, for any finite prime $\ellfr \nmid p$ of $F$ of good reduction for $A$,
\[
\charpoly(\Frob_\ellfr | T_\P A) = X^2 - a_\ellfr X + N \ellfr
\]
(and is independent of $p$ and $\P$), where $a_\ellfr := \Tr\bigl(\rho_\P(\Frob_\ellfr)\bigr) \in \cOE$ and $N\ellfr\in\Z$ is the norm of $\ellfr$ in $E/\Q$.

\subsection{Cup product and Tate duality}\label{sec:cup-tate-products}

Let now $L/F$ be a finite extension of fields, let $v$ be a place of $L$ and let
\[
\cup_{\P, M, v}^\gamma \colon \hone(L_v, A[\P^M]) \times \hone(L_v, A[\P^M]) \to \cOP/\P^M
\]
be the cup product attached to the Weil pairing $(\, , \, )_{\P, M}^\gamma$, for any $M\ge 1$. This is perfect, as follows by the following lemma.
\begin{lemma}
    For any $N$ finte discrete $\cOP[G_{F_v}]$-module, the cup product
    \[
    \cup_{\cOP} \colon \hone(F_v, N) \times \hone(F_v, D_{\cOP}(N)(1)) \to \cohomology^2(F_v, E_\P/\cOP(1)) \cong E_\P/\cOP(1)
    \]
    is perfect. Moreover, the following diagram commutes
    \begin{equation}\label{eq:diagram-cup-compatibility}
    \begin{tikzcd}
       \hone(F_v, D(N)(1)) \ar[r, "\adj(\cup)"]\ar[d]& D\bigl(\hone(F_v, N)\bigr)\ar[d]\\
        \hone(F_v, D_{\cOP}(N)(1)) \ar[r, "\adj(\cup_{\cOP})"]& 
        D_{\cOP}\bigl(\hone(F_v, N)\bigr),
    \end{tikzcd}
    \end{equation}
    where the horizontal maps are the adjoint morphisms attached respectively to $\cup_{\cOP}$ and the cup product
    \[
    \cup \colon \hone(F_v, N) \times \hone(F_v, D(N)(1)) \to \cohomology^2(F_v, \Q_p/\Z_p(1)) \cong \Q_p/\Z_p(1),
    \]
    while the vertical arrows are (induced by) the isomorphism of Lemma \ref{lemma:pontryagin-iso}.
\end{lemma}

\begin{proof}
    By Tate's local duality \cite[7.2.6]{neukirch:cnf}, it's enough to prove the commutativity of the square. Identifying $D_{\cOP}(X)$ with $\Hom_{\cOP}(X, \Hom_{\Z_p}(\cOP, \Q_p/\Z_p))$ for any discrete $\cOP$-module $X$, as seen in the proof of Lemma \ref{lemma:pontryagin-iso}, this can be explicitely checked using the explicit description of the adjunction isomorphism with respect to  $(\otimes, \Hom)$ and of the cup products.
\end{proof}

Assume until the end of this subsection that $M$ satisfies \eqref{ass:P-M-princ}. The next proposition defines an $\cOP/\P^M$-linear version of the Tate pairing.
\begin{proposition}\label{prop:tate-pairing}
The image of $A(L_v)/\P^M A(L_v)$ in $\hone(L_v, A[\P^M])$ via the Kummer map is a maximal isotropic subgroup of $\hone(L_v, A[\P^M])$ for the cup product $\cup_{\P, M, v}^\gamma$. In other words, $\cup_{\P, M, v}^\gamma$ induces a perfect pairing
\[
\langle \; , \;\rangle_{\P, M, v}^\gamma \colon A(L_v)/\P^M A(L_v) \times \hone(L_v, A)[\P^M] \to \cOP/\P^M.
\]
\end{proposition}

\begin{proof}
We apply the the same argument in the proof \cite[Theorem I.3.2]{milne:arithmetic-duality} (which combines I.0.3, I.0.8,  I.0.14, I.0.16, I.0.18, I.2.1, I.3.1 in \emph{loc.~cit.}) to the pairing 
\[
A[\P^M] \times A[\P^M] \to \Q_p/\Z_p(1),
\]
induced by $(\, , \, )_{\P, M}^\gamma$, applying Lemma \ref{lemma:pontryagin-iso}, and to the short exact sequence 
\[
0\to A[\P^M] \to A \xrightarrow{\pi_M} A \to 0:
\] 
we get therefore (together with \eqref{ass:deg-lambda}) that there is a commutative diagram 
\begin{equation}
        \label{eq:diagram-tate-duality'}
        \begin{tikzcd}[cramped, column sep=small]
            0 \ar[r] & A(L_v)/\P^M A(L_v) \ar[r] \ar[d] & \hone(L_v, A[\P^M]) \ar[r] \ar[d] & \hone(L_v, A)[\P^M] \ar[r] \ar[d] & 0\\
            0 \ar[r] &  D\Bigl(\hone(L_v, A)[\P^M]\Bigr) \ar[r] & D(\hone(L_v, A[\P^M])) \ar[r] &  D\Bigl(A(L_v)/\P^M A(L_v)\Bigr) \ar[r] & 0
        \end{tikzcd}
    \end{equation}
(up to a sign), where the middle vertical arrow is the adjoint isomorphism induced by that pairing. Therefore it's enough to show that the leftmost vertical map (which is known to be injective) is a bijection. Then, by the commutatity of \eqref{eq:diagram-cup-compatibility}, the lemma follows composing it with the isomorphism of Lemma \ref{lemma:pontryagin-iso}.

We prove its surjectivity by order considerations. Recall that, by Theorem I.3.2 of \emph{loc.~cit.}, the limit of cup products with respect to the $\bmod \, n$ Weil parings induces a perfect pairing
\[
\langle \, , \, \rangle_v \colon A^\vee(L_v) \times \hone(L_v, A) \to \Q/\Z,
\]
which is by definition $\cOE$-linear, in the sense that for any $\phi \in \cOE= \End_F(A)$ the equality $\langle P , \phi_\ast \alpha \rangle_v = \langle \phi^\vee P , \alpha \rangle_v$ holds.
Therefore under the adjoint isomorphism attached to $\langle\,,\,\rangle_v$, we have that $\phi^\vee(A^\vee(L_v))$ corresponds to $(\phi_\ast)^\ast D(\hone(L_v, A))$ and thus
\[
\frac{A^\vee(L_v)}{\phi^\vee(A^\vee(L_v)) } \iso \frac{D(\hone(L_v, A))}{(\phi_\ast)^\ast D(\hone(L_v, A))} \cong D(\hone(L_v, A)[\phi_\ast]).
\]
Applying this to $\phi = \pi_M$ and recalling the assumption \eqref{ass:deg-lambda}, this shows that 
\[
\#A(L_v)/\P^M A(L_v) = \#\hone(L_v, A)[\P^M],
\]
finishing the proof.
\end{proof}

\subsection{Ordinary reduction}\label{sec:ordinary}

Assume for this subsection that $A$ has good reduction at a finite place $\p \mid p$ of $F$, let $\cA_\p/\cO_{F_\p}$ be the Néron model of $A_\p := A \otimes_F F_{\p}$ and $\tilde{A}_\p/k_\p$ be its special fibre, where $k_\p$ denotes the residue field of $F$ at $\p$. 
\begin{lemma}\label{lemma:good-reduction-iso-tate-modules}
    For any prime $\ell$ different from $p$, we have that $T_\ell A \cong T_\ell \tilde{A}_\p$ as $G_{F_\p}$-modules.
\end{lemma}

\begin{proof}
    This follows from \cite[Lemma 2]{serre-tate:good-red-av}, since $A[\ell^n]$ is unramified for any $n\ge1$ by the Néron--Ogg--Shafarevich criterion for abelian varieties, i.e., \cite[Theorem 1]{serre-tate:good-red-av}.
\end{proof}

\begin{definition}
    We say that $A$ has ordinary reduction at $\p$ if $\rk_{\Z_p}T_p \tilde{A}_\p = g (= \frac{1}{2} \rk_{\Z_p}T_p A)$.
\end{definition}

Assume that $\ell \nmid \deg(\gamma)$ and write $P_{\p}(X) := \charpoly(\Frob_\p |T_\ell A) \in \Z[X]$. For a prime $\lambda$ of $E$ above $\ell$ write also
$f_\p(X) := \charpoly (\Frob_\p | T_\lambda A) = X^2 - a_\p X + N\p \in \cOE[X]$.
These polynomials are related, by \cite[Proposition 11.9]{shimura:anf-sympl-disc-groups}, by the relation
\begin{equation}
P_\p(X) = N_{E/\Q}(f_\p(X)):=\prod_{\sigma\colon E\to\C}\sigma(f_\p(X)).\label{eq:relation-char-poly}
\end{equation}
\begin{lemma}\label{lemma:ordinary-equivalent-1}
    Let $a_{\p, g}$ be the coefficient of $X^g$ in $P_{\p}(X)$. Then $A$ has ordinary reduction at $\p$ if and only if $p \nmid a_{\p, g}$. 
\end{lemma}

\begin{proof}
    It is known that $A$ has good ordinary reduction if and only if its Newton polygon has slopes $0$ and $1$, both with multiplicities $g$. By \cite[Corollary 15.36]{edixhoven-moonen-van-der-geer:av}, this is true if and only if $p$ does not divide the coefficient of $X^g$ in the characteristic polynomial of $\Frob_\p$ acting on $T_\ell\tilde{A}_\p$. The latter equals $P_{\p}(X)$, by Lemma \ref{lemma:good-reduction-iso-tate-modules}.
\end{proof}

\begin{proposition}\label{proposition:ordinary-equivalent-2}
    $A$ has ordinary reduction at $\p$ if and only if $a_\p \in \cOEP^\times$, for any prime $\P \mid p$ of $E$.
\end{proposition}

\begin{proof}
    We follow the argument of \cite[Proposition 3.8]{wang-tian:ordinary-GL-2-av}.
    By \eqref{eq:relation-char-poly} (and recalling that $[E:\Q]=g$) we know that $a_{\p, g}=N_{E/\Q}(a_\p) + N\p \cdot \ast$, for some $\ast \in \cOE \cap \Q = \Z$. This implies that $p\nmid a_{\p, g}$ if and only if $a_\p\in\cO_{E,\P}^\times$. The result follows from Lemma \ref{lemma:ordinary-equivalent-1}.
\end{proof}

By \cite[Section 7.3, Proposition 6]{bosch-lut-raynaud:neron-models}, any isogeny $\alpha \colon A \to A$ defined over $F$ extends to a morphism on $\cA_\p$  that specialises to an isogeny of the special fibre $\tilde{A}_\p$. In this way, we get a action of $\cOE$ on $\tilde{A}_\p$ (commuting with the $G_{k_\p}$-action) and therefore we may consider the $\P^M$-torsion $\tilde{A}_\p[\P^M]$, for any $M>0$, and the $\P$-adic Tate module $T_\P \tilde{A}_\p$ of $\tilde{A}_\p$. In the rest, we will consider also the following weaker ordinarity assumption.
\begin{definition}
    Let $\P$ be a prime of $E$ above $p$. We say that $A$ has $\P$-ordinary reduction at $\p$ if $\rk_{\cOP}T_\P \tilde{A}_\p = 1 (=\frac{1}{2} \rk_{\cOP} T_\P A)$. 
\end{definition}
Notice that $A$ has ordinary reduction at $\p$ if and only if it has $\P$-ordinary reduction at $\p$ for any prime $\P$ of $E$ above $p$.

\begin{lemma}\label{lemma:reduction-divisible}
    $\tilde{A}_\p[\P^\infty]$ is a divisible $\cOP$-module.
\end{lemma}
\begin{proof}
    Let $\pi\in \cOE$ such that $\pi \in \P \smallsetminus \P^2$. Since $\pi$ is a uniformiser of $\cOP$, it is enough to prove that $\tilde{A}_\p[\P^\infty]$ is $\pi$-divisible. Now, since $\pi$ acts as an isogeny on $\tilde{A}_\p$, it is surjective on $\tilde{A}_\p(\bar{k}_\p)$. Moreover, the preimage $Q \in \tilde{A}_\p(\bar{k}_\p)$ of an element $P \in \tilde{A}_\p(\bar{k}_\p)[\P^\infty]$ via $\pi$ still lies in $\tilde{A}_\p(\bar{k}_\p)[\P^\infty]$: indeed, let $m$ such that $P \in \tilde{A}_\p(\bar{k}_\p)[\P^{m}]$ and denote by $t$ the order of $\P$ in the class group of $E$. Then 
    \[
    Q \in \tilde{A}_\p(\bar{k}_\p)[\pi^{t(m+1)}] = \tilde{A}_\p(\bar{k}_\p)[\P^{t(m+1)}].
    \]
    To prove the latter equality, consider a generator $z$ of $\P^t$, so that $\pi^{t(m+1)} = z^{m+1}x$ for some $x \notin \P$ and $
    R\in \tilde{A}_\p(\bar{k}_\p)[\pi^{t(m+1)}]$; then there exist $r, s \in \cOE$ such that $1 = rx + s\pi^{t(m+1)}$, thus 
    \[
    z^{t(m+1)}R = z^{t(m+1)}(rx R + s\pi^{t(m+1)}R) = r (z^{t(m+1)}x)R = r (\pi^{t(m+1)}R)=0.
    \]
    Thus $Q \in \tilde{A}_\p(\bar{k}_\p)[\P^{t(m+1)}]$. The reverse implication is trivial.
\end{proof}

\begin{proposition}\label{prop:P-ordinary-equiv}
    $A$ has $\P$-ordinary reduction at $\p$ if and only if $\tilde{A}_\p[\P^M] \cong \cOE/\P^M$ for any $M \ge 1$. 
\end{proposition}

\begin{proof}
By Lemma \ref{lemma:reduction-divisible}, we have that $\tilde{A}_\p[\P^\infty]=(E_\P/\cOP)^r$ for some integer $r \ge 0$ (as a consequence of Maltis structure theorem for injective modules). $\P$-ordinarity implies that $r=1$.
\end{proof}

In the next proposition $\chi_{\P, M}$ denotes composition of the cyclotomic character with the natural map $\Z_p^\times \inj \cOEP \surj \cOE/\P^M$. Recall that, by Lemma \ref{lemma:det}, we have that $\det(\rho_{\P, M}) = \chi_{\P, M}$.
\begin{proposition}\label{prop:P-ordinary-shape}
    Suppose that $A$ has $\P$-ordinary reduction at $\p$ and let $I_\p\subseteq D_\p$ be respectively the inertia and decomposition group at $\p$. Then, for any $M\ge1$ there is a basis of $A[\P^M]$ such that $\rho_{\P,M}(D_\p)$ is upper-triangular and
    \[
    \rho_{\P,M}|_{I_\p} = 
    \begin{pmatrix}
    \chi_{\P, M} & \ast\\ 0&1    
    \end{pmatrix}.
    \]
\end{proposition}

\begin{proof}
Let $\bar{\mathfrak{m}}$ be the maximal ideal of $\Fbar_\p$ and $\mathcal{F}$ be the formal group over $\cO_{F_\p}$ attached to $A_\p$. By \cite[Theorem C.2.6]{hindry-silverman:diophantine-geo} and \cite[Proposition IV.6.4(a)]{silverman:advanced}, we have an exact sequence
\[
\begin{tikzcd}
0\ar[r]&  \mathcal{F}(\bar{\mathfrak{m}}) \ar[r] & A(\Fbar_\p) \ar[r] & \tilde{A}_\p(\bar{k}_\p) \ar[r]& 0,
\end{tikzcd}
\]
that induces for any $m \geq 1$, the exact sequence
\begin{equation}\label{fp_exact_sequence1}
\begin{tikzcd}
    0\ar[r]& \mathcal{F}(\bar{\mathfrak{m}})[\P^m]\ar[r]&A(\Fbar_\p)[\P^m]\ar[r, "{\pi'}"]& \tilde{A}_\p(\bar{k}_\p)[\P^m].
\end{tikzcd}
\end{equation}
Now choose $m \geq M$ such that $\P^m$ is principal, say generated by $\alpha$, so that $A(\Fbar_\p)[\P^m] =A(\Fbar_\p)[\alpha]$ and $\tilde{A}(\bar{k}_\p)[\P^m] = \tilde{A}(\bar{k}_\p)[\alpha]$. Thus, it follows by \cite[Lemma 2.1.2]{tan:gen-mazur-th} that $\pi'$ is surjective. 
From this fact, since $A(\Fbar_\p)[\P^m] \cong (\cO_{\P}/\P^m)^2$ and $\tilde{A}(\bar{k}_\p)[\P^m] \cong \cO_{\P}/\P^m$ by Proposition \ref{prop:P-ordinary-equiv}, we deduce that $\mathcal{F}(\bar{\mathfrak{m}})[\P^m] \cong \cO_{\P}/\P^m$. This then implies that $\mathcal{F}(\bar{\mathfrak{m}})[\P^M] \cong \cO_{\P}/\P^M$.

It follows that we have a short exact sequence
\begin{equation}\label{fp_exact_sequence2}
\begin{tikzcd}
    0\ar[r]& \mathcal{F}(\bar{\mathfrak{m}})[\P^M]\ar[r]&A(\Fbar_\p)[\P^M]\ar[r, "{\pi}"]& \tilde{A}_\p(\bar{k}_\p)[\P^M] \ar[r]&0
\end{tikzcd}
\end{equation}
and that we can choose a basis $P, Q$ of $A(\Fbar_\p)[\P^M] \cong (\cOE/\P^M)^2$ such that $P \in \mathcal{F}(\bar{\mathfrak{m}})[\P^M]$.
Let $\sigma \in D_\p$ and let $\rho_{\P, M}(\sigma)=\bigl(\begin{smallmatrix} a & b \\ c & d \end{smallmatrix}\bigr) \in \GL_2(\cOE/\P^M)$ in the basis $\set{P, Q}$. Since $\mathcal{F}(\bar{\mathfrak{m}})[\P^M]$ is stable under $\Gal(\Fbar_\p/F_\p)$, being the kernel of the equivariant reduction map, then $c=0$. Now let $\sigma \in I_\p$; we have that
\[d\pi(Q)=\pi(bP+dQ)=\pi(\sigma Q)=\sigma(\pi(Q))=\pi(Q)\]
and therefore (since $Q$ is a basis element) $d=1$. It follows that $a = \det\rho_{\P, M}(\sigma)=\chi_{\P, M}(\sigma)$, whence the result.
\end{proof}

\section{Modular Abelian varieties and Heegner points}\label{sec:modular-av-and-heegner}

\subsection{Shimura curves}\label{sec:shimura-curves}

For the reader's convenience, we recall some properties of Shimura curves over totally real fields, mainly following \cite[§1]{nekovar:euler-system-method}.

Recall that $F$ is a totally real field of degree $d$ over $\Q$. Fix an archimedean prime $\tau_1$ of $F$ and a finite set $S_\B$ of non-archimedean primes of $F$ whose cardinality has opposite parity with respect to $d$. Let $\B$ be the (unique, up to isomorphism) quaternion algebra over $F$ ramified at the set
\begin{equation*}
    \Ram(\B)=S_\B\cup\set{\tau\mid\infty \ : \ \tau\ne\tau_1}.
\end{equation*}
For every real embedding $\tau\colon F\hookrightarrow \R$ put $\B_\tau:= \B\otimes_{F,\tau}\R$. When $\tau=\tau_1$, we also fix an isomorphism of $\R$-algebras $\eta_\infty\colon \B_{\tau_1}\cong\M_2(\R)$.

Let $H$ be a compact open subgroup of the finite ideles $\Bhat^\times$ of $\B$. The double coset
\begin{equation*}
    M_H(\C):=\B^\times\backslash (\mathcal{H}^\pm\times \Bhat^\times)/H
\end{equation*}
has a natural structure as a Riemann surface, where $\B^\times$ acts on $\mathcal{H}^\pm:=\C\setminus\R$ via M\"obius transformations (seeing $\B^\times\subseteq \B_{\tau_1}^\times$ inside $\GL_2(\R)$ by applying the isomorphism $\eta_\infty$) and on $\Bhat^\times$ by left multiplication, and $H$ acts as right multiplication on $\Bhat^\times$ and trivially on $\mathcal{H}^\pm$. Moreover, it consists of the $\C$-points of an algebraic curve $M_H$, that is smooth over $F$. We denote by $[z,b]$ the point of $M_H(\C)$ represented by the pair $(z,b)\in\mathcal{H}^\pm\times\Bhat^\times$.

When $F=\Q$ and $S_\B=\emptyset$ (i.e., when $\B=\M_2(\Q)$) the curves $M_H$ are the usual modular curves, and we denote by $M^\ast_H$ their standard Baily-Borel compactification, obtained adjoining a set of cusps of $M_H$. If $\B\ne \M_2(\Q)$, the curves $M_H$ are proper over $F$ (and the Riemann surface $M_H(\C)$ is compact); in this case we just set $M_H^\ast:=M_H$.

The action of $\Fhat^\times$ on $M_H(\C)$ defined by 
\begin{equation*}
    \langle a\rangle[z,b]:=[z,ba]
\end{equation*}
 for every $z\in\mathcal{H}^\pm$, $b\in\Bhat^\times$ and $a\in \Fhat^\times$ induces an action of $\Fhat^\times$ on $M_H$ and $M_H^\ast$, that factors through the finite quotient $\Fhat^\times/(\Fhat^\times\cap H)F^\times$. We denote by $N_H$ and $N_H^\ast$ the corresponding quotient curves, which are again smooth over $F$. We also have an identification
\begin{equation*}
    N_H(\C)=\B^\times\backslash (\mathcal{H}^\pm\times \Bhat^\times)/H\Fhat^\times
\end{equation*}
as Riemann surfaces. In particular, the Shimura curves $N_H$ and $N_H^\ast$ depend only on the subgroup $H\Fhat^\times$ of $\Bhat^\times$. Each curve $M_H^\ast$ and $N_H^\ast$ is irreducible over $F$, but not necessarily geometrically irreducible. By a little abuse of notation, we will use the notation $[z,b]$ to denote points on $N_H$, too. The ambient curve will be clear from the context and, from a certain point on, we will only be interested on $N_H$.

\subsection{Hecke operators}\label{sec:Hecke-operators}

Let $S\supseteq S_\B$ be the minimal set of non archimedean primes of $F$ such that $H=H_S H^S$, where $H_S\subseteq \prod_{\ellfr\in S}\B_\ellfr^\times$ and $H^S=\prod_{\ellfr\notin S}H_\ellfr$ with $H_\ellfr$ a maximal compact subgroup of $\B_\ellfr^\times$. In other words, for each non-archimedean prime $\ellfr\notin S$ there exists a (unique) maximal $\cO_{F_\ellfr}$-order $R_\ellfr\subseteq \B_\ellfr$ such that $H_\ellfr=R_\ellfr^\times$.

\begin{definition}
    For every non-archimedean prime $\ellfr\notin S$, the \textbf{Hecke correspondence} $T_\ellfr$ on $M_H^\ast$ is the double coset operator
\begin{equation*}
    T_\ellfr=[H t_\ellfr H]\colon M_H^\ast\dashrightarrow M_H^\ast
\end{equation*}
for any element $t_\ellfr\in R_\ellfr\cap \B_\ellfr^\times$ (seen naturally inside $\widehat{B}^\times$) satisfying $\ord_\ellfr(\nr(t_\ellfr))=1$, where $\nr$ is the reduced norm and $\ord_\ellfr$ is the standard $\ellfr$-adic order.
\end{definition}
 For example, we can fix a uniformizer $\varpi_\ellfr$ of $\cO_{F_\ellfr}$ and take the element $t_\ellfr$ corresponding, under an isomorphism $R_\ellfr\cong\M_2(\cO_{F_\ellfr})$, to the matrix $\left(\begin{smallmatrix}
    \varpi_\ellfr&0\\0&1
\end{smallmatrix}\right)$. After fixing a decomposition $Ht_\ellfr HF^\times = \coprod_i \beta_i HF^\times$, the action of $T_\ellfr$ on $\Div(M_H^\ast(\C))$ is given by
\begin{equation*}
    T_\ellfr([z,b])=\sum_i [z,b\beta_i]
\end{equation*}
for every $(z,b)\in\mathcal{H}^\pm\times \Bhat^\times$.

The Hecke correspondences $T_\ellfr$ (for $\ellfr\notin S$) act on any ``reasonable'' cohomology theory for $M_H^\ast$ and on its Jacobian $\Jac(M_H^\ast)$. Moreover, if we replace $H$ by $H\Fhat^\times$, we can build Hecke correspondences for $N_H^\ast$ with similar properties.

\begin{example}\label{ex:Hecke-operators}
Let's compute $T_\ellfr$, for $\ellfr \notin S$ and
$t_\ellfr$ corresponding to 
$ \left(\begin{smallmatrix}
\varpi_\ellfr&0\\0&1
\end{smallmatrix}\right)$, 
under a fixed identification of $R_\ellfr$ with $\M_2(\cO_{F_\ellfr})$, as correspondence on $M_H^\ast$ (resp.~$N_H^\ast$).

This amounts to find a decomposition $Ht_\ellfr HF^\times = \coprod_i \beta_i HF^\times$ (resp.~$Ht_\ellfr H\hat{F}^\times = \coprod_i \beta_i H\hat{F}^\times$). Since $t_\ellfr \in \Bhat^\times$ has trivial component at any $\ellfr'\ne \ellfr$, it is enough to find a decomposition of 
$\GL_2(\cO_{F_\ellfr})
\left(\begin{smallmatrix}
\varpi_\ellfr&0\\0&1
\end{smallmatrix}\right)\GL_2(\cO_{F_\ellfr}) = \coprod_i \beta_i \GL_2(\cO_{F_\ellfr})$ such that the $\beta_i$ are not $F^\times$-proportional (resp.~$F_\ellfr^\times$-proportional). Similarly as in \cite[discussion around Proposition 5.2.1]{diamond-shurman} (taking right cosets instead of left ones) one can show that in both cases the set of $\beta_i$'s can be chosen to be the set
$\left\{\left(\begin{smallmatrix}
1&0\\0& \varpi_\ellfr
\end{smallmatrix}\right)\right\}\cup\left\{\left(\begin{smallmatrix}
\varpi_\ellfr&j\\0& 1
\end{smallmatrix}\right) \ | \ \bar{j}\in \cOF/\varpi_\ellfr\cOF \right\}$, or $\left\{\left(\begin{smallmatrix}
0&1\\\varpi_\ellfr&0
\end{smallmatrix}\right)\right\}\cup\left\{\left(\begin{smallmatrix}
\varpi_\ellfr&j\\0& 1
\end{smallmatrix}\right) \ | \ \bar{j}\in \cOF/\varpi_\ellfr\cOF \right\}$, where $j \in \cO_{F_\ellfr}$ is a lift of $\bar{j}$.
Thus, both in $\Jac(M_H^\ast)$ and $\Jac(N_H^\ast)$, we have that
\begin{equation}\label{eq:T-ellfr-formula}
T_\ellfr([z, b]) = \hspace{-5pt}\sum_{\bar{j} \in \cOF/\varpi_\ellfr\cOF} \hspace{-10pt}\left[z, b
\left(\begin{smallmatrix}
    \varpi_\ellfr&j\\0&1
\end{smallmatrix}\right)\right] + \left[z, b
\left(\begin{smallmatrix}
    1&0\\0&\varpi_\ellfr
\end{smallmatrix}\right)\right] = \hspace{-5pt}\sum_{\bar{j} \in \cOF/\varpi_\ellfr\cOF} 
\hspace{-10pt}
\left[z, b
\left(\begin{smallmatrix}
    \varpi_\ellfr&j\\0&1
\end{smallmatrix}\right)\right] + \left[z, b
\left(\begin{smallmatrix}
    0&1\\\varpi_\ellfr&0
\end{smallmatrix}\right)\right].
\end{equation}
\end{example}

\begin{definition}
    Let $\mathfrak{h}(H)$ be the subring of $\End\bigl(\Jac(N_H^\ast)\bigr)$ generated by all Hecke operators $T_\ellfr$ for any $\ellfr\notin S$.
\end{definition}
Recall that (see \cite[§1.16-1.18]{nekovar:euler-system-method}) the abelian variety $\Jac(N_H^\ast)$ has a canonical principal polarization and the corresponding Rosati involution fixes the Hecke algebra $\mathfrak{h}(H)$.

\subsection{CM points}\label{sec:heegner-points}

Let's now fix a CM extension $K/F$ such that 
\begin{equation}\label{ass:heegner-hypothesis}\tag{Heeg}
    \text{each prime $\q \in S_\B$ either ramifies or is inert in $K/F$,}
\end{equation}
and fix an embedding $\iota_K\colon K\hookrightarrow \B$ of $F$-algebras (which exists, under this hypothesis, thanks to \cite[Proposition 14.6.7]{voight:quaternion-algebras}). As noticed in \cite[Lemma 2.2]{nekovar:euler-system-method}, there is a unique point $z\in\mathcal{H}^+$ which is fixed by the action of $(\eta_\infty\circ\iota_K)(K^\times)$.

\begin{definition}
    The set of \emph{CM points by $K$} on the curve $N_H^\ast$ is the set 
    \[
    \CM(N_H^\ast,K)=\set{[z,b]\in N_H^\ast(\C) \ | \ b\in\Bhat^\times}.
    \]
\end{definition}

The reciprocity law (see \cite[(2.4)]{nekovar:euler-system-method}) states that $\CM(N_H^\ast,K)\subseteq N_H^\ast(K^\ab)$ and that the Galois action of $\Gal(K^\ab/K)$ on $\CM(N_H^\ast,K)$ is described, via the reciprocity map in class field theory $\rec_K\colon \widehat{K}^\times\to\Gal(K^\ab/K)$, by the formula
\begin{equation}\label{eq:action-of-reciprocity-CM-points}
    \rec_K(a)[z,b]=[z,\hat{\iota}_K(a)b]
\end{equation}
for all $a\in\widehat{K}^\times$.

Let's fix once and for all a CM point $x=[z,b]$ on $N_H^\ast$ and denote by $K(x)$ its field of definition. As shown in \cite[(4.1)]{nekovar:euler-system-method}, the extension $K(x)/K$ is determined, through the reciprocity map in class field theory, by the subgroup $K^\times\Fhat^\times Z$ of $\Khat^\times$, where $Z:=\hat{\iota}_K^{-1}(b H \cOhatF^\times b^{-1})$. Let's denote by $\cond(x)$ the ideal of $\cO_F$ such that $K[\cond(x)]$ is the smallest ring class field of $K$ containing $K(x)$, whence $Z\supseteq \cOhat_{\cond(x)}^\times$, where $\cO_{\cond(x)}$ is the order of conductor $\cond(x)$ of $K$.
\begin{remark}
    We follow the convention of \cite{zhang:gross-zagier-gl2}, \cite{nekovar:euler-system-method} and \cite{howard:iwasawa-gl2-abelian-varieties} about ring class fields, so that for any nonzero ideal $\cond$ of $F$, the reciprocity map induces an isomorphism
    \[
    \rec_K \colon \Khat^\times/K^\times \Fhat^\times \cOhat_{\cond}^\times \iso \Gal(K[\cond]/K),
    \]
    where $\cO_{\cond} = \cO_F + \cond \cOK$ is the order of $K$ of conductor $\cond$. 
    Note that in literature there is also another convention: see \cite{cornut-vatsal:CM-points},\cite{cornut-vatsal:nontriviality-rankin-selberg}, \cite{longo:anticyclotomic-imc-hilbert}.

    For further properties of the ring class fields (in particular, we will frequently use the fact that they are generalized dihedral over $F$), see \cite[Section 2.6]{nekovar:euler-system-method}
\end{remark}

Recall that, for each non-archimedean prime $\ellfr\notin S$ of $F$, where $S$ is the set of primes of $F$ defined at the beginning of Section \ref{sec:Hecke-operators}, there exists a (unique) maximal $\cO_{F_\ellfr}$-order $R_\ellfr\subseteq \B_\ellfr$ such that $H_\ellfr=R_\ellfr^\times$. From now on we fix, for any such a place $\ellfr$, an isomorphism $\eta_\ellfr\colon\B_\ellfr \cong \M_2(F_\ellfr)$ such that $(\eta_\ellfr\circ\Ad(b_\ellfr)^{-1})(R_\ellfr)=\M_2(\cO_{F_\ellfr})$, where $\Ad(b_\ellfr)$ denotes conjugation by $b_\ellfr$. More precisely, when $\ellfr\nmid\mathfrak{c}(x)$ is split in $K$, we choose $\eta_\ellfr$ with the property that
\begin{equation*}
    (\eta_\ellfr\circ \Ad(b_\ellfr)^{-1} \circ \iota_{K,\ellfr})(a,b)= \bigl(\begin{smallmatrix}
    a & b-a\\0&b    
    \end{smallmatrix}\bigr)
\end{equation*}
for every $(a,b)\in K_\ellfr\cong F_\ellfr\times F_\ellfr$. This can be done because any two optimal embeddings of $\mathcal{O}_{K_\ellfr}$ in $\M_2(\mathcal{O}_{F_\ellfr})$ are conjugated by an element of $\GL_2(\mathcal{O}_{F_\ellfr})$ (see \cite[Proposition 30.5.3]{voight:quaternion-algebras}).

Following \cite[Proposition 2.10]{nekovar:euler-system-method}, define the ideal $\mathfrak{I}_0$ of $\cO_K$ as
\[
\mathfrak{I}_0 = \lcm \set{(u-1) : u \in (\cO^\times_K)_{\tors}, \ u \ne 1}. 
\]
\begin{definition}\label{def:allowed-primes}
Let
$
\allowedprimes = \set{ \text{$\ellfr$ prime of $F$} \, : \, \text{$\ellfr$ is unramified in $K/F$}, \, \ellfr \notin S, \,  \ellfr \nmid \cond(x), \, \ellfr \cO_K \nmid \mathfrak{I}_0}
$
and, for any $r \ge 0$, let $\allowed_r$ be the set of products of powers of $r$ distinct elements of $\allowedprimes$ (with the convention that $\allowed_0 = \set{(1)}$). Let $\allowed = \bigcup_{r \ge 0} \allowed_r$.
\end{definition}

\begin{definition}\label{def:heegner-points-shimura}
    Let $h\colon \allowed\to \Bhat^\times$ be the function with the following properties:
    \begin{itemize}
        \item $h((1))=1$;
        \item for each $\ellfr\in\allowedprimes$, let $h(\ellfr) \in \B_\ellfr \subseteq \Bhat$ be the element such that $\eta_\ellfr(h(\ellfr))= \bigl(\begin{smallmatrix}
            \varpi_\ellfr & 0\\0& 1
        \end{smallmatrix}\bigr)$;
        \item for each $r>1$ and $\allowed_r\ni \nfrak=\ellfr_1^{\alpha_1}\cdots\ellfr_r^{\alpha_r}$ (with $\ellfr_i\in\allowedprimes$), put $h(\nfrak)=h(\ellfr_1)^{\alpha_1}\cdots h(\ellfr_r)^{\alpha_r}$.
    \end{itemize}
    For every $\nfrak\in\allowed$, define the CM points
    \begin{equation*}
        x(\nfrak):=[z,bh(\nfrak)]\in\CM(N_H^\ast,K).
    \end{equation*}
    We denote by $K(x(\nfrak))$ the field of definition of $x(\nfrak)$.
\end{definition}

Recall that $Z:=\hat{\iota}_K^{-1}(b H \cOhatF^\times b^{-1})$ and define $Z(\nfrak):=\hat{\iota}_K^{-1}(b h(\nfrak) H \cOhatF^\times h(\nfrak)^{-1} b^{-1})$ for all $\nfrak\in \allowed$. The decomposition of $H=H_S H^S$ of the beginning of §\ref{sec:Hecke-operators} induces a decomposition $Z=Z_S Z^S$ with $Z_S\subseteq \prod_{\ellfr\in S}\cO_{K_\ellfr}^\times$ and $Z^S=\prod_{v\notin S} Z_\ellfr$, with $Z_\ellfr=\iota_{K,\ellfr}^{-1}(b_\ellfr H_\ellfr \widehat{\cO}_{F_\ellfr}^\times b_\ellfr^{-1})$. 

The following proposition describes the Galois group $\Gal(K(x(\nfrak))/K)$. Let $N\ellfr:=\# (\cO_F/\ellfr)$ and define
\begin{equation*}
    \varepsilon(\ellfr)=\begin{cases}
        1 &\text{if $\ellfr$ splits in $K$;}\\
        -1 &\text{if $\ellfr$ is inert in $K$;}\\
        0 &\text{if $\ellfr$ is ramified in $K$,}
    \end{cases}
\end{equation*}
for every ideal $\ellfr$ of $\cO_F$.

\begin{proposition}\label{prop:computation-quotient-Zs}
    Let $\nfrak\in \allowed$ and write $\alpha_\ellfr:=\ord_{\ellfr}(\nfrak)$, for any prime $\ellfr$ of $F$.

\begin{enumerate}[label=(\roman*)]
    \item The reciprocity map induces an isomorphism
    \[
    \Khat^\times/K^\times \Fhat^\times Z(\nfrak) \iso \Gal(K(x(\nfrak))/K).
    \]
    Moreover, $Z/Z(\nfrak) \iso \prod_{\ellfr\mid \nfrak} Z_\ellfr/Z(\nfrak)_\ellfr$.
    \item If $\alpha_\ellfr >0$, the group $Z_\ellfr/Z(\nfrak)_\ellfr$ is the product of a cyclic group of order $N\ellfr-\varepsilon(\ellfr)$ with a group of order $(N\ellfr)^{\alpha_\ellfr-1}$, isomorphic to $(\Z/\ell^{\alpha_\ellfr-1}\Z)^{f_\ellfr}$, where $(\ell)=\ellfr\cap\Q$ and $f_\ellfr$ is the inertia degree of $\ellfr$ in $F/\Q$.
    \item $Z(\nfrak) = Z \cap \cOhat_\nfrak^\times$ and $K(x(\nfrak)) \subseteq K[\mathfrak{c}(x)\nfrak]$. 
\end{enumerate}
\end{proposition}
\begin{proof}
    (i) The statement about the reciprocity map is a consequence of \cite[Proposition 2.5]{nekovar:euler-system-method}. By the definition of $h(\nfrak)$, there is the decomposition
    \begin{equation}\label{eq:decomposition-Z-proof}
        Z(\nfrak) = Z_S \prod_{\ellfr \mid \nfrak} Z(\nfrak)_\ellfr \prod_{\ellfr \notin S, \ellfr \nmid \nfrak} Z_{\ellfr},
    \end{equation}
    with $Z(\nfrak)_\ellfr =\iota_{K,\ellfr}^{-1}(b_\ellfr h(\ellfr)^{\alpha_\ellfr}H \cOhatF^\times h(\ellfr)^{-\alpha_\ellfr}b_\ellfr^{-1})$, for any $\ellfr \mid \nfrak$. Therefore, we obtain that $Z/Z(\nfrak) \cong \prod_{\ellfr\mid \nfrak} Z_\ellfr/Z(\nfrak)_\ellfr$. 
    
    (ii) In the following, let 
    \[i := \eta_\ellfr\circ\mathrm{Ad}(b_\ellfr)^{-1} \circ \iota_{K_\ellfr}\colon K_\ellfr\hookrightarrow \M_2(F_\ellfr).\]
    Since $Z\supseteq\cOhat^\times_{\cond(x)}$ and $\ellfr\nmid \cond(x)$, we have that $i^{-1}(\GL_2(\cO_{F_\ellfr}))=Z_\ellfr=\cO_{K_\ellfr}^\times$.
    
    Suppose that $\ellfr\mid\mathfrak{n}$ is inert in $K$, for that $K_\ellfr$ is a field. Applying \cite[Proposition 4.7.(iv)]{nekovar:euler-system-method} (with $E=F_\ellfr$, $E'=K_\ellfr$, $R=\M_2(\cO_{F_\ellfr})$, $g=\eta_\ellfr(h(\ellfr))\in \M_2(F_\ellfr)$ and $r=\alpha_\ellfr>0$, in the notation of \emph{loc.~cit.}), the embedding $i$ induces an isomorphism
        \[
        \frac{Z_\ellfr}{Z(\nfrak)_\ellfr}=
        \frac{i^{-1}(\GL_2(\cO_{F_\ellfr}))}{i^{-1}(\eta_\ellfr(h(\ellfr))^{\alpha_\ellfr} \GL_2(\cO_{F_\ellfr}) \eta_\ellfr(h(\ellfr))^{-\alpha_\ellfr})} = 
        \frac{\cO_{K_\ellfr}^\times}{\cO_{F_\ellfr}^\times(1+\varpi_\ellfr^{\alpha_\ellfr}\cO_{K_\ellfr})} \cong \frac{(\cOK/\ellfr^{\alpha_\ellfr}\cOK)^\times}{(\cOF/\ellfr^{\alpha_\ellfr}\cOF)^\times},
        \]
    where the first equality holds by definition of $i$. The structure of the last group appearing in the equation coincides with the claimed structure in (ii).

    Suppose that $\ellfr\mid\nfrak$ is split in $K$, for that $K_\ellfr\cong F_\ellfr\times F_\ellfr$. By our choice of the isomorphism $\eta_\ellfr$, the embedding $i$ has the shape $i(a,b)=\left(\begin{smallmatrix}
        a&b-a\\0&b
    \end{smallmatrix}\right)$ for all $(a,b)\in K_\ellfr$. Therefore, by making some explicit computations, one can check that $Z_\ellfr = i^{-1}(\GL_2(\cO_{F_\ellfr})) = \cO_{F_\ellfr}^\times \times \cO_{F_\ellfr}^\times$ and that $Z(\nfrak)_\ellfr = i^{-1}(\eta_\ellfr(h(\ellfr))^{\alpha_\ellfr} \GL_2(\cO_{F_\ellfr}) \eta_\ellfr(h(\ellfr))^{-\alpha_\ellfr}) = \cO_{F_\ellfr}^\times (1, 1) + \cO_{F_\ellfr}(0, \varpi_\ellfr^{\alpha_\ellfr})$.
   Therefore,
    \[
        \frac{Z_\ellfr}{Z(\nfrak)_\ellfr}=\frac{i^{-1}(\GL_2(\cO_{F_\ellfr}))}{i^{-1}(\eta_\ellfr(h(\ellfr))^{\alpha_\ellfr} \GL_2(\cO_{F_\ellfr}) \eta_\ellfr(h(\ellfr))^{-\alpha_\ellfr})} = \frac{\cO_{F_\ellfr}^\times\times\cO_{F_\ellfr}^\times}{\cO_{F_\ellfr}^\times (1, 1) + \cO_{F_\ellfr}(0, \varpi_\ellfr^{\alpha_\ellfr})} \cong (\cO_{F}/\ellfr^{\alpha_\ellfr})^\times.
        \]
    The claim follows, as the last group has the structure prescribed in (ii).
    
    (iii) In both the inert and the split case, in the proof of point (ii) we also showed that
    $Z(\nfrak)_\ellfr=Z_\ellfr\cap(\cO_{\nfrak})_\ellfr^\times$ for every $\ellfr\mid\nfrak$. Therefore, the decomposition \eqref{eq:decomposition-Z-proof} implies that $Z(\nfrak)=Z\cap \widehat{\cO}_{\nfrak}^\times$. Since $Z\supseteq \widehat{\cO}^\times_{\mathfrak{c}(x)}$, then $Z(\nfrak)\supseteq \widehat{\cO}^\times_{\mathfrak{c}(x)}\cap\widehat{\cO}^\times_{\nfrak}=\widehat{\cO}^\times_{\mathfrak{c}(x)\nfrak}$, therefore $K(x(\nfrak)) \subseteq K[\mathfrak{c}(x)\nfrak]$.
\end{proof}

\begin{definition}
    Let $u\colon \N\to\N$ be the function defined by
    \[
    u(r) =\begin{cases}
        [K^\times\cap\Fhat^\times Z : F^\times] \quad &\text{if $r=0$};\\
        1   \qquad &\text{if $r>0$}.
    \end{cases}
    \]
\end{definition}

\begin{corollary}\label{cor:degree-x-n-ells}
    Let $\nfrak\in \allowed_r$, $\ellfr\in \allowedprimes$ and let $\alpha_\ellfr:=\ord_{\ellfr}(\nfrak)$.
    \begin{enumerate}[label=(\roman*)]
        \item If $\alpha_\ellfr=0$, then $K(x(\nfrak\ellfr))/K(x(\nfrak))$ is cyclic of degree $(N\ellfr-\varepsilon(\ellfr))/u(r)$.
        \item If $\alpha_\ellfr>0$, then $K(x(\nfrak\ellfr^n))/K(x(\nfrak))$ has degree $N\ellfr^n$ and is isomorphic to $(\Z/\ell^n\Z)^{f_\ellfr}$ for all $n\ge 1$.
        \item For any $\mfrak\in\allowed_s$ with $s\ge 0$ and $\mfrak\subseteq\nfrak$, we have that
        \begin{equation*}
            [K(x(\nfrak)):K(x(\mfrak))]=\prod_{\substack{\ellfr\mid\nfrak \\ \ellfr\nmid\mfrak}}\frac{N\ellfr-\varepsilon(\ellfr)}{u(s)}\prod_{\ellfr\mid(\nfrak/\mfrak)}N\ellfr^{\ord_\ellfr(\nfrak)-\max\{1,\ord_\ellfr(\mfrak)\}}.
        \end{equation*}
    \end{enumerate}
\end{corollary}
\begin{proof}
    The proof of the first two points follows the same lines of the proof of \cite[Proposition 4.8.(i)]{nekovar:euler-system-method}, using Proposition \ref{prop:computation-quotient-Zs} in place of \cite[Proposition 4.5]{nekovar:euler-system-method}.
     The third point follows directly from the first two.
\end{proof}

We now show that the CM points of Definition \ref{def:heegner-points-shimura} satisfy some compatibility properties with respect to the action of the Galois trace and of the Hecke operators. In order to do this, from now on we identify $N_H^\ast$ with its image in $\Jac(N_H^\ast)$ using the map defined in \cite[(1.19)]{nekovar:chow-groups} (and called $\iota$ therein).

\begin{proposition}\label{prop:norm-relations-cm}
    Let $\nfrak\in \allowed_r$ and $\ellfr\in \allowedprimes$ such that $\ellfr\nmid\nfrak$. Then the following equalities
    \begin{gather}
   u(r) \Tr_{K(x(\nfrak\ellfr))/K(x(\nfrak))} x(\nfrak\ellfr) = 
   \begin{cases}
       T_\ellfr x(\nfrak) & \text{if $\ellfr$ is inert in $K$;}\\
       (T_\ellfr  - \Frob_\lambda -\Frob_{\bar{\lambda}})x(\nfrak) \quad& \text{if $\ellfr$ splits in $K$;}
   \end{cases}\label{eq:norm-basic-cm}\\
   \Tr_{K(x(\nfrak\ellfr^{\alpha+1}))/K(x(\nfrak\ellfr^\alpha))} x(\nfrak\ellfr^{\alpha+1}) = T_\ellfr x(\nfrak \ellfr^\alpha) - x(\nfrak\ellfr^{\alpha-1}), \qquad \forall \alpha\ge1\label{eq:norm-higher-cm}
   \end{gather}
   hold in $\Jac(N_H^\ast)(K^\ab)$, where $\lambda$ and $\bar{\lambda}$ are the two primes of $K$ occurring in the decomposition $\ellfr\cO_K=\lambda\bar{\lambda}$, when this is the case.
\end{proposition}
\begin{proof}
    We'll prove these formulas as divisors on $N_H^\ast \otimes_F \bar{\Q}$; the same argument of \cite[proof of Proposition 4.8]{nekovar:euler-system-method} shows that the same equalities descend to $\Jac(N_H^\ast)$. Moreover, even though the inert case of \eqref{eq:norm-basic-cm} is already proven in \emph{loc.~cit.}, we give here different (and more explicit) proof, for the convenience of the reader.

     It follows from \cite[(2.7.3) and Proposition (2.10)]{nekovar:euler-system-method} that the kernel of the natural surjection 
    \[
    \frac{Z(\nfrak)}{Z(\nfrak\ellfr)} \surj \Gal\biggl(\frac{K(x(\nfrak\ellfr))}{K(x(\nfrak))}\biggr)
    \]
    has cardinality $u(r)$. Combining this with the action of the reciprocity map \eqref{eq:action-of-reciprocity-CM-points} we obtain that
    \begin{equation}\label{eq:proof-norm-relation}
        u(r) \Tr_{K(x(\nfrak\ellfr))/K(x(\nfrak))} x(\nfrak\ellfr) =
    \sum_{[a] \in \frac{Z(\nfrak)}{Z(\nfrak\ellfr)}} \hspace{-5pt} \rec_K(a) \cdot [(z, bh(\nfrak\ellfr))] =
    \sum_{[a] \in \frac{Z(\nfrak)}{Z(\nfrak\ellfr)}}  [(z, \hat{\iota}_K(a)bh(\nfrak\ellfr))].
    \end{equation}
    Since $Z(\nfrak)/Z(\nfrak\ellfr) \cong Z(\nfrak)_\ellfr/Z(\nfrak\ellfr)_\ellfr = Z_\ellfr/Z(\ellfr)_\ellfr$ we may chose the $a$'s to have trivial component outside $\ellfr$; thus in order to conclude the computation we need only to study the $\ellfr$-components of the points in the sum. As done in the proof of Proposition \ref{prop:computation-quotient-Zs}, let  
    \[i := \eta_\ellfr\circ\mathrm{Ad}(b_\ellfr)^{-1} \circ \iota_{K,\ellfr}\colon K_\ellfr\hookrightarrow \M_2(F_\ellfr).\]
    We also identify $B_\ellfr$ with $\M_2(F_\ellfr)$ via the isomorphism $\eta_\ellfr$, dropping it from the notation for simplicity.
    
    Assume that $\ellfr$ is inert in $K$. Applying \cite[Proposition 4.7, (iii) and (iv)]{nekovar:euler-system-method} (with $E=F_\ellfr$, $E'=K_\ellfr$, $R=\M_2(\cO_{F_\ellfr})$, $g=h(\ellfr)\in \M_2(F_\ellfr)$ and $r=1$ in the notation of \emph{loc.~cit.}), the embedding $i$ induces a set theoretical bijection
    \begin{equation*}
        \frac{Z_\ellfr}{Z(\nfrak\ellfr)_{\ellfr}} = \frac{i^{-1}(\GL_2(\cO_{F_\ellfr}))}{i^{-1}(h(\ellfr) \GL_2(\cO_{F_\ellfr}) h(\ellfr)^{-1})} \xrightarrow{1:1} \frac{\GL_2(\cO_{F_\ellfr})}{U^0(\varpi_\ellfr)},
    \end{equation*}
    where $U^0(\varpi_\ellfr)=\GL_2(\cO_{F_\ellfr})\cap h(\ellfr)\GL_2(\cO_{F_\ellfr})h(\ellfr)^{-1}=\left\{\left(\begin{smallmatrix}
            a &\varpi_\ellfr b\\ c&d
        \end{smallmatrix}\right) \ | \ a,b,c,d\in\cO_{F_\ellfr} \right\}\cap\GL_2(\cO_{F_\ellfr})$. The quotient on the right is represented by the matrices $\bigl(\begin{smallmatrix}
    0 & 1\\
    1 & 0
\end{smallmatrix}\bigr)$ and
$\bigl(\begin{smallmatrix}
    1 & j\\
    0 & 1
\end{smallmatrix}\bigr)$,
with $j\in\cO_F$ varying among the representatives of $\cO_F/\ellfr$. Thus the $\ellfr$-components of the second entry of the points in the sum appearing in \eqref{eq:proof-norm-relation} is given by
\[
b_\ellfr  
\bigl(\begin{smallmatrix}
    1 & j\\
    0 & 1
\end{smallmatrix}\bigr)b_\ellfr^{-1} b_\ellfr h(\ellfr) = b_\ellfr 
\bigl(\begin{smallmatrix}
    \varpi_\ellfr & j\\
    0 & 1
\end{smallmatrix}\bigr)
\]
for any $j$ and by
\[
b_\ellfr  
\bigl(\begin{smallmatrix}
    0 & 1\\
    1 & 0
\end{smallmatrix}\bigr)b_\ellfr^{-1} b_\ellfr h(\ellfr) = b_\ellfr 
\bigl(\begin{smallmatrix}
    0 & 1\\
    \varpi_\ellfr & 0
\end{smallmatrix}\bigr)
\]
By the explicit definition of $T_\ellfr$ of \eqref{eq:T-ellfr-formula}, the formula \eqref{eq:norm-basic-cm} in the inert case follows.

Assume now that $\ellfr$ is split in $K$. Recalling the computations in the proof of Proposition \ref{prop:computation-quotient-Zs} and observing that one has, again by the explicit description of $i$, that $Z(\nfrak\ell)_\ellfr = i^{-1}(U^0(\varpi_\ellfr))$, we see that
$i$ induces a set theoretical injection
\begin{equation*}
    \frac{Z_\ellfr}{Z(\nfrak\ellfr)_\ellfr}=\frac{\cO_{F_\ellfr}^\times\times\cO_{F_\ellfr}^\times}{\cO_{F_\ellfr}^\times(1,1)+\cO_{F_\ellfr}(0,\varpi_\ellfr)}\longinj \frac{\GL_2(\cO_{F_\ellfr})}{U^0(\varpi_\ellfr)}.
\end{equation*}
Indeed, in this case the source, being isomorphic to $(\cO_{F_\ellfr}/\ellfr)^\times$, has cardinality $N\ellfr -1$, while the target $N\ellfr+1$. Moreover, one can check that the only two classes not lying in the image of this injection are those represented by the matrices $\bigl(\begin{smallmatrix}1&1\\0&1\end{smallmatrix}\bigr)$ and $\bigl(\begin{smallmatrix}0&1\\1&0\end{smallmatrix}\bigr)$. 
Thus, repeating the computation above for the inert case and comparing with \eqref{eq:T-ellfr-formula}, it turns out that the sum \eqref{eq:proof-norm-relation} equals
\[
T_\ellfr x(\nfrak) - [(z, bh(\nfrak)\bigl(\begin{smallmatrix}\varpi_\ellfr&1\\0&1\end{smallmatrix}\bigr))] -
[(z, bh(\nfrak)\bigl(\begin{smallmatrix}1&0\\0&\varpi_\ellfr\end{smallmatrix}\bigr))]
\]
Since $\Frob_\lambda$ and $\Frob_{\bar{\lambda}}$ are respectively the image via $\rec_K$ of the idèle which has $\varpi_\ellfr \in F_\ellfr \cong K_\lambda, K_{\bar{\lambda}}$ at the $\lambda$ (resp.~$\bar{\lambda}$)-component and $1$ elsewhere, then
\[
\Frob_{\lambda} \cdot  x(\nfrak) = [(z, \iota_{K,\ellfr}(\varpi_\ellfr, 1)bh(\nfrak))] \qquad \text{and} \qquad 
\Frob_{\bar{\lambda}} \cdot x(\nfrak) = [(z, \iota_{K,\ellfr}(1,\varpi_\ellfr)bh(\nfrak))]
\]
Moreover, the $\ellfr$-component of the second entry of these points is respectively
\[
\iota_{K,\ellfr}(\varpi_\ellfr, 1)b_\ellfr = b_\ellfr i(\varpi_\ellfr, 1) = b_\ellfr \bigl(\begin{smallmatrix}
    \varpi_\ellfr & 1 - \varpi_\ellfr\\0&1
\end{smallmatrix}\bigr) =
b_\ellfr \bigl(\begin{smallmatrix}
    \varpi_\ellfr & 1\\0&1
\end{smallmatrix}\bigr)
\bigl(\begin{smallmatrix}
    1 & -1\\0&1
\end{smallmatrix}\bigr)
\]
and
\[
\iota_{K,\ellfr}(1,\varpi_\ellfr)b_\ellfr = b_\ellfr i(1,\varpi_\ellfr) = b_\ellfr \bigl(\begin{smallmatrix}
    1&\varpi_\ellfr - 1\\0& \varpi_\ellfr
\end{smallmatrix}\bigr) =
b_\ellfr \bigl(\begin{smallmatrix}
    1&0\\0&\varpi_\ellfr
\end{smallmatrix}\bigr)
\bigl(\begin{smallmatrix}
    1 & \varpi_\ellfr-1\\0&1
\end{smallmatrix}\bigr).
\]
Since $\bigl(\begin{smallmatrix}
    1 & -1\\0&1
\end{smallmatrix}\bigr), \bigl(\begin{smallmatrix}
    1 & \varpi_\ellfr-1\\0&1
\end{smallmatrix}\bigr) \in \GL_2(F_\ellfr)$, it follows that
\[
\Frob_{\lambda} \cdot  x(\nfrak) = [(z, bh(\nfrak)\bigl(\begin{smallmatrix}\varpi_\ellfr&1\\0&1\end{smallmatrix}\bigr))] \qquad \text{and} \qquad \Frob_{\lambda} \cdot  x(\nfrak) = [(z, bh(\nfrak)\bigl(\begin{smallmatrix}1&0\\0&\varpi_\ellfr\end{smallmatrix}\bigr))],
\]
proving formula \eqref{eq:norm-basic-cm} in the split case.

Similarly to \eqref{eq:proof-norm-relation}, for any $\alpha \ge 1$ we obtain the formula
\begin{align}\label{eq:proof-norm-relation-2}
    \Tr_{K(x(\nfrak\ellfr^{\alpha+1}))/K(x(\nfrak\ellfr^\alpha))} x(\nfrak\ellfr^{\alpha+1}) = \sum_{[a] \in \frac{Z(\nfrak\ellfr^{\alpha})_\ellfr}{Z(\nfrak\ellfr^{\alpha+1})_\ellfr}} [(z, \hat{\iota}_K(a) b h(\nfrak\ellfr^{\alpha+1})].
\end{align}
Moreover, it follows by \cite[Proposition 4.7]{nekovar:euler-system-method} in the inert case and by an explicit computation in the split case as above (again one can check that $Z(\nfrak\ellfr^k)_\ellfr = i^{-1}(U^0(\varpi_\ellfr^k))$, for any $k \ge 1$.) that $i$ induces an bijection
\[
Z(\nfrak\ellfr^{\alpha})_\ellfr/Z(\nfrak\ellfr^{\alpha+1})_\ellfr \xrightarrow{1:1} U^0(\varpi_\ellfr^{\alpha})/U^0(\varpi_\ellfr^{\alpha+1})
\]
(the surjectivity can be checked counting the cardinality of both sides, applying Corollary \ref{cor:degree-x-n-ells}). 
The elements of the right hand side can be represented by the matrices
$\bigl(\begin{smallmatrix}
    1 & \varpi_\ellfr^{\alpha}j\\
    0 & 1
\end{smallmatrix}\bigr)$,
with $j\in\cO_F$ varying among the representatives of $\cO_F/\ellfr$. Thus the $\ellfr$ components of the second entry of the points appearing in the sum in \eqref{eq:proof-norm-relation-2} are given by 
\[
b_\ellfr  
\bigl(\begin{smallmatrix}
    1 & \varpi^\alpha j\\
    0 & 1
\end{smallmatrix}\bigr)b_\ellfr^{-1} b_\ellfr h(\ellfr)^{\alpha+1} = b_\ellfr h(\ellfr)^{\alpha}
\bigl(\begin{smallmatrix}
    \varpi_\ellfr & j\\
    0 & 1
\end{smallmatrix}\bigr).
\]
The explicit description of $T(\ellfr)$ proves that the sum in \eqref{eq:proof-norm-relation-2} equals
\[
T_\ellfr x(\nfrak\ellfr^{\alpha})- [(z, 
 b_\ellfr h(\ellfr^{\alpha})\bigl(\begin{smallmatrix}
    1&0\\0&\varpi_\ellfr
\end{smallmatrix}\bigr))].
\]
The relation 
\[
b_\ellfr h(\ellfr)^\alpha\bigl(\begin{smallmatrix}
    1&0\\0&\varpi_\ellfr
\end{smallmatrix}\bigr) = b_\ellfr h(\ellfr)^{\alpha-1} \cdot \varpi_\ellfr
\]
implies that $ [(z, 
b_\ellfr h(\ellfr^{\alpha})\bigl(\begin{smallmatrix}
    1&0\\0&\varpi_\ellfr
\end{smallmatrix}\bigr)] = x(\nfrak\ellfr^{\alpha-1})$, therefore we obtain \eqref{eq:norm-higher-cm}.
\end{proof}

\subsection{Modular abelian varieties}
As explained in \cite[(1.17)]{nekovar:euler-system-method}, there is a finite decomposition 
\begin{equation*}
    \mathfrak{h}(H)\otimes_\Z\Q\cong\prod_{j\in J} E_j,
\end{equation*}
where the $E_j$ are totally real fields. For each $j\in J$, denote by $\theta_j$ the natural morphism
\begin{equation*}
    \theta_j\colon \mathfrak{h}(H)\to\mathfrak{h}(H)\otimes_\Z\Q\to E_j.
\end{equation*}

\begin{proposition}[Nekov{\'a}{\v{r}}]\label{prop:nekovar-structure-jacobian}
    Let $\ellfr\notin S$ be a prime of $F$.
    \begin{enumerate}[label=(\roman*)]
        \item  There exists a $\mathfrak{h}(H)$-linear isogeny $\Jac(N_H^\ast)\to\prod_{j\in J} A_j^{b_j}$ defined over $F$, where, for any $j \in J$, $b_j\ge 1$ and $A_j$ are $F$-simple abelian varieties over $F$ of dimension $[E_j:\Q]$ satisfying $\End_F(A_j)=\cO_{E_j}$, on which $\mathfrak{h}(H)$ acts via $\theta_j$.
        \item The abelian varieties $A_j$ are unique up to an $\cO_{E_j}$-linear isogeny and, for each polarization of $A_j$ defined over $F$, the corresponding Rosati involution acts trivially on $E_j=\End_F^0(A_j)$.
        \item The trace of the action of $\Frob_\ellfr$ on $T_\P A_j$ coincides with $\theta_j(T_\ellfr)$.
    \end{enumerate}
    
\end{proposition}
\begin{proof}
    It follows from \cite[Proposition 1.18]{nekovar:euler-system-method}.
\end{proof}

\begin{remark}
    Note that \cite[Proposition 1.18(iii)]{nekovar:euler-system-method} shows also that there is a Hilbert modular form $f$ of parallel weight $(2, \dots, 2)$ such that $V_\P A := T_\P A \otimes \Q_p$ is (the dual of) the representation attached to $f$.
\end{remark}

Until the end of Section \ref{sec:modular-av-and-heegner}, fix an abelian variety $A$ of dimension $g$ defined over $F$ and assume that $A$ satisfies the following assumption:
\leqnomode
\begin{equation}\label{ass:modular-av}\tag{Mod}
\begin{gathered}
  \text{$A$ is one of the $A_j$'s coming from Proposition \ref{prop:nekovar-structure-jacobian}, for a quaternion algebra $\B$}\\
  \text{ramified at $S_\B \cup \set{\tau \mid \infty : \tau \ne \tau_1 }$ and a compact open subgroup $H$ of $\Bhat^\times$.}
\end{gathered}
\end{equation}
\reqnomode
By Proposition \ref{prop:nekovar-structure-jacobian}, $A$ satisfies \eqref{ass:GL-2}, and let $E = \End^0_F(A)$. Moreover, \eqref{ass:O-E-lin} is satisfied by any polarization $\gamma$ of $A$, since the Rosati involution associated with $\gamma$ acts trivially on $E$. Let us denote by
\begin{equation*}
    \phi\colon \Jac(N_H^\ast)\surj A
\end{equation*}
the surjection given by Proposition \ref{prop:nekovar-structure-jacobian}, which is Hecke and $G_F$-equivariant. By Proposition \ref{prop:nekovar-structure-jacobian}.(iii), we also have that, for every $\ellfr\notin S$, the element $a_\ellfr:=\Tr(\Frob_\ellfr| T_\P A)$ coincides with the action of $T_\ellfr$ on $A$, i.e. with $\theta_j(T_\ellfr)$.

We now define a class of Heegner points $y(\nfrak)$ on $A$, projecting  
the CM points $x(\nfrak)$ defined in the previous section (seen into $\Jac(N_H^\ast)$ through the morphism of \cite[§1.19]{nekovar:euler-system-method}) onto $A$ via $\phi$.
\begin{definition}
    For each $\nfrak\in\allowed$, we define the subfield $K(x(\nfrak))'\subseteq K(x(\nfrak))$ by
    \begin{equation*}
        K(x(\nfrak))'=
        \begin{cases}
            K(x)&\text{if $\nfrak=(1)$};\\
            K(x(\ellfr_1^{\alpha_1}))\cdots K(x(\ellfr_r^{\alpha_r})) &\text{if $\nfrak=\ellfr_1^{\alpha_1}\cdots\ellfr_r^{\alpha_r}$ for $r\ge 1$ and $\ellfr_i\in\tilde{\admissible_1}$}
        \end{cases}
    \end{equation*}
    and put $G(\nfrak)=\Gal(K(x(\nfrak))'/K(x))$.
\end{definition}
\begin{lemma}\label{lemma:degree-prime-extensions}
    For any $\nfrak=\prod_{i=1}^r \ellfr_i^{\alpha_i} \in \allowed_r$ (for $r \ge 0$), we have that 
    \begin{enumerate}[label=(\roman*)]
        \item the canonical map $G(\nfrak)\to \prod_{i=1}^r G(\ellfr_i^{\alpha_i})$ is an isomorphism;
        \item $[K(x(\nfrak)):K(x(\nfrak))'] = u(r)u(0)^{r-1}$.
    \end{enumerate}
\end{lemma}
\begin{proof}
    The case $r=0$ is trivial, so let $r\ge 1$. Point (i) is showed exactly as in \cite[Proposition 4.10.(ii)]{nekovar:euler-system-method}. From this, it descends that $\#G(\nfrak) = u(0)^{-r}\prod_{i=1}^r (N\ellfr_i-\varepsilon(\ellfr_i))(N\ellfr_i)^{\alpha_i-1}$, as follows from Corollary \ref{cor:degree-x-n-ells}. Moreover, we have $\#\Gal\bigl(K(x(\nfrak))/K(x)\bigr) = u(0)^{-1}\prod_{i=1}^r (N\ellfr_i-\varepsilon(\ellfr_i))(N\ellfr_i)^{\alpha_i-1}$, again by Corollary \ref{cor:degree-x-n-ells}. The result follows.
\end{proof}

\begin{definition}\label{def:heegner-points-av}
    For each $r\ge 0$ and $\nfrak\in\allowed_r$, define
    \begin{equation*}
        y(\nfrak)=\frac{u(0)}{u(r)}\Tr_{K(x(\nfrak))/K(x(\nfrak))'}\phi(x(\nfrak))\in A(K(x(\nfrak))').
    \end{equation*}
\end{definition}

\begin{proposition}\label{prop:norm-relations-heegner}
   Let $\nfrak \in \allowed_r$ and $\ellfr \in \allowedprimes$ such that $\ellfr\nmid\nfrak$ and, if $\ellfr$ splits in $K$, let $\lambda$ and $\bar{\lambda}$ be the two primes of $K$ such that $\ellfr\cO_K=\lambda\bar{\lambda}$. Then the following equalities hold in $A(K^\ab)$
   \begin{gather}
   \Tr_{K(x(\nfrak\ellfr))'/K(x(\nfrak))'} y(\nfrak\ellfr) = 
   \begin{cases}
       a_\ellfr y(\nfrak) & \text{if $\ellfr$ is inert in $K$;}\\
       (a_\ellfr  - \Frob_\lambda -\Frob_{\bar{\lambda}})y(\nfrak) \quad& \text{if $\ellfr$ splits in $K$;}
   \end{cases}\label{eq:norm-basic-heegner-y}\\
   \Tr_{K(x(\nfrak\ellfr^{\alpha+1}))'/K(x(\nfrak\ellfr^\alpha))'} y(\nfrak\ellfr^{\alpha+1}) = a_\ellfr y(\nfrak \ellfr^\alpha) - u(\alpha-1)y(\nfrak\ellfr^{\alpha-1}), \qquad \forall \alpha\ge1.\label{eq:norm-higher-heegner-y}
   \end{gather}
\end{proposition}

\begin{proof}
    Equation \eqref{eq:norm-basic-heegner-y} is a straightforward consequence of Proposition \ref{prop:norm-relations-cm} together with the Hecke and Galois equivariance of the map $\phi$.

    Let's move to the proof of \eqref{eq:norm-higher-heegner-y}. As above, applying Proposition \ref{prop:norm-relations-cm} together with the Hecke and Galois equivariance of the map $\phi$, we obtain the relation
    \begin{equation}\label{eq:middle-proof-norm-relations-y}
        \Tr_{K(x(\nfrak\ellfr^{\alpha+1}))'/K(x(\nfrak\ellfr^\alpha))'} y(\nfrak\ellfr^{\alpha+1}) = a_\ellfr y(\nfrak \ellfr^\alpha) - u(0)\Tr_{K(x(\nfrak\ellfr^\alpha))/K(x(\nfrak\ellfr^\alpha))'} \phi(x(\nfrak\ellfr^{\alpha-1})).
    \end{equation}
    Moreover, there is the following diagram of field extensions
    \begin{equation}\label{eq:diagram-extensions-morm-relations}
        \begin{tikzcd}
	& {K(x(\nfrak\ellfr^\alpha))} & \\
	{K(x(\nfrak\ellfr^{\alpha-1}))} && {K(x(\nfrak\ellfr^\alpha))'} \\
	& {K(x(\nfrak\ellfr^{\alpha-1}))'}
	\arrow[no head, from=1-2, to=2-3]
	\arrow[no head, from=2-1, to=1-2]
	\arrow[no head, from=2-1, to=3-2]
	\arrow[no head, from=3-2, to=2-3]
\end{tikzcd}
    \end{equation}
     If $\alpha>1$, by Corollary \ref{cor:degree-x-n-ells} and Lemma \ref{lemma:degree-prime-extensions} the upper left and lower right field extensions have degree $N\ellfr$, and the remaining two have degree $u(0)^r$. Since $u(0)\mid N\ellfr-\varepsilon(\ellfr)$, we have that $u(0)^r$ and $N\ellfr$ are coprime, thus it follows that $K(x(\nfrak\ellfr^{\alpha-1}))' = K(x(\nfrak\ellfr^{\alpha-1}))\cap K(x(\nfrak\ellfr^\alpha))'$ and $K(x(\nfrak\ellfr^\alpha))=K(x(\nfrak\ellfr^{\alpha-1}))K(x(\nfrak\ellfr^\alpha))'$. This implies the equality 
     \begin{equation*}
         \Tr_{K(x(\nfrak\ellfr^\alpha))/K(x(\nfrak\ellfr^\alpha))'} \phi(x(\nfrak\ellfr^{\alpha-1}))=\Tr_{K(x(\nfrak\ellfr^{\alpha-1}))/K(x(\nfrak\ellfr^{\alpha-1}))'} \phi(x(\nfrak\ellfr^{\alpha-1})),
     \end{equation*}
     therefore relation \eqref{eq:middle-proof-norm-relations-y} yields \eqref{eq:norm-higher-heegner-y} when $\alpha>1$.

     Assume now that $\alpha=1$. In this case, the degrees of the four extensions in \eqref{eq:diagram-extensions-morm-relations}, read in lexicographic order, are $(N\ellfr-\varepsilon(\ellfr))/u(r)$, $u(0)^r$, $u(r)u(0)^{r-1}$ and $(N\ellfr-\varepsilon(\ellfr))/u(0)$. When $r=0,1$ we have that $K(x(\nfrak))=K(x(\nfrak))'$, therefore we can explicitly compute the second summand of \eqref{eq:middle-proof-norm-relations-y}, obtaining \eqref{eq:norm-higher-heegner-y}. 

     Assume now that $r\ge 2$. We divide the proof in two steps.

     \textbf{Step 1:} $K(x(\nfrak))\cap K(x(\nfrak\ellfr))'=K(x(\nfrak))'$. Let $\tilde{\ellfr}$ be a prime of $K(x(\nfrak))'$ above $\ellfr$, and observe that $\tilde{\ellfr}$ is unramified in $K(x(\nfrak))$. Therefore, it is enough to prove that $\tilde{\ellfr}$ is totally ramified in $K(x(\nfrak\ellfr))'$. Thanks to the isomorphism of Lemma \ref{lemma:degree-prime-extensions}.(i), this is equivalent to proving that $\ellfr_1:=\tilde{\ellfr}\cap K(x)$ is totally ramified in $K(x(\ellfr))$. When $\ellfr$ is inert in $K/F$, this is \cite[Proposition 4.6.(iii)]{nekovar:euler-system-method}. Similarly to the proof contained therein, when $\ellfr$ is split in $K/F$, the inertia group at $\ellfr_1$ of $K(x(\ellfr))/K(x)$ is the image of the composite map 
     \begin{equation*}
         \begin{tikzcd}
	{\cO_{K(x)_{\ellfr_1}}^\times} & {\cO_{K_\lambda}^\times} & {\frac{\cO_{K_{\bar{\lambda}}}^\times\times\cO_{K_\lambda}^\times}{\cO_{F_\ellfr}^\times(1,1)+\lambda\cO_{K_\lambda}}=\frac{Z_\ellfr}{Z(\ellfr)_\ellfr}} & {\Gal(K(x(\ellfr))/K(x)),}
	\arrow["{f_1}", from=1-1, to=1-2]
	\arrow["{f_2}", from=1-2, to=1-3]
	\arrow["{f_3}", from=1-3, to=1-4]
\end{tikzcd}
     \end{equation*}
     where $\lambda=\ellfr_1\cap K$, $\bar{\lambda}$ is the conjugate of $\lambda$, $f_1$ is the norm map, $f_2$ is the natural projection and $f_3$ is the projection induced by \cite[(2.7.3)]{nekovar:euler-system-method}. All these maps are surjective, hence $\ellfr_1$ is totally ramified in $K(x(\ellfr))$.

     \textbf{Step 2:} 
    Let now $F:=K(x(\nfrak))K(x(\nfrak\ellfr))'\subseteq K(x(\nfrak\ellfr))$. Thanks to step 1, a degree computation yields that $[K(x(\nfrak\ellfr)):F]=u(0)$. Therefore, we have that
     \begin{equation*}
         \Tr_{K(x(\nfrak\ellfr))/K(x(\nfrak\ellfr))'} x(\nfrak)=u(0)\Tr_{F/K(x(\nfrak\ellfr))'} x(\nfrak)=u(0)\Tr_{K(x(\nfrak))/K(x(\nfrak))'} x(\nfrak),
     \end{equation*}
     applying again Step 1 for the last equality.
     Plugging this relation into \eqref{eq:middle-proof-norm-relations-y}, we obtain the claimed equality \eqref{eq:norm-higher-heegner-y} in the remaining cases.
\end{proof}

\subsection{Kolyvagin classes}

From now until the end of the next section we set the ground for the proof of our first main theorem. A crucial role will be played by Heegner points with inert conductor, for this is why we introduce a bit more notation.

Let $\P$ be a prime of $E$ above $p$. For this subsection, assume the following hypothesis:
\leqnomode
\begin{equation}\label{ass:u-0}
    p\nmid u(0),\tag{$u(0)$ $p$-ndiv} 
\end{equation}
\reqnomode
    that is satisfied, for example, whenever $p\nmid [\cO_K^\times:\cO_F^\times]$. Following \cite[Section 5]{nekovar:euler-system-method}, we define the set of \emph{Kolyvagin primes}.
\begin{definition}
    Let 
    \[\admissible_1 = \set{ \text{$\ellfr$ prime of $F$} \, : \, \text{$\ellfr$ is inert in $K/F$}, \, \ellfr \notin S, \,  \ellfr \nmid (p)\cond(x), \, \ellfr \cO_K \nmid \mathfrak{I}_0} \subseteq \allowedprimes;\]
    for every $M \ge 1$, let $\admissible_1(M)$ be the set of primes $\ellfr \in\admissible_1$ such that the conjugacy class of $\Frob_{\ellfr}$ in $\Gal\bigl(K(x)(A[\P^M])/F\bigr)$ coincides with the conjugacy class of the complex conjugation $\tau_c$. 
    
    Moreover, let $\admissible_r(M)$ be the set of squarefree products of length $r$ of ideals in $\admissible_1(M)$ (with the convention that $\admissible_0(M) = \{(1)\}$) and set $\admissible(M) = \bigcup_{r \ge 0} \admissible_r(M)$.
\end{definition}
\begin{remark}
    In particular, if $\ellfr\in\admissible_1(M)$, then $a_\ellfr \equiv N(\ellfr) + 1 \equiv 0 \bmod {\P^M}$, where $a_\mathfrak{l}\in\cO_E$ is the trace of the action of $\Frob_\ellfr$ on $T_\P A$, which coincides with the action of the Hecke operator $T_\ellfr$ on $A$ thanks to Proposition \ref{prop:nekovar-structure-jacobian}(iii).
\end{remark}

For every $\ellfr \in \admissible_1(M)$, the group $G(\ellfr) = \Gal\bigl(K(x(\ellfr))/K(x)\bigr)$ is cyclic of order $(N(\ellfr)+1)/u(0)$ (see Corollary \ref{cor:degree-x-n-ells}), say generated by $\sigma_\ellfr$. The $\ellfr$-th Kolyvagin derivative and trace operator are defined as
\[
D_\ellfr = \sum_{i=1}^{|G(\ellfr)|-1} i \sigma_\ellfr^i, \qquad \Tr_\ellfr = \sum_{i=0}^{|G(\ellfr)|-1}\sigma_\ellfr^i,
\]
and they satisfy the telescopic identity $(\sigma_\ellfr -1)D_\ellfr = N(\ellfr)+1 - \Tr_\ellfr$. Let now $\ellfr_1,\dots,\ellfr_r\in\admissible_1(M)$ such that $\nfrak=\ellfr_1\cdots\ellfr_r \in \admissible_r(M)$. Thanks to Lemma \ref{lemma:degree-prime-extensions}, we know that $G(\nfrak)\cong\prod_{j=1}^r G(\ellfr_j)$ and we define the $\nfrak$-th Kolyvagin derivative operator as 
\[
D_\nfrak = D_{\ellfr_1} \cdots D_{\ellfr_r} \in \Z[G(\ellfr_1)] \otimes \dots \otimes\Z[G(\ellfr_r)] = \Z[G(\nfrak)].
\]
As shown in \cite[Lemma 5.5]{nekovar:euler-system-method}, the image of $D_\nfrak y(\nfrak)\in A(K(x(\nfrak))')$ in $A(K(x(\nfrak))')\otimes\cOP/\P^M$ is fixed by the action of $G(\nfrak)$. Let's now assume \eqref{ass:P-M-princ}. Then, the exact sequence \eqref{eq:ses-kummer-P-M} with $L=K(x(\nfrak))'$ yields an injective map
\begin{equation*}
    0\longrightarrow A(K(x(\nfrak))')/\P^M A(K(x(\nfrak))')\xrightarrow{\delta_{K(x(\nfrak))'}} \hone(K(x(\nfrak))', A[\P^M]).
\end{equation*}
{Nekov{\'a}{\v{r}}'s version of Kolyvagin's descent (see \cite[(5.10)]{nekovar:euler-system-method}) explicitely builds some cocycles $c_x(\nfrak)\in Z^1(K(x), A[\P^M])$ (called $c(\nfrak)$ in \emph{loc.~cit.}) with the property that the restriction of the cohomology class $\cbf_x(\nfrak):=[c_x(\nfrak)]\in \hone(K(x), A[\P^M])$ to $K(x(\nfrak))'$ coincides with $\delta_{K(x(\nfrak))'}(D_\nfrak y(\nfrak))$ modulo $\P^M$. 
Another output of {Nekov{\'a}{\v{r}}'s machinery are some explicit cocycles $d_x(\nfrak)\in Z^1(K(x), A)$ (called $d(\nfrak)$ in \emph{loc.~cit.}) with the property that the image of $\cbf_x(\nfrak)$ in $\hone(K(x), A)[\P^M]$ (via the second map of \eqref{eq:ses-kummer-P-M}) coincides with $\dbf_x(\nfrak):=[d_x(\nfrak)]$. 

\begin{definition}
    For every $\nfrak\in\admissible(M)$, define $\cbf(\nfrak)=\Cor^{K(x)}_{K}\cbf_x(\nfrak)\in \hone(K,A[\P^M])$ and $\dbf(\nfrak)=\Cor^{K(x)}_{K}\dbf_x(\nfrak)\in \hone(K,A)[\P^M]$.
\end{definition}

For every prime $v$ of $K$, let  $\tilde{A}_v$ is the special fibre at $v$ of the Néron model of $A \otimes K$, that is defined over the residue field $k(v)$ of $K$ at $v$, and $\pi_0(\tilde{A}_v) = \tilde{A}_v/\tilde{A}_v^0$ its group of components. For the rest of this section, we make the following assumption on the Tamagawa numbers of $A\otimes K$:
\leqnomode
\begin{equation}\label{ass:Tamagawa}\tag{Tama}
    \text{$p \nmid \# \pi_0(\tilde{A}_v)$, for every prime $v$ of $K$.}
\end{equation}
\reqnomode
\begin{remark}
    Since $\pi_0(\tilde{A}_v)=\set{1}$ for every prime $v$ of good reduction for $A \otimes K$, the previous assumption only excludes a finite number of primes $p$.
\end{remark}

\begin{proposition}\label{prop:kolyvaginclasses_properties}
    Let $\nfrak\ellfr\in\admissible(M)$ with $\ellfr\in\admissible_1(M)$ and let $v$ be a non-archimedean prime of $K$ that doesn't divide $\nfrak$. Then
    \begin{enumerate}[label=(\roman*)]
        \item $\dbf(\nfrak)_v=0$.\label{item:loc-d-at-v-trivial}
        \item if $v$ lies above $\ellfr$, there is a $\tau_c$-anti-equivariant $\cOP/\P^M$-linear isomorphism
        \begin{equation*}
            \psi_v\colon \hone(K_v, A)[\P^M]\xrightarrow{\sim} \honeur(K_v, A[\P^M])
        \end{equation*}
        such that $\psi_v(\dbf(\nfrak\ellfr)_v)=\cbf(\nfrak)_v$.\label{item:finite-singular}
    \end{enumerate}
\end{proposition}

\begin{proof}
    \emph{(i)}. By the definition of \cite[Proposition-Definition 5.9(iv)]{nekovar:euler-system-method}, $\dbf_x(\nfrak)$ is inflated by a class $\dbf'_x(\nfrak) \in \hone\bigl(G(\nfrak), A(K(x(\nfrak))')\bigr)$. For any prime $w \mid v$ of $K(x)$, choose a prime $\bar{w}$ of $\Qbar$ above $v$; by abuse of notation we denote by $w$ its restriction to any intermediate field. Using the inflation-restriction exact sequence for $K(x(\nfrak))'/K(x)$ and $K(x(\nfrak))_w'/K(x)_w$ and the compatibility of the restriction
    we find that $\dbf_x(\nfrak)_w \in \hone(K(x)_w, A)[\P^M]$ is inflated by a class of
    \[
    \hone\biggl(\frac{K(x(\nfrak))_w'}{K(x)_w}, A(K(x(\nfrak))_w')\biggr)[\P^M] \subseteq \honeur(K(x)_w, A)[\P^M].
    \]
    Indeed, the latter follows since $K(x(\nfrak))/K(x)$ is unramified at any place above above $w \nmid \nfrak$. Now allow $w$ vary over the places of $K(x)$ above $v$. The commutative diagram (see \cite[B.2]{mastella-zerman:anticyclotomic-euler-kolyvagin-systems})
    \[
    \begin{tikzcd}
        \hone(K(x), A)[\P^M] \ar[r, "\res"] \ar[d, "\cores^{K(x)}_K"]& \bigoplus_{w | v} \hone(K(x)_w, A)[\P^M] \ar[r, "\res"]\ar[d, "\Phi"]& \ar[d] \bigoplus_{w|v} \hone(K(x)_{w}^\ur, A)[\P^M] \ar[d]\\
        \hone(K, A)[\P^M] \ar[r, "\res"]& \hone(K_v, A)[\P^M] \ar[r, "\res"] & \hone(K_v^\ur, A)[\P^M]
    \end{tikzcd}
    \]
    expressing the commutativity (in a semilocal sense) of the corestriction with the localization implies
    that $\dbf(\nfrak)_v \in \honeur(K_v, A)[\P^M]$. This proves our claim, as 
    \[
    \honeur(K_v, A)[\P^M] \cong \hone(k(v), \pi_0(\tilde{A}_v))[\P^M] = 0
    \]
    by \cite[I.3.8]{milne:arithmetic-duality} and \eqref{ass:Tamagawa}.
    
    \emph{(ii)} For any $v \mid \ellfr$, $\psi_{v}$ is obtained as the composition the isomorphism of Remark \ref{remark:hones},  
    the isomorphism $\ev_{\sigma_\ellfr}$ defined in \cite[(5.15)]{nekovar:euler-system-method} (recall that, since $\ellfr \in \admissible_1(M)$, it splits completely in $K(x)/K$ and therefore $K_v = K(x)_w$ for any prime $w \mid v$ of $K(x)$), the action of $\Frob_\ellfr$ on $A[\P^M]$ and the inverse of the evaluation at $\Frob_v$ isomorphism
    \[
    \ev_{\Frob_v} \colon \honeur(K_v, A[\P^M]) \cong \Hom(\langle \Frob_v \rangle, A(K_v)[\P^M]) \iso A(K_v)[\P^M].
    \]
    Since our Frobenius element is the arthmetic one, while \emph{loc.~cit.}~uses the geometric Frobenius,  then $\psi_v = -\Frob_\ellfr \circ \Phi_v^{-1}$, where $\Phi_v$ is defined in \emph{loc.~cit.}; thus, it follows from Proposition 5.16 and (5.15.1) of \emph{loc.~cit.}~(together with the functoriality of the corestriction) $\psi_v(\dbf(\nfrak\ellfr)_v)=\cbf(\nfrak)_v$ and that $\psi_v$ is  $\tau_c$-anti-equivariant.
        \end{proof}

\section{Vanishing of $\Sha(A/K)[\P^\infty]$}\label{vanishingresult_section}

From now on, and for the rest of this section, fix an abelian variety $A$ of dimension $g$ defined over a totally real field $F$ of degree $d$ over $\Q$ satisfying \eqref{ass:modular-av}, a polarization $\gamma$ of $A$, a CM extension $K$ of $F$ satisfying \eqref{ass:heegner-hypothesis} and a prime number $p$ satifying
\leqnomode
\begin{equation}\label{ass:p-ndiv}\tag{$p$-ndiv}
    \text{$p \nmid 2[\cO_K^\times: \cO_F^\times]\deg(\gamma)$.}
\end{equation}
\reqnomode
Write $E=\End_F^0(A)$, fix a prime $\P$ of $E$ above $p$ and let $\cOP := \cOEP$. Moreover, fix an integer $M \ge 1$ satisfying \eqref{ass:P-M-princ}. 

Since \eqref{ass:modular-av} implies \eqref{ass:GL-2}, we may consider the $\P$-adic Tate module $T_\P A$ of $A$, which is naturally a $\cOP[G_F]$-module. Write 
\[
\bar{\rho}_\P := \rho_{\P, 1} \colon G_F \to \Aut_{\cO_E/\P}(A[\P]) \cong \GL_2(\cOE/\P)
\]
for the group homomorphism attached to its residual representation $A[\P]$ on which we assume the following conditions. 
\leqnomode
\begin{gather}
    \text{$A[\P]$ is an absolutely simple $(\cOP/\P)[G_K]$-module}.\label{ass:abs-G-K-irr}\tag{abs-$G_K$-irr}\\    
    \begin{gathered}
        \text{If $M>1$, $F \supseteq \Q(\mu_p)^+$ and $\cOE/\P \ne \F_3$, then}\\
        \text{$\bar{\rho}_\P(G_K) \not \cong D_{2n}$ for any $\cOE/\P$-exceptional odd integer $n$,}
    \end{gathered}\label{ass:image-residual}\tag{$\bar{\rho}_\p(G_K)$-exc}
\end{gather}
\reqnomode
where, $D_{2n}$ denotes the dihedral group of order $2n$; for the definition of $\cOE/\P$-exceptional integer see \cite[Definition 5.13]{matar-nekovar:kolyvagin}.
In particular, these assumptions imply the following results.
\begin{lemma}\label{torsion_lemma}
$A(K)[\fp^{n}]=0$, for any $n > 0$.
\end{lemma}
\begin{proof} 
Let $m=\dim_{\cOP/\fp}(A(K)[\fp]) \le 2$. If $m=2$, then $A(K)[\fp]=A[\fp]$, therefore the action of $G_K$ on $A[\P]$ is trivial, contradicting \eqref{ass:abs-G-K-irr}. If $m=1$, then $A(K)[\fp]$ is a nontrivial $\cOP/\fp$-subspace of $A[\fp]$ fixed by the action of $G_K$; this contradicts \eqref{ass:abs-G-K-irr} again. Therefore $m=0$, or $A(K)[\fp]=0$. The claim follows for any $n > 0$.
\end{proof}

\begin{lemma}\label{lemma:hone-vanishing}
    We have $\hone(K(A[\P^M])/K, A[\P^M])=0$.
\end{lemma}

\begin{proof}
     If $M=1$, consider $H = \bar{\rho}_\p(G_K) \cong \Gal(K(A[\P])/K)$. If $H$ contains a nontrivial homothety, we conclude by Sah's lemma \cite[Proposition 2.7(b)]{Sah}. Otherwise, it follows from \eqref{ass:abs-G-K-irr} and \cite[Proposition 5.15(1)]{matar-nekovar:kolyvagin} that $p \nmid \# H = \# \Gal(K(A[\P])/K)$, thus $\hone(K(A[\P])/K, A[\P])=0$, since $A[\P]$ is a $p$-group.
     
     If $M>1$, the result follows from \eqref{ass:abs-G-K-irr} and \cite[Corollary 5.21]{matar-nekovar:kolyvagin}, once showed that one of the conditions (a)-(g') in \emph{loc.~cit.}~is satisfied. First, note that if $F  \not\supseteq \Q(\mu_p)^+$ or $\cOE/\P = \F_3$, then respectively (a') or (e) are fulfilled. Instead, if $F \supseteq \Q(\mu_p)^+$ and $\cOE/\P \ne \F_3$, then \eqref{ass:image-residual} ensures that (b') is. Indeed, \eqref{ass:u-0}, which is part of \eqref{ass:p-ndiv}, implies that $p \nmid \# (\mu_K/\mu_F) = \frac{\# \mu_K}{2}$. In particular $\Q(\mu_p) \not\subseteq K$. 
\end{proof}

 Since \eqref{ass:modular-av} implies  \eqref{ass:O-E-lin} and since \eqref{ass:p-ndiv} implies \eqref{ass:u-0} and \eqref{ass:deg-lambda}, under this set of hypotheses the machinery developed in the previous sections applies. In particular, Heegner points and Kolyvagin classes can be defined as in Section \ref{sec:modular-av-and-heegner}, after fixing $x\in\CM(N_H^\ast,K)$. 
\begin{definition}
    Let $y_K = \Tr_{K(x)/K} y(1) \in A(K)$.
\end{definition}
In particular, we have that $\cbf(1) = \delta_K(y_K \mod \P^M)$.
The aim of this section is to prove the following theorem.
\begin{theorem}\label{main_theorem}
Under the previous assumptions, suppose that $y_K \notin \fp A(K)$. 
Then 
\begin{equation*}
    \Selm_{\fp^M}(A/K)=\delta_K(A(K)/\fp^M A(K))=(\cOE/\P^M) \cdot \delta_K(y_K \bmod \P^M) \cong \cOE/\fp^M,
\end{equation*}
and thus $\Sha(A/K)[\fp^M]=0$. 
\end{theorem}

We first set up some notation. For any field extension $L/F$ and any prime $\ellfr$ of $F$, set
\begin{align*}
\qquad \qquad &A(L_\ellfr)/\fp^M A(L_\ellfr):=\oplus_{v \mid \ellfr} A(L_v)/\fp^M A(L_v),&\\[0.5em]
&\hone(L_\ellfr, A[\P^M]):=\oplus_{v \mid \ellfr} \hone(L_v, A[\P^M]),&\\[0.5em]
&\hone(L_\ellfr, A)[\fp^M]:=\oplus_{v \mid \ellfr}\hone(L_v, A)[\fp^M],
\end{align*}
where the sums are taken over all primes $v$ of $L$ dividing $\ellfr$. With this notation, there are natural semi-localization maps
\begin{align*}
\qquad \qquad &\res_{\ellfr}\colon A(L)/\fp^M A(L) \to A(L_\ellfr)/\fp^MA(L_\ellfr),&\\[0.5em]
&\res_{\ellfr}\colon \hone(L, A[\P^M]) \to \hone(L_\ellfr, A[\P^M]),&\\[0.5em]
&\res_{\ellfr}\colon \hone(L, A)[\fp^M] \to \hone(L_\ellfr, A)[\fp^M].
\end{align*}
If $z$ is an element in any of the groups $A(L)$, $\hone(L, A[\P^M])$ or $ \hone(L, A)[\fp^M]$, then we will often denote $\res_{\ellfr}(z)$ by $z_{\ellfr}$.

Given a $\Z[\frac{1}{2}][\tau_c]$-module $N$, there is a decomposition $N=N^+ \oplus N^-$, where $N^+$ and $N^-$ are the submodules on which $\tau_c$ acts as $+1$ and $-1$, respectively. Also, if $x \in N$ and $X \subseteq N$, we let $x^{\pm}=\frac{1}{2}(x\pm \tau_c x)$ and
$X^{\pm}=\set{x^{\pm} \; | \; x \in X}$.
 \begin{lemma}\label{eigenspace_structure_lemma}
For any $n \geq 1$ we have that $A[\P^n]^{\pm} \cong \cOP/\fp^n$.
\end{lemma}
\begin{proof}

First we show this for $n=1$. Suppose that for some $\varepsilon \in \{\pm 1\}$ we have that $A[\P]^{\varepsilon}=A[\P]$. Then for any $P, Q \in A[\P]$ we have
\[
\tau_c \cdot (P, Q)_{\P, 1}^\gamma=(\tau_c P, \tau_c Q)_{\P, 1}^\gamma=(\varepsilon P, \varepsilon Q)_{\P, 1}^\gamma=(P, Q)_{\P, 1}^\gamma
\]
Moreover, $\tau_c \cdot (P, Q)_{\P, 1}^\gamma = \chi_p(\tau_c) (P, Q)_{\P, 1}^\gamma = - (P, Q)_{\P, 1}^\gamma$, thus $(P,Q)_{\P, 1}^\gamma=0$ for all $P,Q \in A[\P]$. This contradicts the non-degeneracy of the pairing. Thus $A[\P]^{\pm}$ are proper $\cOP$-submodules of $A[\P]$. Since $A[\P]=A[\P]^+ \oplus A[\P]^-$, we see that $A[\P]^{\pm} \cong \cOP/\fp$.

When $n> 1$, we can conclude using the decomposition $A[\P^n]=A[\P^n]^+ \oplus A[\P^n]^-$ together with the fact that $A[\P^n]^\pm [\P]=A[\P]^\pm\cong\cOP/\P$.
\end{proof}

\begin{lemma}\label{H1_structure_lemma}
Let $\ellfr \in \admissible_1(M)$. Then $\hone(K_\ellfr,A)[\fp^M]$ is a free $\cOP/\fp^M$-module of rank two.
\end{lemma}
\begin{proof}
Let $v$ be the unique prime of $K$ above $\ellfr$. By composing the isomorphism of Proposition \ref{prop:kolyvaginclasses_properties}(ii) with the evaluation at $\Frob_v$, we obtain an $\cOP/\fp^M$-isomorphism 
\[
\hone(K_v,A)[\fp^M] \cong A(K_v)[\fp^M].
\]
As $\ellfr \in \admissible_1(M)$, the prime $v$ splits completely in $K(A[\P^M])/K$, therefore 
\[
A(K_v)[\fp^M]=A(\bar{K}_v)[\fp^M]\cong (\cOP/\P^M)^2.\qedhere
\]
\end{proof}

\subsection{Pairing and Galois extensions}\label{sec:Kolyvaigin-pairing}
In the rest of this section, let $L=K(A[\P^M])$ and $\cG=\Gal(L/K)$. 
\begin{proposition}\label{restriction_injective_prop1}
The restriction
\[
\res_L\colon \hone(K, A[\P^M]) \to \hone(L, A[\P^M])^{\cG}=\Hom_{\cG}(G_L^\ab,A[\P^M])
\]
is injective.
\end{proposition}
\begin{proof}
The result follows from the inflation-restriction exact sequence and Lemma \ref{lemma:hone-vanishing}.
\end{proof}

\begin{lemma}\label{exactseq_lemma}
Let $n >1$ and $\alpha \in \fp^{n-1} \setminus \fp^n$. Then the multiplication by $\alpha$ induces an exact sequence of $\cOP[G_F]$-modules
\[
0 \to A[\P^{n-1}] \to A[\P^n] \xrightarrow{\alpha} A[\P] \to 0.\]
\end{lemma}
\begin{proof}
Since $\fp^{n-1}\cOP= \alpha \cOP$, then the sequence
$0 \to A[\P^{n-1}] \to A[\P^n] \xrightarrow{\alpha} A[\P]$ is exact.
Thanks to order considerations, the last map must be surjective.
\end{proof}

By Proposition \ref{restriction_injective_prop1}, there is a pairing
\begin{equation}\label{eq:Galois-pairing}
\begin{tikzcd}[row sep=tiny]
[\; , \; ] \colon \hone(K, A[\P^M]) \times G_L^\ab \ar[r] & A[\P^M]\\
 \hspace{65pt}(s,\rho) \ar[r, mapsto]& {[s,\rho]}:=(\res_L s)(\rho)
\end{tikzcd}
\end{equation}
which satisfies $[\sigma \cdot s, \sigma \cdot \rho]=[s, \sigma \cdot \rho]=\sigma \cdot [s, \rho]$ for all $s \in \hone(K, A[\P^M])$, $\rho \in G_L^\ab$ and $\sigma \in \mathcal{G}$.
Note that if $[s, \rho]=0$ for all $\rho \in G_L^\ab$, then $s=0$ by the injectivity of $\res_L$.

Let $S \subseteq \hone(K, A[\P^M])$ be an $\cOP$-submodule with finite cardinality. Let $G_L^S$ be the subgroup of $G_L^\ab$ made by all $\rho \in G_L^\ab$ such that $[s, \rho]=0$ for all $s \in S$, i.e.,
\begin{equation*}
    G_L^S:=\bigcap_{s\in S}\ker(\res_L s),
\end{equation*}
and let $L_S$ be the fixed field of $G_L^S$. Then, $L_S/L$ is a finite (abelian) Galois extension and, since the image via $\res_L$ of any element of $S$ is a $\mathcal{G}$-equivariant morphism, also $L_S/K$ is Galois.
Therefore, the pairing \eqref{eq:Galois-pairing} induces a $\cG$-module injection
\[
\phi_S\colon \Gal(L_S/L) \hookrightarrow \Hom_{\cOP}(S, A[\P^M]).
\] 
Similarly, there is an $\cOP$-module injection
\[\psi_S\colon S \hookrightarrow \Hom_{\cG}(\Gal(L_S/L), A[\P^M]).\]
\begin{definition}\label{def:the-O-module-X_S}
    If $S \subseteq \hone(K, A[\P^M])$ is an $\cOP$-submodule with finite cardinality, we denote by $X_S$ the $\cOP$-submodule of $\Hom_{\cOP}(S, A[\P^M])$ generated by $\img \phi_S$. 
\end{definition}

Note that $X_S$ is an $\cOP[\cG]$-submodule, 
i.e., the inclusion map
\[
\phi_{S,\fp}\colon X_S \hookrightarrow \Hom_{\cOP}(S, A[\P^M])
\]
is an $\cOP[\cG]$-module homomorphism. 
We have moreover a natural $\cOP$-module injection
\[
\psi_{S,\fp}\colon S \hookrightarrow \Hom_{\cOP[\cG]}(X_S, A[\P^M]),
\]
defined by $\psi_{S, \P}(s)(x) = \phi_{S, \P}(x)(s)$, for any $s \in S$, $x \in X_S$. In particular, if $x = \sum_i r_i \phi_S(\rho_i)$, for $r_i \in \cOP$ and $\rho_i \in \Gal(L_S/L)$, we have that $\psi_{S, \P}(s)(x) = \sum_i r_i \psi_S(s)(\rho_i)$. 
\begin{remark}\label{rk:tau_c-equivariance}
    If $S$ is stable under the action of $\Gal(K/F)$ on $\hone(K, A[\P^M])$, all the constructions of this section are compatible with the action of $\tau_c$: if this is the case, $L_S$ is Galois also over $F$,
    $\phi_S$ (resp.~$\phi_{S, \P}$) is a $\Gal(L/F)$-module (resp.~$\cOP[\Gal(L/F)]$-module) injection, 
    $\psi_S$ and $\psi_{S, \P}$ are both $\cOP[\Gal(K/F)]$-module injections.
\end{remark}

\begin{proposition}\label{pairing_prop1}
Assume that $S$ is a finite $\cOP$-submodule of $\hone(K, A[\P^M])$. Then the maps $\phi_{S,\fp}$ and $\psi_{S,\fp}$ are isomorphisms.
\end{proposition}
\begin{proof}
The finite module $S$ is finitely generated over $\cOP$ and is annihilated by $\fp^M$. Therefore, by the structure of finitely generated modules over a PID, there is an $\cOP/\fp^M$-module isomorphism
\begin{equation*}
    S \cong \bigoplus_{i=1}^r \cOP/\fp^{m_i},
\end{equation*}
where $0 \le m_1 \le m_2 \le \cdots \le m_r \le M$. Let $D=\#\cOP/\fp$, so that $\#S=D^{\sum_{i=1}^r m_i}$. Then
\begin{gather*}
\Hom_{\cOP}(S, A[\P^M]) \cong \Hom_{\cOP}\Bigl(\bigoplus_{i=1}^r \cOP/\fp^{m_i}, A[\P^M]\Bigr)\cong\\
 \cong \bigoplus_{i=1}^r \Hom_{\cOP}( \cOP/\fp^{m_i}, A[\P^M])
 = \bigoplus_{i=1}^r\Hom_{\cOP/\fp^{m_i}}(\cOP/\fp^{m_i}, A[\P^{m_i}])
\cong \bigoplus_{i=1}^r A[\P^{m_i}].
\end{gather*}
Thus, recalling that $S$ is a trivial $\cG$-module, the previous is an isomorphism of $\cOP[\cG]$-modules, therefore composed with $\phi_{S, \P}$ we find an $\cOP[\cG]$-module injection 
\[\tilde{\phi}_{S, \P}: X_S \hookrightarrow \bigoplus_{i=1}^r A[\P^{m_i}].\]
By Lemma \ref{exactseq_lemma}, for any $t>1$ there is an $\cOP[\cG]$-module isomorphism $A[\P^t]/A[\P^{t-1}] \cong A[\P]$ and, 
by \eqref{ass:abs-G-K-irr}, $A[\P]$ is a simple $\cOP[\cG]$-module. Therefore, $\bigoplus_{i=1}^r A[\P^{m_i}]$ has a composition series all of whose factors are isomorphic to the simple $\cOP[\cG]$-module $A[\P]$ and 
\[
\text{length}_{\cOP[\cG]}\biggl(\bigoplus_{i=1}^r A[\P^{m_i}]\biggr)=\sum_{i=1}^r m_i.
\]
Let $c=\text{length}_{\cOP[\cG]}(X_S)$. From the injectivity of $\tilde{\phi}_{S, \P}$ we get that
\begin{equation}\label{inequality1}
c \leq \sum_{i=1}^r m_i.
\end{equation}
By the Jordan-Holder theorem, we have a composition series for the $\cOP[\cG]$-module $X_S$ $$0=N_0 \subset N_1 \subset \cdots \subset N_c=X_S$$
where $N_i/N_{i-1} \cong A[\P]$. We now use this composition series in order to compute the order of $\Hom_{\cOP[\cG]}(X_S, A[\P^M])$. For any $i >0$ we have an exact sequence
\[0 \to N_{i-1} \to N_i \to N_i/N_{i-1} \to 0.\]
From this, we get an exact sequence
\begin{equation}\label{Hom_exactseq}
0 \to \Hom_{\cOP[\cG]}(N_i/N_{i-1}, A[\P^M]) \to \Hom_{\cOP[\cG]}(N_i, A[\P^M]) \to \Hom_{\cOP[\cG]}(N_{i-1}, A[\P^M]).
\end{equation}
Since $N_i/N_{i-1} \cong A[\P]$, we have that
\[
\Hom_{\cOP[\cG]}(N_i/N_{i-1}, A[\P^M]) \cong \Hom_{\cOP[\cG]}(A[\P], A[\P^M]) \cong \Hom_{\cOP[\cG]}(A[\P], A[\P]).
\]
By \eqref{ass:abs-G-K-irr} and \cite[Theorem 29.13]{CR}, there is an $\cOP$-module isomorphism
\[\Hom_{\cOP[\cG]}(A[\P], A[\P]) \cong \cOP/\fp.\]
Thus, by an induction argument using the exact sequence (\ref{Hom_exactseq}), we have \[\#\Hom_{\cOP[\cG]}(X_S, A[\P^M]) \leq D^c.\]
Therefore, by the injectivity $\psi_{S, \P}$ we have
$$D^{\sum_{i=1}^r m_i}=\#S \leq \#\Hom_{\cOP[\cG]}(X_S, A[\P^M]) \leq D^c,$$
whence $\sum_{i=1}^r m_i \leq c$. So, by the inequality (\ref{inequality1}), we have $\sum_{i=1}^r m_i=c$. This implies that $\phi_{S,\P}$ and $\psi_{S, \P}$ are isomorphisms.
\end{proof}

As a consequence of this proposition, from now on can (and will) use the symbol $X_S$ to denote $\Hom_{\cOP}(S,A[\P^M])$. Keep also in mind that, by definition, every element of $X_S$ can be written as an $\cOP$-linear combination of elements of $\phi_S(\Gal(L_S/L))$.

\subsection{Compatibility Properties}\label{compatibility_sebsection}

We need to study some compatibility properties, when the $\cOP$-module $S$ varies. For this section, assume that
\begin{equation}\label{eq:basic-exact-sequence}
    \begin{tikzcd}
	0 & S & T & {T/S} & 0
	\arrow[from=1-1, to=1-2]
	\arrow["{\iota}", from=1-2, to=1-3]
	\arrow["{\theta}", from=1-3, to=1-4]
	\arrow[from=1-4, to=1-5]
\end{tikzcd}
\end{equation}
is an exact sequence of finite $\cOP$-submodules of $\hone(K, A[\P^M])$. The pairing \eqref{eq:Galois-pairing} induces a pairing
\begin{equation}\label{eq:Galois-pairing-quotients}
\begin{tikzcd}[row sep=tiny]
[\; , \; ] \colon T/S \times \Gal(L_T/L_S) \ar[r] & A[\P^M]\\
 \hspace{65pt}(t+S,\rho) \ar[r, mapsto]& {[t+S,\rho]}:=(\res_L t)(\rho).
\end{tikzcd}
\end{equation}

Applying the same argumets presented after Definition \ref{def:the-O-module-X_S}, one can show that this pairing induces injections
\begin{equation*}
    \phi_{T/S}\colon \Gal(L_T/L_S) \hookrightarrow \Hom_{\cOP}(T/S, A[\P^M]),
\end{equation*}
\begin{equation*}
    \psi_{T/S}\colon T/S \hookrightarrow \Hom_{\cG}(\Gal(L_T/L_S), A[\P^M]).
\end{equation*}

\begin{definition}
    Denote by $X_{T/S}$ the $\cOP$-submodule of $\Hom_{\cOP}(T/S, A[\P^M])$ generated by $\img \phi_{T/S}$. 
\end{definition}
Then, there are injections
\begin{equation*}
    \phi_{T/S,\fp}\colon X_{T/S} \hookrightarrow \Hom_{\cOP}(T/S, A[\P^M]),
\end{equation*}
\begin{equation*}
    \psi_{T/S,\fp}\colon T/S \hookrightarrow \Hom_{\cOP[\cG]}(X_{T/S}, A[\P^M])
\end{equation*}
of $\cOP[\cG]$-modules. The same argument used to prove Proposition \ref{pairing_prop1} shows that the maps $\phi_{T/S, \fp}$ and $\psi_{T/S, \fp}$ are isomorphisms. In particular, it follows that we can identify $X_{T/S}$ with $\Hom_{\cOP}(T/S, A[\P^M])$.

Denote by $\iota^\ast\colon X_T\to X_S$ and by $\theta^\ast\colon X_{T/S}\to X_T$ the pullbacks of $\iota$ and $\theta$, respectively. 

 \begin{proposition}\label{prop:exact-sequence-X}
    The maps $\theta^\ast$ and $\iota^\ast$ induce an exact sequence
    \begin{equation*}
    \begin{tikzcd}
	0 & X_{T/S} & X_T & X_{S} & 0.
	\arrow[from=1-1, to=1-2]
	\arrow["{\theta^\ast}", from=1-2, to=1-3]
	\arrow["{\iota^\ast}", from=1-3, to=1-4]
	\arrow[from=1-4, to=1-5]
\end{tikzcd}
\end{equation*}
Moreover,
\begin{enumerate}[label=(\roman*)]
    \item for any $\xi=\sum_{i=1}^n r_i \phi_T(\rho_i)\in X_T$, with $r_i \in \cOP$ and $\rho_i \in \Gal(L_T/L)$, we have that $\iota^\ast(\xi)=\sum_{i=1}^n r_i \phi_S(\rho_i|_{L_S})$;
    \item for any $\xi=\sum_{i=1}^n r_i \phi_{T/S}(\rho_i)\in X_{T/S}$ with $r_i \in \cOP$ and $\rho_i \in \Gal(L_T/L_S)$, we have that $\theta^\ast(\xi)=\sum_{i=1}^n r_i \phi_T(\rho_i)$.
\end{enumerate}
\end{proposition}
\begin{proof}
    Points (i) and (ii) are just a routine verification. Thanks to the left-exactness of the $\Hom_{\cOP}(-, A[\P^M])$ functor, we're just left to show that $\iota^\ast$ is surjective. This is a consequence of point (i) and of the surjectivity of the natural projection $\Gal(L_T/L)\to\Gal(L_S/L)$.
\end{proof}

The maps $\iota^\ast$ and $\theta^\ast$ induce pullbacks $\iota^{**}\colon \Hom_{\cOP[\cG]}(X_S, A[\P^M])\to \Hom_{\cOP[\cG]}(X_T, A[\P^M])$ and $\theta^{**}\colon \Hom_{\cOP[\cG]}(X_T, A[\P^M])\to \Hom_{\cOP[\cG]}(X_{T/S}, A[\P^M])$.

\begin{proposition}
    There is a commutative diagram with exact rows
    \begin{equation}\label{full_diagram2}
\begin{tikzcd}[column sep=tiny, cramped]
    0 \ar[r] &S \ar[d, "\psi_{S, \fp}"] \ar[r, "\iota"] &T \ar[d, "\psi_{T, \fp}"] \ar[r, "\theta"] &T/S \ar[d, "\psi_{T/S, \fp}"] \ar[r] &0\\
0 \ar[r] &\Hom_{\cOP[\cG]}(X_S, A[\P^M]) \ar[r, "{\iota^{**}}"] &\Hom_{\cOP[\cG]}(X_T, A[\P^M]) \ar[r, "{\theta^{**}}"] &\Hom_{\cOP[\cG]}(X_{T/S}, A[\P^M]) \ar[r] &0.
\end{tikzcd}
\end{equation}
\end{proposition}
\begin{proof}
    The commutativity of the left and right squares is just a routine verification. Since the top sequence is exact and the vertical maps are isomorphisms, also the bottom sequence is exact. 
\end{proof}

\begin{remark}\label{rk:iota-theta-tau-equivariant}
    The maps $\iota^\ast$, $\theta^\ast$, $\iota^{**}$ and $\theta^{**}$ are naturally $\cG$-equivariant. When \eqref{eq:basic-exact-sequence} is also an exact sequence of $\Gal(K/F)$-modules, they are also $\Gal(K/F)$-equivariant.
\end{remark}

\subsection{Direct Products}\label{directproduct_subsection}
In this paragraph, keep denoting by $S$ a finite cardinality $\cOP$-submodule of $\hone(K, A[\P^M])$ and assume that $S=S_1\oplus S_2$ for two $\cOP$-modules $S_1$ and $S_2$. The principal aim of this section is to show that the module $X_S\cong X_{S_1}\oplus X_{S_2}$ is actually generated as an $\cOP$-module by $\phi_S\bigl(\Gal(L_S/L_{S_1}\cap L_{S_2})\bigr)$. 

Since $G_L^S=G_L^{S_1} \cap G_L^{S_2}$, we obtain that $L_S=L_{S_1}L_{S_2}$.
Therefore,
$$\Gal(L_S/L_{S_1}\cap L_{S_2})\cong \Gal(L_{S_1}/L_{S_1}\cap L_{S_2})\times \Gal(L_{S_2}/L_{S_1}\cap L_{S_2}).$$
Applying Chebotarev's density theorem on this direct product will be much more useful for us than applying it on the Galois group $\Gal(L_S/L)$. 

The following diagram settles the notation for the natural arrows attached to the decomposition $S=S_1\oplus S_2$:
\begin{equation*}
    \begin{tikzcd}
	{S/S_2} & {S_1\oplus S_2} & {S/S_1} \\
	{S_1} && {S_2}
	\arrow["{\theta_1}"', from=1-2, to=1-1]
	\arrow["{\theta_2}", from=1-2, to=1-3]
	\arrow["{\eta_1}", from=2-1, to=1-1]
	\arrow["{\eta_2}"', from=2-3, to=1-3]
\end{tikzcd}
\end{equation*}
All these maps induce pullbacks after applying the functor $X_{\bullet}=\Hom_{\cOP}(\bullet, A[\P^M])$, to be denoted with an upper star. Since $X_\bullet$ is an additive functor, we immediately obtain the following result.

\begin{lemma}\label{lem:directproduct_lemma}
    The morphism $\theta_1^\ast+\theta_2^\ast\colon X_{S/S_2}\oplus X_{S/S_1}\to X_S$ is an isomorphism.
\end{lemma}

We also have the following explicit description of the isomorphism $\eta_i^\ast$.

\begin{lemma}
    For any $\xi=\sum_{i=1}^n r_i \phi_{S/S_2}(\rho_i) \in X_{S/S_2}$ with $r_i \in \cOP$ and $\rho_i \in \Gal(L_S/L_{S_2})$, we have that
    \begin{equation*}
        \eta_1^\ast(\xi)=\sum_{i=1}^n r_i \phi_{S_1}(\rho_i|_{L_{S_1}}).
    \end{equation*}
\end{lemma}
\begin{proof}
    It is just a routine verification.
\end{proof}

\begin{corollary}
$X_{S_1}=\cOP \cdot \phi_{S_1}\bigl(\Gal(L_{S_1}/L_{S_1} \cap L_{S_2})\bigr)$ and $X_{S_2}=\cOP \cdot \phi_{S_2}\bigl(\Gal(L_{S_2}/L_{S_1} \cap L_{S_2})\bigr)$.
\end{corollary}
\begin{proof}
Standard Galois theory implies that restriction to $L_{S_1}$ induces an isomorphism 
\[
\Gal(L_S/L_{S_2}) \cong \Gal(L_{S_1}/L_{S_1} \cap L_{S_2}).
\]
Then, the previous lemma implies that the image of the map $\eta_1^\ast$ is $\cOP\cdot\phi_{S_1}(\Gal(L_{S_1}/L_{S_1} \cap L_{S_2}))$, which coincides with the whole $X_{S_1}$ as $\eta_1^\ast$ is surjective. The proof for $X_{S_2}$ is identical.
\end{proof}

\begin{proposition}
We have that $X_S=\cOP \cdot \phi_S\bigl(\Gal(L_S/L_{S_1} \cap L_{S_2})\bigr)$.
\end{proposition}
\begin{proof}
By Lemma \ref{lem:directproduct_lemma} and Proposition \ref{prop:exact-sequence-X} we have
\begin{align*}
X_S&=\theta_1^\ast(X_{S/S_2})+\theta_2^\ast(X_{S/S_1})\\
&=\theta_1^\ast(\cOP \cdot \phi_{S/S_2}(\Gal(L_S/L_{S_2}))+\theta_2^\ast(\cOP \cdot \phi_{S/S_1}(\Gal(L_S/L_{S_1}))\\
&=\cOP \cdot \phi_S(\Gal(L_S/L_{S_2}))+\cOP \cdot \phi_S(\Gal(L_S/L_{S_1}))\\
&=\cOP\cdot \phi_S(\Gal(L_S/L_{S_2})\Gal(L_S/L_{S_1}))\\
&=\cOP\cdot \phi_S(\Gal(L_S/L_{S_1} \cap L_{S_2})).\qedhere
\end{align*}
\end{proof}

\begin{remark}\label{rk:theta-eta-tau-equivariant}
    If $S_1$ and $S_2$ are $\Gal(K/F)$-modules, then $\theta_i^\ast$ and $\eta_i^\ast$ are $\Gal(K/F)$-equivariant morphisms.
\end{remark}

\subsection{Setup of the proof}\label{setup_proof_subsection}

We now proceed to prove Theorem \ref{main_theorem}. With this aim, we further assume that $y_K\notin \P A(K)$. The technique that we will apply is inspired by \cite{BD}. Moreover, in order to simplify the notations we will perform this proof at first assuming \eqref{ass:Tamagawa} on $A$. Then in Section \ref{sec:removing-tama} we will show how to modify the Kolyvagin classes in order to remove this assumption.

Let $\ellfr$ be a prime of $F$. 
Note that, by definition, the restriction at any prime $v \mid \ellfr$ of $K$ of a class in $\Selm_{\fp^M}(A/K)$ lies in the image of the (injective) local Kummer map. Thus we have a map (that we denote by $\res_\ellfr$ with a slight abuse of notation)
\begin{equation}\label{Selmer_map1}
\res_{\ellfr}\colon \Selm_{\fp^M}(A/K) \to A(K_{\ellfr})/\fp^M  A(K_\ellfr).
\end{equation}
In the following, for an $(\cOP/\fp^M)$-module $C$, define $C^\ast:= \Hom_{\cOP}(C, \cOP/\fp^M)$. Endowing $C$ with the discrete topology, we have that $C^\ast =  \Hom_{\cOP}(C, E_\P/\cOP) = D_{\cOP}(C)$, as defined in Section \ref{sec:duality}. In particular, the functor $(-)^\ast$ is exact, since $E_\P/\cOP$ is divisible over a PID and thus injective. Applying it to $\res_{\ellfr}$ and composing it with the $\cOP$-module isomorphism 
\begin{equation}\label{Tatelocalduality_isomorphism1}
\phi_{\ellfr} \colon \hone(K_{\ellfr}, A)[\fp^M] \iso \bigl(A(K_{\ellfr})/\fp^M A(K_\ellfr)\bigr)^\ast
\end{equation}
induced by the local Tate pairing (see Proposition \ref{prop:tate-pairing}), we get an $\cOP$-module homomorphism
\[\uppsi_{\ellfr} := \res_{\ellfr}^\ast \circ\, \phi_\ellfr: \hone(K_{\ellfr}, A)[\fp^M] \to \Selm_{\fp^M}(A/K)^\ast.\]
This map will play a fundamental role in our proof. In particular we will need the following technical result.
\begin{lemma}\label{global_duality_proposition1}
For any $\alpha \in \hone(K, A)[\P^M]$, then $\sum_\ellfr \uppsi_\ellfr(\res_\ellfr\alpha)= 0$,
where the sum is taken over all primes $\ellfr$ of $F$.    
\end{lemma}

\begin{proof}
    For any $s \in \Selm_{\fp^M}(A/K)$ and $\alpha \in \hone(K, A)[\P^M]$, we have that
    \[
    \uppsi_\ellfr(\res_\ellfr\alpha)(s)= \phi_\ellfr(\res_\ell \alpha)(\res_\ellfr s) = \sum_{v \mid \ellfr} \langle \res_v s, \res_v \alpha \rangle_v,
    \]
    the sum varying over the primes $v \mid \ellfr$ of $K$. Therefore we have the equality
    \[
    \sum_{\ellfr} \uppsi_\ellfr(\res_\ellfr\alpha)(s) = \sum_v \langle \res_v s, \res_v \alpha \rangle_v,
    \]
    where the sums vary respectively over the primes $\ellfr$ of $F$ and $v$ of $K$, which shows the claim, since the right hand side vanishes, as follows from the global reciprocity law for elements in the Brauer group of $K$ (\cite[Theorem 8.1.17]{neukirch:cnf}).
\end{proof}
Since $y_K \notin \fp A(K)$ and $p$ is odd, the decomposition $y_K=y_K^+ + y_K^-$ induced by the action of $\tau_c$ implies that $y_K^\epsilon \notin \fp A(K)$ for some $\epsilon \in \{\pm \}$. From now on, let's fix an $\epsilon$ with this property.

As in the previous sections, let $L=K(A[\P^M])$ and denote by 
\[
\delta = \delta_K: A(K)/\fp^M A(K) \to \hone(K, A[\P^M])
\]
the global Kummer map (see Section \ref{sec:selmer-sha}). Let $\ellfr_1 \in \admissible_1(M)$ be an auxiliary prime. Define the following $\cOP$-submodules  of $\hone(K, A[\P^M])$:
\begin{itemize}
    \item $S_1 :=\cOP\cdot\cbf(1)^\epsilon=\cOP\cdot\delta\bigl(y_K^\epsilon+\fp^M A(K)\bigr)$;
    \item $S_2 :=\cOP\cdot\cbf(\ellfr_1)^{-\epsilon}$;
    \item $S:=S_1+S_2$.
\end{itemize}
There is the following diagram of field extensions:
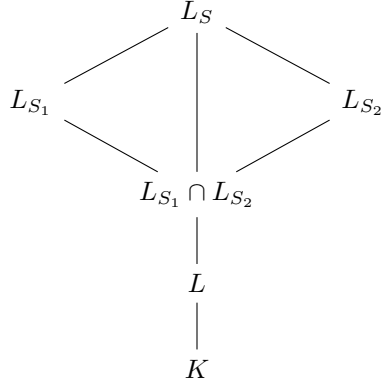
\[
\begin{tikzcd}
   &L_S\ar[ld, -] \ar[rd, -] \ar[dd, -]\\
    L_{S_1} \ar[rd, -] &&L_{S_2}\ar[ld, -]\\
    & L_{S_1} \cap L_{S_2} \ar[d, -]&\\
    & L \ar[d, -]&\\
    & K & 
\end{tikzcd}
\]
Since $\cbf(1)^\epsilon$ and $\cbf(\ellfr_1)^{-\epsilon}$ are by definition eigenvectors for the action of the complex conjugation $\tau_c$,  we have that  $S_1$, $S_2$ and $S$ are $\tau_c$-modules and, by Remark \ref{rk:tau_c-equivariance}, $L_{S_1}$, $L_{S_2}$ and $L_S$ are all Galois over $F$.

Given a subset $U$ of $\Gal(L_S/L_{S_1} \cap L_{S_2})$, define
\[
\mathscr{L}(U)=\set{\ellfr \in \admissible_1(M) \; | \; \Frob_{\ellfr}(L_S/F)=[\tau_c u] \; \text{for some} \; u \in U }\setminus\{\ellfr_1\},
\]
where $[\tau_c u]$ denotes the conjugacy class of $\tau_c u$ in $\Gal(L_S/F)$.  As usual, denote by $U^\pm$ the projection of $U$ to the $\pm 1$ eigenspaces of $\Gal(L_S/L_{S_1} \cap L_{S_2})$ for the action of $\tau_c$. We have the following key proposition.

\begin{proposition}\label{generating_Sel_proposition}
If $\phi_S(U^+)$ generates $X_S^+$ as an $\cOP$-module, then $\img \uppsi_{\ellfr}$ with $\ellfr$ ranging over $\mathscr{L}(U)$ generates $\Selm_{\fp^M}(A/K)^{\ast}$.
\end{proposition}

Before proving this proposition we need a lemma.

\begin{lemma}\label{lem:img-psi-l-generates-selmer-dual}
Suppose that for each $s \in \Selm_{\fp^M}(A/K)^+ \cup \Selm_{\fp^M}(A/K)^-$ there exists a finite set $Y_s \subseteq \mathscr{L}(U)$ with the property that $\res_{\ellfr}(s)=0$ for all $\ellfr \in Y_s$ implies that $s=0$. Then $\img \uppsi_{\ellfr}$ with $\ellfr$ ranging over $\mathscr{L}(U)$ generate $\Selm_{\fp^M}(A/K)^{\ast}$.
\end{lemma}

\begin{proof}
Let $Y$ be the union of the sets $Y_s$ as $s$ ranges over $\Selm_{\fp^M}(A/K)^+ \cup \Selm_{\fp^M}(A/K)^-$. Since each set $Y_s$ is finite and $\Selm_{\fp^M}(A/K)$ is finite (by \cite[I.4.15 and I.6.6]{milne:arithmetic-duality}), 
then $Y$ is finite. Consider the map
\[ 
\Theta\colon \Selm_{\fp^M}(A/K) \to \bigoplus_{\ellfr \in Y} A(K_{\ellfr})/\fp^MA(K_{\ellfr})
\]
given by $\Theta(t)=(\res_{\ellfr}(t))_{\ellfr \in Y}$ for any $t\in \Selm_{\fp^M}(A/K)$. The properties of the sets $Y_s$ imply that the restrictions $\Theta^+$ and $\Theta^-$ of $\Theta$ to $\Selm_{\fp^M}(A/K)^+$ and $\Selm_{\fp^M}(A/K)^-$, respectively, are injective. Since $\Theta$ is $\tau_c$-equivariant, this also implies that $\Theta=\Theta^+\oplus\Theta^-$ is injective, and thus $\Theta^\ast$ is surjective.
Moreover, since the set $Y$ is finite, we have an $\cOP$-isomorphism
\[
\psi\colon \bigoplus_{\ellfr \in Y} \Bigl(A(K_{\ellfr})/\fp^MA(K_{\ellfr})\Bigr)^\ast \to \Bigl(\bigoplus_{\ellfr \in Y} A(K_{\ellfr})/\fp^M A(K_{\ellfr})\Bigr)^\ast
\]
defined as follows: for $(\chi_{\ellfr})_{\ellfr \in Y} \in \bigoplus_{\ellfr \in Y} (A(K_{\ellfr})/\fp^M)^\ast$ and let $(\alpha_{\ellfr})_{\ellfr \in Y} \in \bigoplus_{\ellfr \in Y} A(K_{\ellfr})/\fp^M$, then $\psi\bigl((\chi_{\ellfr})_{\ellfr \in Y}\bigr)\bigl((\alpha_{\ellfr})_{\ellfr \in Y}\bigr)=\sum_{\ellfr \in Y} \chi_{\ellfr}(\alpha_{\ellfr})$.
We have that $\Theta^\ast \circ \psi=\sum_{\ellfr \in Y} \res_{\ellfr}^\ast$. Indeed, 
\begin{gather*}
\bigl((\Theta^\ast \circ \psi) ((\chi_{\ellfr})_{\ellfr \in Y} )\bigr)(t) 
=\bigl(\psi((\chi_{\ellfr})_{\ellfr \in Y} )\circ \Theta\bigr)(t)
=\psi((\chi_{\ellfr})_{\ellfr \in Y} )(\Theta(t))
=\psi((\chi_{\ellfr})_{\ellfr \in Y} )((\res_{\ellfr}(t))=\\
=\sum_{\ellfr \in Y}  \chi_{\ellfr}(\res_{\ellfr}(t))
=\Bigl(\sum_{\ellfr \in Y}  \chi_{\ellfr} \circ \res_{\ellfr}\Bigr)(t)
=\Bigl(\Bigl(\sum_{\ellfr \in Y}  \res_{\ellfr}^\ast\Bigr)((\chi_{\ellfr})_{\ellfr \in Y} )\Bigr)(t),
\end{gather*}
for all $(\chi_{\ellfr})_{\ellfr \in Y} \in \bigoplus_{\ellfr \in Y} (A(K_{\ellfr})/\fp^M)^\ast$ and $t \in \Selm_{\fp^M}(A/K)$.

 Moreover, consider the map $\phi\colon \bigoplus_{\ellfr \in Y} \hone(K_{\ellfr}, A)[\fp^M] \to \bigoplus_{\ellfr \in Y} (A(K_{\ellfr})/\fp^M)^\ast$ defined as the sum of the $\cOP$-module isomorphisms $\phi_\ellfr$ of \eqref{Tatelocalduality_isomorphism1}, for $\ellfr \in Y$.
 We have
\[
\Theta^\ast \circ \psi \circ \phi=\Bigl(\sum_{\ellfr \in Y} \res_{\ellfr}^\ast\Bigr) \circ \phi = \sum_{\ellfr \in Y} \res_{\ellfr}^\ast \circ \phi_{\ellfr}=\sum_{\ellfr \in Y} \uppsi_{\ellfr}.
\]
Since $\phi$ and $\psi$ are isomorphims and $\Theta^\ast$ is surjective, then $\sum_{\ellfr \in Y} \uppsi_{\ellfr}$ is surjective. Thus $\img \uppsi_{\ellfr}$ with $\ellfr$ ranging over $\mathscr{L}(U)$ generate $\Selm_{\fp^M}(A/K)^{\ast}$.
\end{proof}

\begin{proof}[Proof of Proposition \ref{generating_Sel_proposition}]
Let $s \in \Selm_{\fp^M}(A/K)^a$ for some $a \in \set{\pm}$. By Lemma \ref{lem:img-psi-l-generates-selmer-dual}, to prove the proposition, it suffices to find a finite set $Y_s \subseteq \mathscr{L}(U)$ with the property that $\res_{\ellfr}(s)=0$ for all $\ellfr \in Y_s$ implies that $s=0$.

Let $T=\cOP\cdot s+S$. Since $T$ is a $\Gal(K/F)$-module, then $L_T/F$ is Galois. Let $$\Xi\colon \Gal(L_T/L_{S_1} \cap L_{S_2}) \to \Gal(L_S/L_{S_1} \cap L_{S_2})$$ be the natural projection map and let $\tilde{U}=\Xi^{-1}(U)$.  We first claim that $\phi_T(\tilde{U}^+)$ generates $X_T^+$ as an $\cOP$-module.

By Remark \ref{rk:iota-theta-tau-equivariant}, the exact sequence of Proposition \ref{prop:exact-sequence-X} induces an exact sequence
\[0 \to X_{T/S}^+ \xrightarrow{\theta^\ast} X_T^+  \xrightarrow{\iota^\ast} X_S^+ \to 0.\]
Let $\langle \tilde{U}^+ \rangle \subseteq \Gal(L_T/L)$ be the subgroup generated by $\tilde{U}^+$. In order to prove our claim, it will clearly suffice to show that $\phi_T(\langle \tilde{U}^+\rangle)$ generates $X_T^+$ as an $\cOP$-module. From the short exact sequence above, it follows that it's enough to show that $\iota^\ast\bigl(\phi_T(\langle \tilde{U}^+\rangle)\bigr)$ generates $X_S^+$ as an $\cOP$-module and $(\theta^\ast)^{-1}\bigl(\phi_T(\langle \tilde{U}^+\rangle)\bigr)$ generates $X_{T/S}^+$ as an $\cOP$-module.

By Proposition \ref{prop:exact-sequence-X}(i), we have that
\[\iota^\ast\bigl(\phi_T(\langle \tilde{U}^+\rangle)\bigr)=\phi_S\bigl(\Xi(\langle \tilde{U}^+ \rangle)\bigr)=\phi_S(\langle U^+ \rangle),\]
and this generates $X_S^+$ by hypothesis. Now we deal with $(\theta^\ast)^{-1}\bigl(\phi_T(\langle \tilde{U}^+\rangle)\bigr)$. By Proposition \ref{prop:exact-sequence-X}(ii)
\[\phi_{T/S}(\langle \tilde{U}^+\rangle \cap \Gal(L_T/L_S)) \subseteq (\theta^\ast)^{-1}(\phi_T(\langle \tilde{U}^+\rangle)).\]
Clearly we have $\Gal(L_T/L_S)^+=\langle \tilde{U}^+\rangle \cap \Gal(L_T/L_S)$ and so 
\[\phi_{T/S}\bigl(\Gal(L_T/L_S)^+\bigr)=\phi_{T/S}\bigl(\Gal(L_T/L_S)\bigr)^+ \subseteq (\theta^\ast)^{-1}\bigl(\phi_T(\langle \tilde{U}^+\rangle)\bigr).\]
Since $X_{T/S}=\cOP\cdot \phi_{T/S}\bigl(\Gal(L_T/L_S)\bigr)$, then $X_{T/S}^+=\cOP\cdot \phi_{T/S}\bigl(\Gal(L_T/L_S)\bigr)^+$. Therefore, $(\theta^\ast)^{-1}\bigl(\phi_T(\langle \tilde{U}^+\rangle)\bigr)$ generates $X_{T/S}^+$ as an $\cOP$-module. Thus we have shown that $\phi_T(\tilde{U}^+)$ generates $X_T^+$ as an $\cOP$-module.

Let $x\in\tilde{U}$. By the Chebotarev's density theorem, we may choose $\ellfr_x \in \mathscr{L}(U)$ such that $\Frob_{\ellfr_x}(L_T/F)=[\tau_c x]$. We let 
$$Y_s:=\{\ellfr_x \, | \, x \in \tilde{U}^+\}.$$
Assume that $\res_{\ellfr}(s)=0$ for all $\ellfr \in Y_s$. Then, for such $\ellfr$'s, $\res_{\ellfr}(s)(\Frob_{\lambda}(L_T/L))=0$ for all primes $\lambda$ of $L$ above $\ellfr$. For any such prime $\lambda$ there is $x\in\tilde{U}^+$ such that 
\[\Frob_{\lambda}(L_T/L)=(\tau_c x)^2=(\tau_c x\tau_c)x =x^2,\] 
and hence $\res_L(s)(x)=0$.

Since $\phi_T(\tilde{U}^+)$ generates $X_T^+$ as an $\cOP$-module, it follows that $\psi_{T, \fp}(s)$ vanishes on $X_T^+$. Recall that $s$ lies in the $a$-eigenspace of the Selmer group. Taking into account Remark \ref{rk:tau_c-equivariance}, it follows that
\[(\psi_{T, \fp}(s))(X_T)=(\psi_{T, \fp}(s))(X_T^-) \subseteq A[\P^M]^{-a}.\]
Therefore $(\img \psi_{T, \fp}(s))[\fp] \subseteq A[\P]^{-a}$ is a proper $\cOP[\cG]$-submodule of $A[\P]$ (see Lemma \ref{eigenspace_structure_lemma}). So it follows from
\eqref{ass:abs-G-K-irr} that $(\img \psi_{T, \fp}(s))[\fp]=0$, whence $(\img \psi_{T, \fp}(s))=0$. Therefore $s=0$ since $\psi_{T, \fp}$ is injective.
\end{proof}

In the next two sections we will show how to choose carefully $\ellfr_1$ and $U$ in order to apply Proposition \ref{generating_Sel_proposition} for proving Theorem \ref{main_theorem}.

\subsection{Choosing the prime $\ellfr_1$}\label{choosing_ell_1_subsection}
We now show how to choose the prime $\ellfr_1$. First we argue with the following lemma.

\begin{lemma}\label{lem:first-lemma-l1}
    \begin{enumerate}[label=(\roman*)]
        \item $S_1$ is a free $\cOP/\fp^M$-module generated by $\cbf(1)^\epsilon$;
        \item the evaluation at $\cbf(1)^\epsilon$ yields an isomorphism $X_{S_1}^+\cong A[\P^M]^\epsilon\cong \cOP/\fp^M$.
    \end{enumerate}
\end{lemma}
\begin{proof}
(i) Let $r \in \cOP$ with $r y_K^\epsilon+\fp^M A(K)=\fp^M A(K)$. We must show that $r \in \fp^M$. If $r \notin \fp\cOP$, then $r$ is a unit in $\cOP$. In this case, $r y_K^\epsilon+\fp^M A(K)=\fp^M A(K)$ implies that $y_K^\epsilon \in \fp^M A(K)$, which is a contradiction.

Now assume that $r \in \fp\cOP \setminus \fp^M\cOP$. Then, by the definition of the $\cOP$-action on $A(K)/\fp^M$, we have that $r y_K^\epsilon+\fp^M A(K)=\alpha y_K^\epsilon+\fp^M A(K)$ for some $\alpha \in \cOE$ with $\alpha \in \fp \setminus \fp^M$. We have $\alpha y_K^\epsilon = \beta P$ where $\beta \in \fp^M$ and $P \in A(K)$.

Then we have the following factorization  
in terms of prime ideals of $\cOE$:
\[\beta \alpha^{-1}\cOE=\fp^s \prod_{i=1}^n \fq_i^{c_i} \prod_{i=n+1}^m \fq_i^{c_i},\]
where $s \geq 1$, $c_i >0$ for $1 \leq i \leq n$, $c_i <0$ for $n+1 \leq i \leq m$ and $\fq_i \neq \fp$ for any $i$.

If the second product of prime ideals is empty, then define $\gamma=1$. Otherwise, by the Chinese remainder theorem, we can choose $\gamma \notin \fp$ with $\gamma \in \fq_i^{-c_i}$ for $n+1 \leq i \leq m$. Then we see that $\gamma \beta \alpha^{-1}=\varpi$ for some $\varpi \in \fp$. Then $\gamma \beta= \varpi \alpha$. Thus $\alpha \gamma y_K^\epsilon=\varpi \alpha P$. So $\alpha(\gamma y_K^\epsilon-\varpi P)=0$. By Lemma \ref{torsion_lemma} this implies that $\gamma y_K^\epsilon = \varpi P$.

Since $\gamma \notin \fp$, therefore for some $\nu \in \fp$ and $\alpha' \in \cOE$ we have that $\alpha' \gamma +\nu =1$. Thus, $y_K^\epsilon=\alpha' \gamma y_K^\epsilon+\nu y_K^\epsilon=\alpha'\varpi P + \nu y_K^\epsilon \in \fp A(K)$. This is a contradiction and therefore we must have $r \in \fp^M\cOP$. 

(ii) As a consequence of point (i), we have that
$$\Hom_{\cOP}(S_1, A[\P^M])=\Hom_{\cOP/\fp^M}(S_1, A[\P^M])\cong \Hom_{\cOP/\fp^M}(\cOP/\fp^M, A[\P^M]) \cong A[\P^M].$$
When $\epsilon=+$ (resp. $\epsilon=-$), then this isomorphism is $\tau_c$-equivariant (resp. $\tau_c$-antiequivariant). Therefore, from Proposition \ref{pairing_prop1} we have that $X_{S_1}^+ \cong A[\P^M]^{\epsilon}$ and by Lemma \ref{eigenspace_structure_lemma} we have an $\cOP$-module isomorphism $A[\P^M]^{\epsilon} \cong \cOP/\fp^M$. 
\end{proof}

Recall that $X_{S_1}=\cOP \cdot \phi_{S_1}(\Gal(L_{S_1}/L))$ and so, since $\phi_{S_1}$ is $\tau_c$-equivariant, $X_{S_1}^+=\cOP \cdot \phi_{S_1}(\Gal(L_{S_1}/L)^+)$. Choose $x \in \Gal(L_{S_1}/L)^+$ such that $\cOP\cdot\phi_{S_1}(x)=X_{S_1}^+\cong \cOP/\fp^M$. Then, since $p$ is odd, we also have that $\cOP \cdot \phi_{S_1}((\tau_c x)^2)=\cOP \cdot \phi_{S_1}(x^2) \cong \cOP/\fp^M$. By Chebotarev's density theorem, we may find $\ellfr_1 \in \admissible_1(M)$ such that $\Frob_{\ellfr_1}(L_{S_1}/F)=[\tau_c x]$. 

\begin{lemma}\label{lem:second-lemma-l1}
    \begin{enumerate}[label=(\roman*)]
        \item $S_2$ and  $\cOP \cdot \dbf(\ellfr_1)_{\ellfr_1}^{-\epsilon}$ are free $\cOP/\P^M$-modules;
        \item the evaluation at $\cbf(\ellfr_1)^{-\epsilon}$ yields an isomorphism $X_{S_2}^+\cong A[\P^M]^{-\epsilon}\cong \cOP/\fp^M$.
    \end{enumerate}
\end{lemma}
\begin{proof}
    (i) Let $v$ be the prime of $K$ above $\ellfr_1$. Then, for any prime $\lambda$ of $L$ over $\ellfr_1$ we have $\Frob_{\lambda}(L_{S_1}/L)=(\tau_c x)^2=x^2$, and this same element coincides with $\Frob_v(L_{S_1}/K)$ since $v$ is split in $L$. By Proposition \ref{prop:comparison-image-kummer-unramified-local-conditions}, the cocycle $\cbf(1)_v^\epsilon$ is unramified. Then, evaluation at $\Frob_v$ gives an isomorphism of $\cOP$-modules
    \[
    \honeur(K_v, A[\P^M]) \cong \Hom(\langle \Frob_v \rangle, A(K_v)[\P^M]) \iso A(K_v)[\P^M] = A[\P^M], 
    \]
    where the last equality follows from the fact that $\ellfr_1 \in \admissible_1(M)$, thus the prime $v$ splits completely in $K(A[\P^M])/K$.
    Under this correspondence, the class $\cbf(1)_v^\epsilon$ corresponds to 
    \[
    \cbf(1)^\epsilon(\Frob_v(L_{S_1}/K))=\cbf(1)^\epsilon(x^2)=\phi_{S_1}(x^2)(\cbf(1)^\epsilon).
    \]
    Since $\phi_{S_1}(x^2)$ is an $\cOP$-generator of $X_{S_1}^+$, the isomorphism of Lemma \ref{lem:first-lemma-l1}(ii) yields that $\cOP \cdot \cbf(1)_v^\epsilon\cong \cOP/\fp^M$. It then follows from Proposition \ref{prop:kolyvaginclasses_properties}\ref{item:finite-singular} that $\cOP \cdot \dbf(\ellfr_1)_{v}^{-\epsilon}\cong \cOP/\fp^M$. In particular, we also have $\cOP \cdot \cbf(\ellfr_1)^{-\epsilon}\cong \cOP/\fp^M$.

    (ii) We can argue as in point (ii) of Lemma \ref{lem:first-lemma-l1}.
\end{proof}

\subsection{Choosing the set $U$}
We will use the decomposition of Section \ref{directproduct_subsection}. By Remark \ref{rk:theta-eta-tau-equivariant}, the isomorphism $X_S \cong X_{S/S_1} \oplus X_{S/S_2}$ of Lemma \ref{lem:directproduct_lemma} is $\tau_c$-equivariant. This, together with the isomorphisms $\eta_1^\ast$ and $\eta_2^\ast$, induces the isomorphisms
\begin{equation*}
    X_S^+ \cong X_{S/S_2}^+ \oplus X_{S/S_1}^+\cong X_{S_1}^+\oplus X_{S_2}^+.
\end{equation*}
By point (ii) of Lemmas \ref{lem:first-lemma-l1} and \ref{lem:second-lemma-l1}, we obtain a decomposition $X_S^+\cong \cOP/\fp^M \oplus \cOP/\fp^M$.

Now, we show how to choose the set $U$. As above, we have $$\cOP/\fp^M \cong X_{S/S_2}^+=\cOP\cdot \phi_{S/S_2}(\Gal(L_S/L_{S_2})^+).$$
Therefore there exists $\rho_1 \in \Gal(L_S/L_{S_2})^+$ with $\cOP \cdot \phi_{S/S_2}(\rho_1) = X_{S/S_2}^+$. Similarly, there exists $\rho_2 \in \Gal(L_S/L_{S_1})^+$ with $\cOP \cdot \phi_{S/S_1}(\rho_2) = X_{S/S_1}^+$. We now define
$$U:=\{\rho_1^{m_1}\rho_2^{m_2} \, | \, m_1, m_2 \in \{1,2\} \}\subseteq \Gal(L_S/L_{S_1}\cap L_{S_2}).$$
Since $\rho_2 \in \Gal(L_S/L_{S_1})^+$ and $\rho_1 \in \Gal(L_S/L_{S_2})^+$, we have $U=U^+$. 

\begin{lemma}
    $\phi_S(U)=\phi_S(U^+)$ generates $X_S^+$ as an $\cOP$-module.
\end{lemma}
\begin{proof}

    By Proposition \ref{prop:exact-sequence-X}(ii), $\phi_S(\rho_1^2\rho_2)-\phi_S(\rho_1\rho_2)=\phi_S(\rho_1)=\theta_1^\ast(\phi_{S/S_2}(\rho_1))\in\phi_S(U)$. Similarly, $\phi_S(\rho_1\rho_2^2)-\phi_S(\rho_1\rho_2)=\phi_S(\rho_2)=\theta_2^\ast(\phi_{S/S_1}(\rho_2))\in\phi_S(U)$.

    By Lemma \ref{lem:directproduct_lemma}, $\theta_1^\ast(X_{S/S_2}^+)+\theta_2^\ast(X_{S/S_1}^+)=X_S^+$. This, together with the fact that the two elements $\phi_{S/S_2}(\rho_1)$ and $\phi_{S/S_1}(\rho_2)$ generate $X_{S/S_2}^+$ and $X_{S/S_1}^+$, yields the claim.
\end{proof}

\begin{lemma}\label{H1_generators_lemma}
For any $\ellfr \in \mathscr{L}(U)$, $\hone(K_{\ellfr}, A)[\fp^M]$ is generated by $\dbf(\ellfr)^{-\epsilon}_{\ellfr}$ and $\dbf(\ellfr\ellfr_1)^{\epsilon}_{\ellfr}$.
\end{lemma}
\begin{proof}
Let $\ellfr \in \mathscr{L}(U)$. We claim that $\cOP\cdot\cbf(1)_\ellfr^\epsilon \cong \cOP \cdot \cbf(\ellfr_1)^{-\epsilon}_\ellfr \cong \cOP/\fp^M$. Let $u=\rho_1^{m_1}\rho_2^{m_2} \in U$ be such that $\Frob_{\ellfr}(L_S/F)=[\tau_c u]$ and let $\lambda$ be a prime of $L$ above $\ellfr$. We have that 
\[
\Frob_{\lambda}(L_S/L)=(\tau_c u)^2=u^2=\rho_1^{2m_1}\rho_2^{2m_2}.
\]
Thanks to Lemma \ref{lem:directproduct_lemma} and the explicit description of the maps $\theta_1^\ast$ and $\theta_2^\ast$ of Proposition \ref{prop:exact-sequence-X}(ii), we have that
\begin{equation*}
    \phi_S(u^2)=\theta_1^\ast(\phi_{S/S_2}(\rho_1^{2m_1}))\oplus\theta_2^\ast(\phi_{S/S_1}(\rho_2^{2m_2}))=\theta_1^\ast(2m_1\phi_{S/S_2}(\rho_1))\oplus\theta_2^\ast(2m_2\phi_{S/S_1}(\rho_2)).
\end{equation*}
Since $2m_1$ is coprime with $p$, we have that $2m_1\phi_{S/S_2}(\rho_1)$ is an $\cOP$-generator of $X_{S/S_2}^+$. Then, $\eta_1^\ast(2m_1\phi_{S/S_2}(\rho_1))=2m_1\phi_{S_1}(\rho_1|_{L_{S_1}})$ is an $\cOP$-generator of $X_{S_1}^+$. From this, exactly as in the proof of Lemma \ref{lem:second-lemma-l1}(i), it follows that $\cOP\cdot \cbf(1)_\ellfr^\epsilon\cong\cOP/\P^M$. Similarly, exchanging the roles of $S_1$ and $S_2$, one shows that $\cOP\cdot \cbf(\ellfr_1)_{\ellfr}^{-\epsilon}\cong\cOP/\P^M$.

By Proposition \ref{prop:kolyvaginclasses_properties}\ref{item:finite-singular}, we get that $\cOP\cdot\dbf(\ellfr)_\ellfr^{-\epsilon} \cong \cOP/\fp^M$ and $\cOP\cdot\dbf(\ellfr \ellfr_1)_\ellfr^\epsilon \cong \cOP/\fp^M$. Since $\dbf(\ellfr)_\ellfr^{-\epsilon}$ and $\dbf(\ellfr\ellfr_1)_\ellfr^\epsilon$ belong to different eigenspaces for the action of $\tau_c$ and $p$ is odd we get that
\[
\cOP\dbf(\ellfr)^{-\epsilon}+  \cOP \cdot\dbf(\ellfr\ellfr_1)_\ellfr^\epsilon\cong \cOP/\fp^M \oplus \cOP/\fp^M.
\]
Therefore, the desired result follows from Lemma \ref{H1_structure_lemma}.\qedhere
\end{proof}

\begin{proposition}\label{img_uppsi_prop}
For any $\ellfr \in \mathscr{L}(U)$, we have that $\img \uppsi_{\ellfr}$ is generated as an $\cOP$-module by $\uppsi_{\ellfr}(\dbf(\ellfr \ellfr_1)_\ellfr^\epsilon)$.
\end{proposition}
\begin{proof}
By Proposition \ref{prop:kolyvaginclasses_properties}\ref{item:loc-d-at-v-trivial}, $\dbf(\ellfr)_v^{-\epsilon} =0$ for any prime $v \nmid \ellfr$ of $K$. By Lemma \ref{global_duality_proposition1}, it follows that $\uppsi_{\ellfr}(\dbf(\ellfr)_\ellfr^{-\epsilon})=0$. Therefore the lemma follows from Lemma \ref{H1_generators_lemma}.
\end{proof}

\subsection{End of the proof}
\begin{proposition}\label{order_Sel_prop}
We have $\Selm_{\fp^M}(A/K) \cong \cOP/\fp^t$ for some $t \leq M$.
\end{proposition}
\begin{proof}
As in the proof of Proposition \ref{img_uppsi_prop}, a combination of Proposition \ref{prop:kolyvaginclasses_properties}\ref{item:loc-d-at-v-trivial} and Lemma \ref{global_duality_proposition1} implies that $\uppsi_{\ellfr_1}(\dbf(\ellfr_1)_{\ellfr_1}^{-\epsilon})=0$. Since $\cOP \cdot \dbf(\ellfr_1)_{\ellfr_1}^{-\epsilon}\cong \cOP/\fp^M$ (by Lemma \ref{lem:second-lemma-l1}(i)) and $\hone(K_{\ellfr_1},A[\P^M])\cong(\cOP/\P)^2$ (by Lemma \ref{H1_structure_lemma}), we obtain that $\img \uppsi_{\ellfr_1}\cong \cOP/\fp^t$ for some $t \leq M$. Using Lemma \ref{global_duality_proposition1} together with Proposition \ref{prop:kolyvaginclasses_properties}\ref{item:loc-d-at-v-trivial}, we get
$$\uppsi_{\ellfr}(\dbf(\ellfr \ellfr_1)_\ellfr^\epsilon) + \uppsi_{\ellfr_1}(\dbf(\ellfr \ellfr_1)_{\ellfr_1}^\epsilon) =0.$$
Therefore, from this and Proposition \ref{img_uppsi_prop} we have that $\img \uppsi_{\ellfr} \subseteq \img \uppsi_{\ellfr_1}$. Since by Proposition \ref{generating_Sel_proposition}
$\img \uppsi_{\ellfr}$ with $\ellfr$ ranging over $\mathscr{L}(U)$ generate $\Selm_{\fp^M}(A/K)^\ast$, we then have
$$\Selm_{\fp^M}(A/K)^\ast=\img \uppsi_{\ellfr_1} \cong \cOP/\fp^t.$$
It follows that $\Selm_{\fp^M}(A/K) \cong \cOP/\fp^t$.
\end{proof}

\begin{proof}[Proof of Theorem \ref{main_theorem} assuming \eqref{ass:Tamagawa}]
Since $y_K \notin \fp A(K)$, by Lemma \ref{lem:first-lemma-l1}(i) we know that $\cOP \cdot \delta_K(y_K^\epsilon\bmod \P^M)$ is a free $\cOP/\fp^M$-submodule of $\Selm_{\fp^M}(A/K)$ of rank one. Therefore, by Proposition \ref{order_Sel_prop}, since $y_K = y_K^\epsilon + y_K^{-\epsilon}$, we must have
\[
\Selm_{\fp^M}(A/K)=\delta_K(A(K)/\fp^M A(K))=(\cOE/\P^M) \cdot \delta_K(y_K \bmod \P^M) \cong \cOE/\fp^M.\qedhere
\]
\end{proof}

\subsection{Removing \eqref{ass:Tamagawa}}\label{sec:removing-tama} 

In this section we briefly explain how to modify the Kolyvagin classes in order to remove 
the assumption \eqref{ass:Tamagawa} from the proof of \ref{main_theorem}. First, note that \eqref{ass:Tamagawa} was only used in order to prove Proposition \ref{prop:kolyvaginclasses_properties}\ref{item:loc-d-at-v-trivial} and this in turn, together with Lemma \ref{global_duality_proposition1}, was used to prove Propositions \ref{img_uppsi_prop} and \ref{order_Sel_prop}.

To remove this assumption, we modify our Kolyvagin classes as follows: let $M \geq 1$ be an integer satisfying \eqref{ass:P-M-princ} and let $\pi_M$ be a generator of the principal ideal $\fP^M$. Let $M_0 \geq 1$ be an integer such that $\pi_M^{M_0}$ annihilates $\pi_0(\tilde{A}_v)[\fP^{\infty}]$. For any $\nfrak\in\admissible(M(1+M_0))$, we have our classes $\cbf(\nfrak) \in \hone(K, A[\fp^{M(1+M_0)}])$ and $\dbf(\nfrak) \in \hone(K, A)[\fP^{M(1+M_0)}]$. We define the modified cohomology classes as $\tilde{\cbf}(\nfrak):=\pi_M^{M_0}\cbf(\nfrak) \in \hone(K, A[\fP^M])$ and $\tilde{\dbf}(\nfrak):=\pi_M^{M_0}\dbf(\nfrak) \in \hone(K, A)[\fP^M]$. We now prove the following proposition without assuming \eqref{ass:Tamagawa}.

\begin{proposition}\label{prop:kolyvaginclasses_properties2}
    Let $\nfrak\ellfr\in\admissible(M(1+M_0))$ with $\ellfr\in\admissible_1(M(1+M_0))$ and let $v$ be a non-archimedean prime of $K$ that doesn't divide $\nfrak$. Then
    \begin{enumerate}[label=(\roman*)]
        \item $\tilde{\dbf}(\nfrak)_v=0$.\label{item:loc-d-at-v-trivial}
        \item if $v$ lies above $\ellfr$, there is a $\tau_c$-anti-equivariant $\cOP/\P^M$-linear isomorphism
        \begin{equation*}
            \tilde{\psi}_v\colon \hone(K_v, A)[\P^M]\xrightarrow{\sim} \honeur(K_v, A[\P^M])
        \end{equation*}
        such that $\tilde{\psi}_v(\tilde{\dbf}(\nfrak\ellfr)_v)=\tilde{\cbf}(\nfrak)_v$.\label{item:finite-singular2}
    \end{enumerate}
\end{proposition}
\begin{proof}
As explained in the proof of Proposition \ref{prop:kolyvaginclasses_properties}, we have that $\dbf(\nfrak)_v$ is inflated by some class in $\honeur(K_v, A)[\P^M] \cong \hone(k(v), \pi_0(\tilde{A}_v))[\P^M]$. Therefore, it follows that $\tilde{\dbf}(\nfrak)_v=0$. This proves (i). 

From Proposition \ref{prop:kolyvaginclasses_properties}(ii) if $v$ lies above $\ellfr$, there is a $\tau_c$-anti-equivariant $\cOP/\P^{M(1+M_0)}$-linear isomorphism
        \begin{equation*}
            \psi_v\colon \hone(K_v, A)[\P^{M(1+M_0)}]\xrightarrow{\sim} \honeur(K_v, A[\P^{(M(1+M_0)}])
        \end{equation*}
        such that $\psi_v(\dbf(\nfrak\ellfr)_v)=\cbf(\nfrak)_v$
Consider the sequence 
\[
0 \to A[\fP^M] \to A[\fP^{M(1+M_0)}] \xrightarrow{\pi_M} A[\fP^{MM_0}] \to 0
\]
Clearly $A[\fP^M]$ is the kernel of the middle map. By order considerations the last map must be surjective. Hence the sequence is exact. As $\ellfr \in \admissible_1(M(1+M_0))$, therefore $v$ splits completely in $K(A[\fP^{M(1+M_0)}])/K$ and so the above exact sequence remains the same when taking $\Gal(K_v^{\ur}/K_v)$-invariants. It follows that $\honeur(K_v, A[\fP^{M(1+M_0)}])[\fP^M]=\honeur(K_v, A[\fP^M])$. Now let $\tilde{\psi}_v$ be the isomorphism $\psi_v$ restricted to $\hone(K_v, A)[\fP^M]$. By the observation just made $\tilde{\psi}_v$ is an isomorphism onto $\honeur(K_v, A[\fP^M])$. Since $\psi_v$ is a $\cO_{\fP}$-module homomorphism, therefore we see that $\tilde{\psi}_v(\tilde{\dbf}(\nfrak\ellfr)_v)=\tilde{\cbf}(\nfrak)_v$ as desired. This proves (ii).
\end{proof}

\begin{proof}[Proof of Theorem \ref{main_theorem}]
Given the properties of our modified classes $\tilde{\cbf}(\nfrak)$ and $\tilde{\dbf}(\nfrak)$ established in Proposition \ref{prop:kolyvaginclasses_properties2}, we proceed as we did above, but replacing them to $\cbf(\nfrak)$ and $\dbf(\nfrak)$. However we need to make two changes. First we change our definition of $L$ to be $L:=K(A[\fP^{M(1+M_0)}])$. Secondly we define 
\[
\mathscr{L}(U)=\set{\ellfr \in \admissible_1(M(1+M_0)) \; | \; \Frob_{\ellfr}(L_S/F)=[\tau_c u] \; \text{for some} \; u \in U },
\]. 
With these changes our method proves Theorem \ref{main_theorem} without the need to assume \eqref{ass:Tamagawa}.
\end{proof}
\section{Residual representation}\label{representation_section}

In this section, we show that in some cases the assumption \eqref{ass:abs-G-K-irr} can be weakened in Theorem \ref{main_theorem}. More precisely, we will prove the following proposition.
\begin{proposition}
    \label{theorem:assumption-reduction}
    Consider an abelian variety $A$ over a totally real field $F$, satisfying \eqref{ass:GL-2} endowed with a polarisation $\gamma \colon A \to A^\vee$ satisfying \eqref{ass:O-E-lin} and \eqref{ass:deg-lambda}. Fix moreover primes $\p$ and $\P$, respectively of $F$ and $E$, both above $p$. If the following assumptions hold:
    \leqnomode
    \begin{gather}
        \text{$A$ has good $\P$-ordinary reduction at $\p$};\label{ass:P-ord}\tag{$\P$-$\mathrm{ord}_\p$}\\ 
        \text{$F_\p(\mu_p)/F_\p$ is a totally ramified extension of degree $p-1$};\label{ass:Delta-p}\tag{$\Delta_\p$}\\
        p>3 \label{ass:p-gt-3}\tag{$p>3$} 
    \end{gather}
    then the assumption
    \begin{equation}\label{ass:G_F-irr}\tag{$G_F$-irr}
        \text{$A[\P]$ is a simple $(\cOE/\P)[G_F]$-module}
    \end{equation}
    \reqnomode
    implies \eqref{ass:abs-G-K-irr}.
\end{proposition}

Therefore, for the rest of this section, fix an abelian variety $A$ satisfying \eqref{ass:GL-2}, fix a polarization $\gamma$ of $A$ and assume that \eqref{ass:O-E-lin} and \eqref{ass:deg-lambda} hold. Fix moreover a prime $\P$ of $E$ above $p$ and 
denote by $\bar{\rho}_\P := \rho_{\P, 1} \colon G_F \to \GL_2(A[\P]) \cong \GL_2(\cOE/\P)$ the group homomorphism attached to the $(\cOE/\P)[G_F]$-module $A[\P]$.

Let moreover $\iota_\P \colon \Z_p \inj \cOP$ and $\pi_\P \colon \cOP \surj \cOE/\P$ be the natural maps.
\begin{lemma}\label{Delta_condition_lemma}
Let $\p$ be a prime of $F$ over $p$ and let $I_\p$ be its inertia subgroup. If condition \eqref{ass:Delta-p} is met, then $\det\bar{\rho}_{\fp}(I_\p)=\pi_{\fp}(\iota_{\fp}(\mathbb{Z}_p^{\times}))$
\end{lemma}
\begin{proof}
Condition $(\Delta_\p)$ implies that $\Qp(\mu_p)$ and $F_\p$ are linearly disjoint over $F_\p$. As $\Qp(\mu_{p^n})/\Qp$ is a cyclic extension for any $n \geq 1$, therefore it follows that we also have $\Qp(\mu_{p^{\infty}})$ and $F_\p$ are linearly disjoint over $F_\p$. Thus we have $\Gal(F_\p(\mu_{p^{\infty}})/F_\p) \cong \Gal(\Qp(\mu_{p^{\infty}})/\Qp)\cong \mathbb{Z}_p^{\times}$. The fact that $F_\p(\mu_p)/F_\p$ is totally ramified and the fact that for any $n \geq 1$ we have that $\Gal(F_\p(\mu_{p^n})/F_\p) \cong \Gal(\Qp(\mu_{p^n})/\Qp)$ is cyclic implies that $F_\p(\mu_{p^{\infty}})/F_\p$ is totally ramified. The statement of the lemma follows from this and Lemma \ref{lemma:det}.
\end{proof}

Recall the following definition.
\begin{definition}
    Let $q= p^n$, $p>2$ and consider $\GL_2(\F_q)$ and fix $\epsilon \in \F_q \smallsetminus \F_q^2$. 
    \begin{itemize}
        \item A split Cartan subgroup is a subgroup conjugated to $
        C_{\mathrm{s}}= \left\{ \bigl(\begin{smallmatrix}
        a &0\\0 & b    
        \end{smallmatrix}\bigr) : a, b \in \F_q^\times \right\}$.
    \item A non-split Cartan subgroup of $\GL_2(\F_q)$ is a subgroup conjugated to 
        \[
        C_{\mathrm{ns}} = \left\{
        \begin{pmatrix}
            a &b\epsilon\\b&a
        \end{pmatrix}:(a,b) \in (\F_q \times \F_q)\smallsetminus (0,0) \right\}.
        \]
    \end{itemize}
\end{definition}
Note that $C_{\mathrm{ns}}$ is cyclic of order $q^2-1$. In fact $k_2 := \F_q[\epsilon]$ is a quadratic extension of $\F_q$ and we can easily verify that
\[
C_{\mathrm{ns}} \iso k^\times_2; \qquad
\begin{pmatrix}
            a &b\epsilon\\b&a
\end{pmatrix}
\mapsto a + b \sqrt{\epsilon}        
\]
Moreover $C_{\mathrm{ns}}$ has index $2$ in its normalizers.

\begin{lemma}\label{nonsplitCartan_image_lemma}
Let $V$ a $2$-dimensional vector space over a finite field $k$ and $\rho: G \to \GL(V)$ a $k$-linear representation that is irreducible, but not absolutely irreducible. Therefore, it follows that $\rho(G) \subseteq \GL(V) \cong \GL_2(k)$ is contained in a nonsplit Cartan subgroup.
\end{lemma}
\begin{proof}
Let $H=\rho(G)$ and $D=\End_{k[H]}(V)$. Since $\rho$ is irreducible, therefore by Schur's lemma $D$ is a division ring and so is a field since $D$ is finite (as $V$ is finite dimensional over $k$). Note that $V$ is a vector space over $D$ and  $\dim_D(V)\dim_k(D)=\dim_k(V)=2$. Since $\rho$ is not absolutely irreducible, therefore by \cite[Theorem 29.13]{CR} $\dim_k(D) \neq 1$. Thus $\dim_k(D) =2$ and therefore we can identify $D^{\times} \subseteq \GL_2(k)$ with a nonsplit Cartan subgroup. Since every element of $H$ commutes with the nonsplit Cartan subgroup $D^\times$, therefore we see that $H \subseteq D^\times$.
\end{proof}

\begin{lemma}\label{nonsplitCartanSL2_lemma}
Let $k$ be a finite field of order $p^n$. Let $C_{ns} \subseteq \GL_2(k)$ be a nonsplit Cartan subgroup. Then $C_{ns} \cap \SL_2(k)$ is a cyclic group of order $p^n+1$
\end{lemma}
\begin{proof}
Let $k_2$ be the finite field of degree 2 over $k$. Then we have an isomorphism $C_{\mathrm{ns}} \cong k_2^{\times}$. Under this isomorphism $C_{ns} \cap \SL_2(k)$ corresponds to $H:=\ker(N_{k_2/k} \colon k^{\times}_2 \to k^{\times})$, where $N_{k_2/k}$ is the norm. Moreover, $N_{k_2/k}$ is surjective: let $G=\Gal(k_2/k)$; by Hilbert Theorem 90 we have that $\hone(G, k^{\times}_2)=1$ and as $k_2^{\times}$ is finite we have by \cite[VII Prop. 8]{serre:local-fields} $\hone(G, k_2^{\times})$ has the same order as $H^2(G, k_2^{\times})$. Thus, $H^2(G, k^{\times}_2)$ is trivial, which is equivalent to the claimed surjectivity.
Therefore, $\#H=\# k^\times_2/\#k^\times = \frac{p^{2n}-1}{p^n-1}=p^n+1$.
\end{proof}

\begin{proposition}
Let $L=F(A[\fp])$ and $S$ be the set of primes of $F$ where $A$ has bad reduction. Then $\bar{\rho}_{\fp}(G_F) \neq \bar{\rho}_{\fp}(G_K)$ if and only if $L \cap K=K$ and these equivalent conditions imply that the primes of $F$ that ramify in $K$ are contained in the set consisting of primes above $p$ and primes in $S$.
\end{proposition}
\begin{proof}
We have $\bar{\rho}_{\fp}(G_F)\cong \Gal(L/F)$ and $\bar{\rho}_{\fp}(G_K)\cong \Gal(KL/K)\cong \Gal(L/L \cap K)$. This proves that $\bar{\rho}_{\fp}(G_F) \neq \bar{\rho}_{\fp}(G_K)$ if and only if $L \cap K=K$. Now assume that these equivalent conditions are true. By a similar proof along the same lines of \cite[VII Proposition 4.1]{silverman:arithmetic-elliptic-curves} we have that $L/F$ is unramified outside the set consisting of primes of $F$ above $p$ and primes where $A$ has bad reduction. If $L \cap K=K$, then in particular $K \subseteq L$ and so any prime of $F$ that ramifies in $K$ must also ramify in $L$ and thus must either be a prime above $p$ or a prime in $S$.
\end{proof}

\begin{proposition}\label{prop:implication-irreducibility}
Assume that the conditions \eqref{ass:P-ord}, for some prime $\p$ of $F$ above $p$, and \eqref{ass:Delta-p} are satisfied. Furthermore, if $p=3$ assume that $f(\fp|3)=1$ and that $\Q(\mu_3)\subseteq K$. Then if $\bar{\rho}_{\fp}$ is irreducible, so is $\bar{\rho}_{\fp}|_{G_K}$
\end{proposition}
\begin{proof}
Suppose that $\bar{\rho}_{\fp}$ is irreducible but that $\bar{\rho}_{\fp}|_{G_K}$ is reducible. Then $\bar{\rho}_{\fp}|_{G_K}$ is semisimple (see \cite[Theorem 7.1.1]{Craven}) and its image is contained in a split Cartan subgroup $C_s$ of $\GL_2(\cO_{\fp}/\fp)$. Then $\bar{\rho}_{\fp}|_{G_K}=\alpha \oplus \beta$ for some characters $\alpha, \beta: G_K \to (\cO_{\fp}/\fp)^{\times}$. If $W$ is the one-dimensional subspace of $A[\fp]$ stable under $G_K$ giving the character $\alpha$, then as the proof of \emph{loc.~cit.}~one shows that $\rho(\tau_c) W$ is the subspace of $A[\fp]$ giving the character $\beta$. Then for any $g \in G_K$ and any $w \in W$ we have $\bar{\rho}_{\fp}(g)\bar{\rho}_{\fp}(\tau_c)w=\bar{\rho}_{\fp}(\tau_c) \bar{\rho}_{\fp}(\tau_c^{-1})\bar{\rho}_{\fp}(g)\bar{\rho}_{\fp}(\tau_c)w$. Therefore, $\beta=\alpha^{\tau_c}$ where we define $\alpha^{\tau_c}(g)=\alpha(\tau_c^{-1}g\tau_c)$. Now let $\wp$ be a prime of $K$ above $\p$ and let $I_{\wp} \subseteq G_{K_{\wp}}$ be the inertia subgroup. Then it follows by Proposition \ref{prop:P-ordinary-shape} that we have $\{\alpha|_{I_{\wp}}, \alpha^{\tau_c}|_{I_{\wp}}\}=\{\chi_{\fp}|_{I_{\wp}},1\}$. Therefore $\chi_{\fp}|_{I_{\wp}}=1$. This implies that $\chi^2_{\fp}(I_\p)=1$. By Lemma \ref{Delta_condition_lemma} we see that this implies that $p=3$. In this case, we have $\cO_{\fp}/\fp=\mathbb{F}_3$ and ${\chi_3}_{|G_K}=1$ by assumption. It follows that $\bar{\rho}_{\fp}(G_K) \subseteq C_s \cap \SL_2(\F_3)=\{\pm I\}$. 

By Lemma \ref{eigenspace_structure_lemma}, we have $\bar{\rho}_{\fp}(\tau_c) \sim \begin{pmatrix} 1 & 0 \\ 0 & -1 \end{pmatrix}$. Thus we see that $\bar{\rho}_{\fp}(G_F) \cong (\Z/2\Z)^a$ for some $a \leq 2$. This contradicts the irreducibility of $\bar{\rho}_{\fp}$.
\end{proof}

\begin{proposition}\label{prop:implication-absolute-irreducibility}
Assume \eqref{ass:P-ord}, for some prime $\p$ of $F$ above $p$. Furthermore, if $p=3$ assume that $\Q(\mu_3)\subseteq K$. If $\bar{\rho}_{\fp}|_{G_K}$ is irreducible, but not absolutely irreducible, then $p=3$ and $\bar{\rho}_{\fp}(G_K)$ is a cyclic group whose order is greater than 2 and divides $N(\fp)+1$, where $N(\fp)$ is the norm. If $\#\bar{\rho}_{\fp}(G_K)$ is equal to 4 or is a power of an odd prime, then $\bar{\rho}_{\fp}(G_F)$ is a dihedral group of order $2\#\bar{\rho}_{\fp}(G_K)$
\end{proposition}
\begin{proof}
According to Proposition \ref{prop:P-ordinary-shape}, for a suitable choice of basis, $\bar{\rho}_{\fp |I_\p} = 
\begin{pmatrix} \chi_{\fp} & \ast \\ 0 & 1 \end{pmatrix}$. Moreover, by Lemma \ref{nonsplitCartan_image_lemma}, $\bar{\rho}_{\fp}(G_K)$ is contained in a nonsplit Cartan subgroup $C_{ns}$. Therefore, with the same notation of the above proof, $\bar{\rho}_{\fp}(I_{\wp}) \subseteq C_{ns}$ is reducible and therefore is made of scalar matrices. 
In particular, we have again $\chi_{\fp}|_{I_{\wp}}=1$, that implies $\chi^2_{\fp}(I_\p)=1$ and thus $p=3$.

Then, by the assumptions, $\bar{\rho}_{\fp}(G_K) \subseteq C_{ns} \cap \SL_2(\cO_{\fp}/\fp)$. By Lemma \ref{nonsplitCartanSL2_lemma} this implies that $\bar{\rho}_{\fp}(G_K)$ is a cyclic group whose order divides $N(\fp)+1$. Furthermore since $\bar{\rho}_{\fp}|_{G_K}$ is irreducible we must have that $\#\bar{\rho}_{\fp}(G_K)>2$. 

Now suppose that $\#\bar{\rho}_{\fp}(G_K)$ is equal to 4 or is a power of an odd prime. Let $c$ be a generator of the cyclic group $\bar{\rho}_{\fp}(G_K)$. To simplify notation we denote $\cO_{\fP}/\fP$ by $k$ and let $k_2$ be the field of degree 2 over $k$. Then we can identify $C_{ns}$ with $k_2^{\times}$. As the order of $c$ divides $N(\fP)+1$ and is greater than 2, it cannot divide $N(\fP)-1$. Therefore $c \notin k$. It follows that $k_2=k[c]$. 

We have $\bar{\rho}_{\fp}(\tau_c) \sim \begin{pmatrix} 1 & 0 \\ 0 & -1 \end{pmatrix}$ and so $\bar{\rho}_{\fp}(\tau_c) \notin C_{ns}$. Assume that $\bar{\rho}_{\fp}(\tau_c)$ commutes with $c$. Then it will commutes with $k_2=k[c]$ and hence it will commute with $C_{ns}$. This will imply that $\bar{\rho}_{\fp}(\tau_c) \in C_{ns}$ which contradicts the fact above. Thus $\bar{\rho}_{\fp}(G_F)$ is not abelian. $\bar{\rho}_{\fp}(G_F)$ is a semidirect product of the (normal) cyclic group $\bar{\rho}_{\fp}(G_K)$ and $\langle \bar{\rho}_{\fp}(\tau_c)\rangle$. Since $\#\bar{\rho}_{\fp}(G_K)$ is equal to 4 or a power of an odd prime, therefore $\Aut(\bar{\rho}_{\fp}(G_K))$ is cyclic and hence has a unique subgroup of order 2. This together with the nonabelian semidirect product description implies that $\bar{\rho}_{\fp}(G_F)$ is a dihederal group.
\end{proof}

\section{Anticyclotomic Iwasawa theory}\label{Iwasawatheory_section}

In this section we apply the vanishing theorem in the previous section and the abstract Iwasawa theoretical results of \cite[Section 3]{matar-nekovar:kolyvagin} in order to generalize \cite[Theorem (6.9)]{matar-nekovar:kolyvagin} from the case of elliptic curves over $\Q$ to the case of modular abelian varieties defined over totally real fields.

Suppose in this section that $A/F$ abelian variety defined over a totally real number field satisfying \eqref{ass:modular-av}, a polarisation $\gamma \colon A \to A^\vee$ and a prime $p$ satisfying \eqref{ass:p-ndiv}. Fix moreover a prime $\P$ of $E = \End_F^0(A)$ and a prime $\p \in \allowedprimes$ of $F$, both above $p$. Let $K/F$ be a CM extension satisfying \eqref{ass:heegner-hypothesis} and such that any prime $\q$ of $F$ above $p$ is unramified in $K/F$.

Finally suppose that \eqref{ass:Tamagawa} and the following assumptions are in force:
\leqnomode
\begin{gather}
    \text{$A$ has good ordinary reduction at any $\q \mid p$ of $F$};\label{ass:ordinary}\tag{$\mathrm{ord}_p$}\\
    p \nmid [K(x):K];\label{ass:deg-Kx}\tag{$\deg_{K(x)}$}\\
    \text{
    $a_\p \not\equiv 1, \epsilon(\q) \bmod{\P'}$, for any prime $\q \mid p$ of $F$ and $\P' \mid p$ of $E$}\label{ass:A2-II}\tag{$a_\p$-not-cong}
\end{gather}
\reqnomode

Denote by $f_\p$ the inertia degree of $\p$ in $F/\Q$. We now construct a subfield $K_\infty$ of 
\[
K(x(\p^\infty)) = \injlim_n K(x(\p^n)),
\]
that we call the anticyclotomic $\Z_p^{f_\p}$-extension of $K$ relative to $x$, as $K_\infty$ is a $\Z_p^{f_\p}$-extension of $K$ pro-dihedral over $F$. Denote $\Delta = \Gal\bigl(K(x(\p))/K\bigr)$. Since $p \nmid \#\Delta = u(0)^{-1}(N\p - \varepsilon(\p))[K(x):K]$ by \eqref{ass:deg-Kx}, it follows, it follows from Corollary \ref{cor:degree-x-n-ells} that for any $n \ge 0$ we have 
\[
\Gal\biggl(\frac{K(x(\p^{n+1}))}{K}\biggr)= \Gal\biggl(\frac{K(x(\p^{n+1}))}{K(x(\p))}\biggr) \times \Gal\biggl(\frac{K(x(\p))}{K}\biggr) \cong (\Z/p^n\Z)^{f_\p} \times \Delta.
\]
We define $K_n := K(x(\p^{n+1}))^\Delta$ (which is the maximal $p$-subextension of $K(\p^{n+1})/K$ as $p \nmid \#\Delta$), so that $\Gamma_n := \Gal(K_n/K) \cong (\Z/p^n\Z)^{f_\p}$ and therefore $K_\infty := \injlim_n K_n = K(x(\p^\infty))^{\Delta}$ is a $\Z_p^{f_\p}$-extension of $K$. It is by construction pro-dihedral over $F$, i.e., 
\[
\Gal(K_\infty/F) = \Gal(K_\infty/K) \rtimes \Gal(K/F) \cong \Gal(K_\infty/K) \rtimes \set{1, \tau_c},
\]
with $\tau_c \gamma \tau_c = \gamma^{-1}$ for any $\gamma \in \Gal(K_\infty/K)$; see \cite[(2.6.5)]{nekovar:euler-system-method}.

\begin{remark}
    Analogously one might construct, under the assumption that $p \nmid [K[\cond]:K]$ for an ideal $\cond$ of $F$, another $\Z_p^{f_\p}$-extension $K_{\cond, \infty}$ of $K$, as subfield of $K[\cond\p^\infty]$.
     See e.g., \cite[Definition-Proposition 1.2.3(2)]{aflalo-nekovar:nontriviality-cm-points} and \cite[Section 2.3]{wang:anti-imc-hmf} and \cite[Section 3]{howard:iwasawa-gl2-abelian-varieties}. Or one can use a different definition of ring class field as in \cite[Definition 2.1]{longo:anticyclotomic-imc-hilbert} (under the hypothesis that $p \nmid h_K$). Under certain $p$-indivisibilities all these constructions coincide. For instance 
if $p \nmid [K[\cond(x)]:K]$ it follows that $K_n$ is the maximal $p$-subextension both of $K(x(\p^{n+1}))/K$ and $K[\cond(x)\p^{n+1}]/K$; therefore $K_\infty = K_{\cond(x), \infty}$.
\end{remark}

Let $\Lambda = \cOP \llbracket \Gal(K_\infty/K)\rrbracket \cong \cOP \llbracket T_1, \dots T_{d_\p}\rrbracket$ be the Iwasawa algebra of $K_\infty/K$. Moreover, for any $n \ge 0$, let $y_n = \Tr_{K(x(\p^{n+1}))/K_n} y(\p^{n+1}) \in A(K_n)$. For any algebraic extension $L/K$ we denote
\[
N_{L/K}(A \otimes \cOE) := \projlim_{\substack{K \subseteq L' \subseteq L\\ \text{finite}, \Tr}} A(L') \otimes_{\cOE} \cOEP \subseteq S_\P(A/L)
\]
the group of $L$-norms. Moreover, when $L=K_\infty$, we will use the terminology \emph{universal norms}.

From \cite{matar-nekovar:kolyvagin} we deduce the following theorem.
\begin{theorem}\label{theorem:iwasawa-th-main}
    Under our running hypothesis, assume \eqref{ass:abs-G-K-irr} and that $y_K \notin \P A(K)$. Then for any intermediate field $K \subseteq L \subseteq K_\infty$
    \[
    \sha(A/L)[\P^\infty] = 0, \quad A(L)\otimes_{\cOE} E_\P/\cOP = \Selm_{\P^\infty}(A/L), \quad N_{L/K}(A \otimes_{\cOE}\cOP) = S_\P(A/L)
    \]
    and both $S_\P(A/L)$ and the Pontryagin dual of $\Selm_{\P^\infty}(A/L)$ are free modules of rank one over $\Lambda_L = \cOP \llbracket \Gal(L/K)\rrbracket$. If moreover $[L:K]< \infty$, then $A(L) \otimes_{\cOE}\cOP = S_\P(A/L)$ is a free $\cOP$-module of rank $\rk_{\cOE} B(L) = [L:K]$. 
    
    Furthermore, for any intermediate fields $K \subseteq L \subseteq L' \subseteq K_\infty$, the natural map
    \[
    N_{L'/K}(A \otimes_{\cOE}\cOP) = S_\P(A/L') \to N_{L/K}(A \otimes_{\cOE}\cOP) = S_\P(A/L)    
    \]
    is surjective and $S_\P(A/L)$ is generated over $\Lambda_L$ by the image $x_L$ of any element 
    \[
    x \in N_{K_\infty/K}(A \otimes_{\cOE} \cOP),
    \] 
    whose image in $\bar{x}_K \in A(K) \otimes \cOE/\P$ is non-zero.   
\end{theorem}

\begin{proof}
    This is \cite[Theorem 3.5]{matar-nekovar:kolyvagin} applied to our concrete setting. One needs only to check that the assumptions (A1)-(A5) and (A7) of \emph{loc.~cit.}~are satisfied: $(A1)_{A, K, \p}$ follows by since the Weil pairing $(\,,\,)^\gamma_{\P}$ is perfect, skew-symmetric and Galois equivariant (see Section \ref{sec:weil-pairing}); for the first part of $(A2)_{A, K, \p}$, observe that our assumption \ref{ass:Tamagawa} implies $\mathrm{Tam}(A/K, \P)$ by (1.8.2) of \emph{loc.~cit.};
   its second part is less immediate and we postpone it; $(A3)_{A, K, \P}$ is Lemma \ref{torsion_lemma} for $n=1$; $(A4)_{A, K, \P}$ follows from Theorem \ref{main_theorem} and the assumption that $y_K \notin \P A(K)$; $(A5)_{A, K}$ is satisfied for $K^+=F$ as $K_\infty/F$ is pro-dihedral.
     
     Let's show that also $(A7)$ is satisfied.  As \eqref{ass:ordinary} implies that $a_\p \in \cOEP^\times$ by Proposition \ref{proposition:ordinary-equivalent-2}, therefore the polynomial $X^2 - a_\p X - N \p$ has two roots $\alpha_\p \in \cOEP^\times$ and $\beta_\p \in \pi^{r_\P f_\p}\cOEP$.
    Define $z_n = \alpha_\p^{-n} y(\p^{n+1}) - u(n)\alpha_\p^{-n-1} y(\p^n) \in A(K(x(\p^{n+1}))) \otimes_{\cOE} \cOEP$. Using the norm relations of Proposition \ref{prop:norm-relations-heegner} we can easily check that the sequence $(z_n)_{n\ge 0}$ form a system of norm compatible elements with respect to the traces $\Tr_{K(x(\p^{n+1}))/K(x(\p^{n}))}$.
    Therefore, denoting $z_n' = \Tr_{K(x(\p^{n+1}))/K_n} z_n$, we have that $(z_n')_{n\ge 0} \in \injlim_n A(K_n) \otimes_{\cOE} \cOEP = N_{K_\infty/K} (A \otimes \cOE)$. 
    
    Moreover, by the norm relations of Proposition \ref{prop:norm-relations-heegner}, it easily follows that
    \[
    \Tr_{K(x(\p))/K}(y(\p))=(a_\p-1-\varepsilon(\p))y_K.
    \]
    Therefore,
    \begin{align*}
    z_0' &= \Tr_{\frac{K(x(\p))}{K}}z_0 = 
    (a_\p - 1 - \varepsilon(\p))y_K - u(0)\alpha_p^{-1}[K(x(\p)):K(x)]y_K =\\  
    &=(a_\p - 1 - \varepsilon(\p))y_K - u(0)\alpha_\p^{-1}\frac{N\p-\varepsilon(\p)}{u(0)}y_K =\\
    &=(\alpha_\p + \beta_\p - 1- \varepsilon(\p)-\beta_\p + \alpha_\p^{-1}\varepsilon(\p))y_K =\\
    &=\alpha_\p^{-1}(\alpha_\p - 1)(\alpha_\p-\varepsilon(\p))y_K \ne 0,
    \end{align*}
    as follows by our assumption \eqref{ass:A2-II}, since $y_K \notin \P A(K)$ (which implies in particular that $y_K \notin A(K)_{\tors} = A(K)_{\cOE-\tors}$ ).

    Let's now turn to the second part of $(A2)_{A, K, \P}$. Let $\q$ be a prime of $F$ above $p$ and $v$ a prime of $K$ above $\q$ and denote respectively by $k_\q$ and $k_v$ the residue fields of $F$ and $K$ at $\q$ and $v$: we have to show that $\tilde{A}_v(k_v)[\P]=0$. Denote by $\tilde{E}$ the Galois closure of $E$ and identify it with its image in $\C$ via any embedding. Choose moreover for any $\sigma \colon E \inj \C$ an extension to an automorphism $\tilde{\sigma}$ of $E$ and a prime $\tilde{\P} \mid \P$ of $\tilde{E}$. Thus $\sigma (f_\p(X)) \in \tilde{E}[X]$, in the notations of Section \ref{sec:ordinary}. Moreover, it follows from \eqref{ass:ordinary} and Lemma \ref{lemma:ordinary-equivalent-1} that $a_\p \in \cOP^\times \subseteq \cO_{\tilde{E}, \tilde{\P}}^\times$ and thus $\sigma(a_\p) \in \cO_{\tilde{E}} \smallsetminus \tilde{\sigma}(\tilde{\P})$, whence (denoting by $\alpha_{\sigma, \p}$ and $\beta_{\sigma, \p}$ the two roots of $\sigma(f_\p(X))$ in $\cO_{\tilde{E}, \tilde{\sigma}(\tilde{\P})}$) $v_{\tilde{\sigma}(\tilde{\P})}(\alpha_{\sigma, \p})=0$ and $v_{\tilde{\sigma}(\tilde{\P})}(\beta_{\sigma, \p})>0$.

Fix an algebraic closure $\Qbar_p$ of $\Qp$ and an embedding of $\tilde{E} \inj \Qbar_p$. For any prime $\tilde{\P}' \mid p$ of $\tilde{E}$ this extends uniquely to a continuous embedding $\tilde{E}_{\tilde{\P}'} \inj \Qbar_p$, i.e., it respects the valuations. Thus (denoting by $\bar{\mathfrak{m}}$ the maximal ideal of $\Qbar_p$), $\alpha_{\sigma, \p} \equiv \sigma(a_\p)$ and $\beta_{\sigma, \p} \equiv 0 \pmod {\bar{\mathfrak{m}}}$.

 Note that, since $\q$ is unramified in $K/F$, then $\tilde{A}_v \cong \tilde{A}_{\q} \otimes_{k_\q} k_v$, where $\tilde{A}_{\q}$ is the special fibre of the Néron model of $A$ at $\q$. Let $n_v= [k_v : k_{\q}]$. Then it follows by \cite[Theorem 16.7(i)]{edixhoven-moonen-van-der-geer:av}, \eqref{eq:relation-char-poly} and the previous discussion
\begin{equation*}
    \# \tilde{A}_v(k_v) =\hspace{-7pt} \prod_{\sigma \in \Gal(E/\Q)} \hspace{-12pt}(1-\alpha_{\sigma,\p}^{n_v})(1-\beta_{\sigma,\p}^{n_v}) \equiv \hspace{-7pt}
    \prod_{\sigma \in \Gal(E/\Q)}\hspace{-12pt} (1-\sigma(a_\p)^{n_v}) = N_{E/\Q}(1- a_\p^{n_v}) \mod {\bar{\mathfrak{m}}}.
\end{equation*}
Actually, since both sides are integers, the previous congruence can be understood $\bmod {p\Z}$. Thus $(A2)_{A, K,\P }$ follows if we assume that $p \nmid N_{E/\Q}(a_\p -1)N_{E/\Q}(a_\p -\varepsilon(\q))$. The latter is visibly equivalent to \eqref{ass:A2-II}.
\end{proof}

\begin{remark}\label{rk:iwasawa-layers}
    The last part of the theorem can be improved in the case $L = K_n$. Indeed, during the proof we constructed a universal norm $(z_n')$, where any $z_n'$ a $\cOP$-linear combination of the traces of the Heegner points $\{y(\p^m)\}_{m\ge0}$ and such that $z_0'$ has nontrivial projection $\bmod \,\P$. 
    Therefore, since $K_n/K$ is a finite extension, it follows directly from the statement of the Theorem that $A(K_n) \otimes_{\cOE} \cOP$ (which is a free $\cOP$-module of rank $p^{nf_\p}$) is generated over $\cOP[\Gal(K_n/K)]$ by the traces of the Heegner points of $\p$-power conductor.
\end{remark}

\begin{remark}
    Since \eqref{ass:ordinary} implies \eqref{ass:P-ord}, it follows from Proposition \ref{theorem:assumption-reduction} that the assumption \eqref{ass:abs-G-K-irr} in Theorem \ref{theorem:iwasawa-th-main} can be replaced by \eqref{ass:G_F-irr}, \eqref{ass:p-gt-3} and \eqref{ass:Delta-p}.
\end{remark}
Note that the combination of Theorem \ref{theorem:iwasawa-th-main} and the previous remarks proves Theorem \ref{theorem:iwasawa-intro}.

%\bibliographystyle{amsalpha}
%\bibliography{Bibliography}
\printbibliography

\end{document}